\documentclass[aap,preprint]{imsart}

\usepackage[OT1]{fontenc}

\usepackage[english,italian]{babel}

\usepackage{amsfonts,amsmath,amscd,amssymb,amsthm}
\usepackage{enumerate}

\usepackage[numbers]{natbib}

\usepackage[colorlinks,citecolor=blue,urlcolor=blue]{hyperref}

\arxiv{1612.03816}

\startlocaldefs

\numberwithin{equation}{section}

\theoremstyle{plain}
\newtheorem{thrm}{Theorem}[section]
\newtheorem{prop}{Proposition}[section]
\newtheorem{lemma}{Lemma}[section]
\newtheorem{crll}[thrm]{Corollary}

\theoremstyle{definition}
\newtheorem{defn}{Definition}[section]

\newtheorem{ass}{Assumptions}[section]

\theoremstyle{remark}

\newtheorem{rem}{Remark}[section]

\DeclareMathOperator{\cl}{cl}
\DeclareMathOperator{\Id}{Id}

\DeclareMathOperator{\supp}{supp}
\DeclareMathOperator{\Prb}{\mathbf{P}}
\DeclareMathOperator{\Mean}{\mathbf{E}}

\DeclareMathOperator{\Law}{Law}

\endlocaldefs

\begin{document}


\newcommand{\trans}[1]{{#1}^\mathsf{T}}
\newcommand{\prbms}[2][]{\mathcal{P}_{#1}(#2)}
\newcommand{\Borel}[1]{\mathcal{B}(#1)}


\newcounter{hypcount}
\newenvironment{hypenv}{\renewcommand{\labelenumi}{(A\arabic{enumi})}\begin{enumerate}\setcounter{enumi}{\value{hypcount}}}{\setcounter{hypcount}{\value{enumi}}\end{enumerate}}

\newcounter{hypcount2}
\newenvironment{Hypenv}{\renewcommand{\labelenumi}{\textsc{(H\arabic{enumi})}}\begin{enumerate}\setcounter{enumi}{\value{hypcount2}}}{\setcounter{hypcount2}{\value{enumi}}\end{enumerate}}

\newenvironment{enumrm}{\begin{enumerate}\renewcommand{\labelenumi}{\textup{(\roman{enumi})}}}{\end{enumerate}}


\newcommand{\hypref}[1]{(ND)}
\newcommand{\Hypref}[1]{(H\ref{#1})}
\newcommand{\Hyprefall}{(H1)\,--\,(H\arabic{hypcount2})}

\renewcommand{\phi}{\varphi}
\renewcommand{\epsilon}{\varepsilon}


\selectlanguage{english}

\begin{frontmatter}

\title{$N$-player games and mean field games with absorption}
\runtitle{Mean field games with absorption}

\begin{aug}
  \author{\fnms{Luciano} \snm{Campi}\thanksref{t1}\ead[label=e1]{L.Campi@lse.ac.uk}}
  \and
  \author{\fnms{Markus} \snm{Fischer}\thanksref{t2}\ead[label=e2]{fischer@math.unipd.it}}

  \thankstext{t1}{Corresponding author.}

  \thankstext{t2}{M.F.\ was partially supported by the research projects ``Mean Field Games and Nonlinear \mbox{PDEs}'' (\mbox{CPDA157835}) of the University of Padua and ``Nonlinear Partial Differential Equations: Asymptotic Problems and Mean-Field Games'' of the Fondazione \mbox{CaRiPaRo}.}

  \runauthor{L.~Campi and M.~Fischer}

  \affiliation{London School of Economics and Political Science and University of Padua}

  \address{Department of Statistics\\
  London School of Economics and Political Science\\
  Houghton Street, London WC2A 2AE\\
  United Kingdom\\
  \printead{e1}}

  \address{Department of Mathematics\\
  University of Padua\\
  via Trieste 63, 35121 Padova\\
  Italy\\
  \printead{e2}}

\end{aug}

\begin{abstract}
We introduce a simple class of mean field games with absorbing boundary over a finite time horizon. In the corresponding $N$-player games, the evolution of players' states is described by a system of weakly interacting It{\^o} equations with absorption on first exit from a bounded open set. Once a player exits, her/his contribution is removed from the empirical measure of the system. Players thus interact through a renormalized empirical measure. In the definition of solution to the mean field game, the renormalization appears in form of a conditional law. We justify our definition of solution in the usual way, that is, by showing that a solution of the mean field game induces approximate Nash equilibria for the $N$-player games with approximation error tending to zero as $N$ tends to infinity. This convergence is established provided the diffusion coefficient is non-degenerate. The degenerate case is more delicate and gives rise to counter-examples.
\end{abstract}



\end{frontmatter}


\section{Introduction} \label{SectIntro}

Mean field games (MFGs, henceforth) were introduced by \citet{lasrylions06a,lasrylions06b,lasrylions07} and, independently, by \citet{huangetal06}, as limit models for symmetric nonzero-sum non-cooperative $N$-player games with interaction of mean field type as the number of players tends to infinity. The limit
relation is commonly understood in the sense that a solution of the MFG allows to construct approximate Nash equilibria for the corresponding $N$-player games if $N$ is sufficiently large; see, for instance, \citet{huangetal06}, \citet{kolokoltsovetal11}, \citet{carmonadelarue13}, and \citet{carmonalacker15}. This approximation result is useful from a practical point of view since the model of interest is commonly the $N$-player game with $N$ very large so that a direct computation of Nash equilibria is not feasible.

The purpose of this paper is to study $N$-player games and related \mbox{MFGs} in the presence of an absorbing set. Thus, a player is eliminated from the $N$-player game as soon as her/his private state process leaves a given open set $O \subset \mathbb R^d$, the set of non-absorbing states, which is the same for all players. We carry out our study for a simple class of continuous-time models with It{\^o}-type dynamics with mean field interaction over a bounded time horizon. More specifically, the vector of private states $\mathbf X^N = (X^N _1 ,\ldots,X_N ^N)$ in the $N$-player game is assumed to evolve according to
\begin{equation} \label{EqIntroDynamics}
\begin{split}
	X^{N}_{i}(t) = X^{N}_{i}(0) &+  \int_{0}^{t} \left( u_{i}(s,\boldsymbol{X}^{N}) + \bar{b}\left(t, X^{N}_{i}(t),\int_{\mathbb{R}^{d}} w(y) \pi^{N}(s,dy)\right)\right)ds 
	\\&+ \sigma W^{N}_{i}(t),\quad t\in [0,T],\; i\in \{1,\ldots,N\},
\end{split}
\end{equation}
where $\boldsymbol{u}= (u_{1},\ldots,u_{N})$ is a vector of feedback strategies with full state information (to be specified below), $W^{N}_{1},\ldots,W^{N}_{N}$ are independent $d$-dimensional Wiener processes defined on some filtered probability space, $\sigma$ is some dispersion matrix, which we assume to be constant for simplicity, and $\bar b$, $w$ are given deterministic functions. Moreover, $\pi^{N}(t,\cdot)$ is the (random) empirical measure of the players' states at time $t$ that have not left $O$, that is,   
\[
	\pi^{N} (t,\cdot)\doteq \begin{cases}
\frac{1}{\bar{N}^{N}(t)}\sum_{j=1}^{N} \mathbf{1}_{[0,\tau^{X^{N}_{j}})}(t)\cdot  \delta_{X^{N}_{j}(t)}(\cdot) &\text{if } \bar{N}^{N}(t) > 0,\\
	\delta_{0}(\cdot) &\text{if } \bar{N}^{N}(t) = 0,
	\end{cases}
\]
where $\bar{N}^{N} (t)\doteq \sum_{j=1}^{N}\mathbf{1}_{[0,\tau^{X^{N}_{j}})}(t)$ is the number of players still in the game at time $t$ and
\[
	\tau^{X^N _j}\doteq \inf\left\{ t \in [0,T]:  X^N _j(t)\notin O\right\},
\]
denotes the time of first exit of $X^N _j$ from $O$, with the convention $\inf \emptyset = \infty$. By definition, $\pi^{N} (t,\cdot)$ equals the Dirac measure in $0\in \mathbb{R}^{d}$ if all players have left $O$ by time $t$. The choice of $\delta_{0}$ in this case is arbitrary and has no impact on the game; this will be clear from the definition of the cost functionals $J^{N}_{i}$ below. Heuristically, when some player exits the set $O$, he/she does not contribute anymore to the empirical measure $\pi^N (t,\cdot)$, which is computed with respect to the ``survivors'' only. Notice that the controls appear linearly in (\ref{EqIntroDynamics}) only for the sake of simplicity. Even though more general dependencies could be considered, we do not aim at giving the minimal set of assumptions under which our results hold true.

Each player wants to minimize expected costs according to a given cost functional over the random time horizon determined as the minimum between the player's time of first exit from $O$ and the overall time horizon $T$. More precisely, player $i$ evaluates a strategy vector $\boldsymbol{u} = (u_{1},\ldots,u_{N})$ according to
\begin{multline*}
	J^{N}_{i}(\boldsymbol{u})\doteq \Mean\Biggl[ \int_{0}^{\tau^{N}_{i}} f\left(s,X^{N}_{i}(s),\int_{\mathbb{R}^{d}} w(y)\pi^{N}(s,dy),u_{i}\left(s,\boldsymbol{X}^{N}\right) \right)ds \\
	+ F\left(\tau^{N}_{i}, X^{N}_{i}(\tau^{N}_{i}) \right) \Biggr],
\end{multline*}
where $\boldsymbol{X}^{N}$ is the solution of Eq.~\eqref{EqIntroDynamics} under $\boldsymbol{u}$ and $\tau_i ^N \doteq \tau^{X^N _i} \wedge T$ is the (random) duration of the game for player $i$. Notice that the cost coefficients $f$, $F$ are the same for all players. As in the dynamics, we have mean field interaction in the costs through the renormalized empirical measures $\pi^N (t,\cdot)$. For simplicity, we only consider finite-dimensional dependencies on the measure variable, namely through integrals of the vector-valued function $w$. More details on the setting with all the technical assumptions will be given in the next sections.

The presence of an absorbing set can be motivated by economic models for credit risk and corporate finance. There, the players can be interpreted as interacting firms striving to maximize some objective functional, for instance the expected value of discounted future dividends, or as banks controlling their rate of borrowing/lending to a central bank as in the systemic risk model proposed by  \citet{carmona-et-al15}. Within both contexts, the absorbing boundary can be naturally seen as a default barrier as in structural models for credit risk. In this paper, we concentrate on the mathematical properties of this family of games, while we postpone their possible applications to future research.

For our class of games, we focus on the construction of approximate Nash equilibria for the $N$-player games through solutions of the corresponding \mbox{MFG}. Our main contributions are as follows:

\begin{itemize}

\item  We introduce the limit model corresponding to the above $N$-player games as $N\to\infty$, namely the \mbox{MFG} with absorption. For a solution of the \mbox{MFG}, the renormalized empirical measures $\pi^N (t,\cdot)$ are replaced by a flow of conditional laws; see Definition~\ref{DefMFGSolution}.

\item Under a non-degeneracy condition on the dispersion matrix $\sigma$, we prove that any regular feedback solution of the \mbox{MFG} induces a sequence of approximate Nash equilibria for the $N$-player games with approximation error tending to zero as $N\to\infty$; see Theorem~\ref{ThApproximateNash}. Here, ``regular'' means that the optimal feedback strategy is (almost surely) continuous with respect to the state variable.

\item Under the same non-degeneracy condition on $\sigma$, we prove existence of a solution in feedback form to the \mbox{MFG} with absorption; see Theorem~\ref{nash-existence}. Moreover, under some additional conditions, we show that the optimal feedback strategies are Markovian and continuous in the state variable; see Proposition~\ref{prop:cont} and Corollary~\ref{crll:nash-cont}. In that situation, we briefly sketch what would be the \mbox{PDE} approach to mean field games with absorption.

\item In the degenerate case, i.e.\ when $\sigma$ may vanish, we provide a counter-example where the solution of the limit \mbox{MFG} is not even nearly optimal (in the sense of inducing approximate Nash equilibria) for the $N$-player games. This is in contrast with what happens in the absence of an absorbing set.

\end{itemize}

The proof of Theorem~\ref{ThApproximateNash} on approximate Nash equilibria is based on weak convergence arguments, controlled martingale problems and a reformulation of the original dynamics and costs using path-dependent coefficients; see, in particular, \eqref{ExCoeffDrift} and \eqref{ExProofNashDrift}. This allows to work with solutions defined over the entire time interval $[0,T]$. The resulting description of the systems should be compared to the set-up used in \citet{carmonalacker15}. There, questions of existence and uniqueness of solutions finite horizon \mbox{MFGs} with non-degenerate noise and functional coefficients are studied through probabilistic methods, and approximate Nash equilibria are constructed from the \mbox{MFG}. Nonetheless, the results of \citet{carmonalacker15} cannot be applied directly to our situation, due in part to different continuity assumptions. What is more, approximate Nash equilibria are constructed there only for dynamics without mean field interaction \citep[][Theorem~4.2]{carmonalacker15}; this assumption implies an independence property not warranted in more general situations. The use of martingale problems in proving convergence to the McKean-Vlasov limit and propagation of chaos for weakly interacting systems has a long tradition; see, for instance, \citet{funaki84}, \citet{oelschlaeger84}, or \citet{meleard96}. In those works, the $N$-particle systems are usually assumed to be fully symmetric, and the dynamics of all particles are determined by the same coefficients. Here, we have to study the passage to the many player (particle) limit also in the presence of a deviating player, which destroys the symmetry of the prelimit systems. The absorbing boundary introduces a discontinuity into the dynamics so that a single deviating player might have a non-negligible effect on the limit behavior of the $N$-player empirical measures. This is the reason why we give a detailed proof of convergence in Appendix~\ref{AppConvergence}.

The proof of Theorem~\ref{nash-existence} on existence of a feedback solution for the limit \mbox{MFG} is based on results on \mbox{BSDEs} with random terminal horizon as in \citet{darlingpardoux97, briandhu98} and the use of the Brouwer-Schauder-Tychonov fixed point theorem applied to a suitable map in the same spirit as \citet{carmonalacker15}. For the continuity of the optimal feedback strategy, we rely on the classical \mbox{PDE} approach to optimal control and adapt to our setting the proof of a regularity result due to \citet{flemingrishel75}.

For the counter-example (given in Section~7), we consider systems with dispersion coefficient $\sigma$ equal to zero; the only source of randomness comes from the initial conditions. We construct, for a specific choice of the initial distribution, a feedback solution of the \mbox{MFG} that is Lipschitz continuous in the state variable and such that the induced strategy vectors are not asymptotically optimal for the corresponding $N$-player games. The reason for this non-convergence is a change in the controllability of the individual player dynamics between limit and prelimit systems in conjunction with the discontinuity of the costs introduced by the absorbing boundary. Also notice that the initial distribution we choose is singular with respect to Lebesgue measure. The counter-example thus holds little surprise from the point of view of optimal control theory. In the context of \mbox{MFGs} without absorption, on the other hand, the connection between solutions of the limit \mbox{MFG} and approximate Nash equilibria for the $N$-player games is known to be robust and persists even for systems with fully degenerate noise; see Theorem~2.11 in \citet{lacker16} for a general result in this direction. From this point of view, the counter-example seems to be interesting.

\paragraph{Related literature} Mean-field models similar to ours have been studied before in different contexts. In a first group of papers, such as \citet{giesecke-et-al13}, \citet{spiliopoulos-et-al14}, \citet{giesecke-et-al15}, a point process model of correlated defaults timing in a portfolio of firms has been introduced. More specifically, a firm defaults with some intensity which follows a mean-reverting jump-diffusion process that is driven by several terms, one of them being the default rate in the pool. This naturally induces a mean-field among the default intensities: every time one firm defaults a jump occurs in the default rate of the pool and hence in the intensities of the survivors. The effects of defaults fade away with time. A law of large number (\mbox{LLN}) for the default rate as the number $N$ of firms goes to infinity is proved in \citet{giesecke-et-al13}, while other results on the loss from default are analyzed in the companion papers \citet{spiliopoulos-et-al14, giesecke-et-al15}. Similar results have been obtained for the interacting particle models proposed by \citet{cvitanic-et-al12} and by \citet{dai-pra-et-al09}, where the effect of defaults on the survivors is permanent.

Apart from the fact that their setting is not controlled, the main difference between our model and theirs is that whenever, in our model, some diffusion is absorbed and the conditional empirical measure is updated accordingly, the diffusions still in $O$ are affected by it in a continuous way since the empirical measure appears only in the drift coefficient. Moreover, while in the setting of Giesecke, Spiliopoulos and coauthors the intensities can have a common noise, this is not included in our model for the sake of simplicity.

More recently, \citet{hamblyledger16}, motivated by various applications from large credit portfolios to neuroscience, have proposed a system of $N$ uncontrolled diffusions that are killed as soon as they go negative. Furthermore, their coefficients are functions of the current proportional loss, which is defined as the proportion out of $N$ of killed diffusion. They prove a \mbox{LLN} for the empirical measure of the population using some energy estimates in combination with weak convergence techniques.

Two more papers, which are related to ours, are those by \citet{delarue-et-al15a, delarue-et-al15b}. Motivated by applications in neuroscience, these authors study the well-posedness of an integrate-and-fire model and its approximation via particle systems. Mathematically speaking, they look at a nonlinear scalar \mbox{SDE} with a jump component responsible for resetting trajectories as soon as they hit a given threshold (occurrence of a ``spike'') and a singular drift term of mean field type (the ``nonlinear'' interaction) depending on the average number of spikes before current time. When the nonlinear interaction is too strong, any solution can have a blow-up in finite time. In our model, we could have a similar cascade effect, which would correspond to a situation when all players get absorbed before the finite horizon $T$. Investigating such a phenomenon in our setting is interesting on its own but goes beyond the scope of this paper.

In the applied literature, \mbox{MFGs} with absorption at zero have been recently considered in \cite{chansircar15a,chansircar15b} within the context of oligopolistic models with exhaustible resources. A common feature that their model shares with ours is that they also keep track of the fraction of active players remaining at time $t$, which appears in the objective functions (through the control) but not in the state variable. Moreover, they look at some particular cases, which are relevant from an economic perspective, and perform some asymptotic expansions corresponding to the case of ``small competition''. A rigorous study addressing existence and uniqueness issues in Chan and Sircar's model has been subsequently done by \cite{graberbensoussan16}.

The last work we mention here is by \citet{bensoussanfrehsegruen14} and bears a different relation to ours. There, the authors construct Nash equilibria using \mbox{PDE} methods for a class of stochastic differential games with a varying number of players (the dynamics are not of mean field type). While the maximum number of players is prescribed, players may leave the game (be pushed out or ``die''), and new players may join. In our case, we only allow players to leave (through absorption). A natural extension of our model would include a mechanism by which players enter the game.

\paragraph{Structure of the paper} The rest of this paper is organized as follows. Section~\ref{SectNotation} introduces some terminology and notation and sets the main assumptions on the dynamics as well as on the cost functionals. Section~\ref{SectPrelimitSystems} describes the setting of $N$-player games with absorption, while Section~\ref{SectLimitSystems} introduces the corresponding \mbox{MFG}. In Section~\ref{SectApproxNash}, one of the main results, namely the construction of approximate Nash equilibria for the $N$-player game from a solution of the limit problem, is stated and proved. Section~\ref{SectExistence} contains the results on the existence of (regular) feedback solutions for the \mbox{MFG}, in particular those with Markov feedback controls, as well as a sketch of the \mbox{PDE} approach to mean field games with absorption. In Section~\ref{SectCounterExmpl}, we provide the aforementioned counter-example in the case of degenerate noise. The technical results used in the paper are all gathered in the Appendix, including the aforementioned propagation-of-chaos-type results in Appendix~\ref{AppConvergence} and a uniqueness result for McKean-Vlasov equations in Appendix~\ref{AppUniqueness}.


\section{Preliminaries and assumptions} \label{SectNotation}

Let $d\in \mathbb{N}$, which will be the dimension of the space of private states, noise values, as well as control actions. The spaces $\mathbb{R}^{n}$ with $n\in \mathbb{N}$ are equipped with the standard Euclidean norm, always indicated by $|.|$. Choose $T > 0$, the finite time horizon.

For $\mathcal{S}$ a Polish space, let $\prbms{\mathcal{S}}$ denote the space of probability measures on $\Borel{\mathcal{S}}$, the Borel sets of $\mathcal{S}$. For $s\in \mathcal{S}$, let $\delta_{s}$ indicate the Dirac measure concentrated in $s$. Equip $\prbms{\mathcal{S}}$ with the topology of weak convergence of probability measures. Then $\prbms{\mathcal{S}}$ is again a Polish space.

Set $\mathcal{X}\doteq \mathbf{C}([0,T],\mathbb{R}^{d})$, which can be seen as the space of individual state trajectories. As usual, equip $\mathcal{X}$ with the topology of uniform convergence, which turns it into a Polish space. Let $\|\cdot\|_{\mathcal{X}}$ denote the supremum norm on $\mathcal{X}$. Denote by $\hat{X}$ the coordinate process on $\mathcal{X}$:
\begin{align*}
	& \hat{X}(t,\phi)\doteq \phi (t),& & t\in [0,T],& &\phi\in \mathcal{X}.&
\end{align*}
Let $(\mathcal{G}_{t})$ be the canonical filtration in $\mathcal{B}(\mathcal{X})$, that is,
\[
	\mathcal{G}_{t}\doteq \sigma \left( \hat{X}(s) : 0\leq s\leq t\right),\quad t\in [0,T].
\]

Given $N\in \mathbb{N}$, we will use the usual identification of $\mathcal{X}^{N} = \times^{N}\mathcal{X}$ with the space $\mathbf{C}([0,T],\mathbb{R}^{N\cdot d})$; $\mathcal{X}^{N}$, too, will be equipped with the topology of uniform convergence. We call a function $u$ defined on $[0,T]\times \mathcal{X}^{N}$ with values in some measurable space \emph{progressively measurable} if $u$ is measurable and, for every $t\in [0,T]$, all $\boldsymbol{\phi}, \boldsymbol{\tilde{\phi}} \in \mathcal{X}^{N}$,
\[
	u(t,\boldsymbol{\phi}) = u(t,\boldsymbol{\tilde{\phi}}) \text{ whenever } \boldsymbol{\phi}(s) = \boldsymbol{\tilde{\phi}}(s) \text{ for all } s\in [0,t].
\]

Let $\Gamma$ be a closed subset of $\mathbb{R}^{d}$, the set of control actions, or action space.

Let $O \subset \mathbb{R}^{d}$ be an open set, the set of non-absorbing states. For $Y$ an $\mathbb{R}^{d}$-valued process defined on some probability space $(\Omega,\mathcal{F},\Prb)$ over the time interval $[0,T]$, let
\[
	\tau^{Y}(\omega)\doteq \inf\left\{t \in [0,T]:  Y(t,\omega)\notin O\right\},\quad \omega\in \Omega,
\]
denote the random time of first exit of $Y$ from $O$, with the convention $\inf\emptyset = \infty$. Clearly, if $Y$ has continuous trajectories and is adapted to some filtration $(\mathcal{F}_{t})$ in $\mathcal{F}$, then $\tau^{Y}$ is an $(\mathcal{F}_{t})$-stopping time. In this case, $\tau^{Y} = \inf\left\{t \in [0,T]:  Y(t)\in \partial O\right\}$, where $\partial O$ denotes the boundary of $O$.

Let $d_{0}\in \mathbb{N}$, and let $w\!: \mathbb{R}^{d} \rightarrow \mathbb{R}^{d_{0}}$,
\begin{align*}
	& \bar{b}\!: [0,T]\times \mathbb{R}^{d}\times \mathbb{R}^{d_{0}} \rightarrow \mathbb{R}^{d},& & \sigma \in \mathbb{R}^{d\times d},& \\
	& f\!: [0,T]\times \mathbb{R}^{d}\times \mathbb{R}^{d_{0}}\times \Gamma \rightarrow [0,\infty),& & F\!: [0,T]\times \mathbb{R}^{d} \rightarrow [0,\infty). &
\end{align*}
The function $w$ will denote an integrand for the measure variable in the drift and the running costs, respectively, $\bar{b}$ a function of the drift integral, $\sigma$ the dispersion coefficient of the dynamics, while $f$, $F$ quantify the running and terminal costs, respectively. Notice that the dispersion coefficient $\sigma$ is a constant matrix and that the cost coefficients $f$, $F$ are non-negative functions. Let us make the following assumptions:
\begin{Hypenv}

	\item \label{HypMeasBounded} Boundedness and measurability: $w$, $\bar{b}$, $f$, $F$ are Borel measurable functions uniformly bounded by some constant $K > 0$. 
	 
	\item \label{HypCont} Continuity: $w$, $f$, $F$ are continuous, and $\bar{b}(t,\cdot,\cdot)$ is continuous, uniformly in $t\in [0,T]$.
		
	\item \label{HypLipschitz} Lipschitz continuity of $\bar{b}$: there exists $\bar{L} > 0$ such that for all $x,\tilde{x}\in \mathbb{R}^{d}$, $y,\tilde{y}\in \mathbb{R}^{d_{0}}$,
	\[
		\sup_{t\in[0,T]} \left| \bar{b}(t,x,y) - \bar{b}(t,\tilde{x},\tilde{y}) \right| \leq \bar{L}\left( \left|x - \tilde{x}\right| + \left|y - \tilde{y}\right| \right).
	\]

	\item \label{HypActionSpace} Action space: $\Gamma \subset \mathbb{R}^{d}$ is compact and convex (and non-empty).
	
	\item \label{HypStateSpace} State space: $O \subset \mathbb{R}^{d}$ is non-empty, open, and bounded such that $\partial O$ is a $C^{2}$-manifold.

\end{Hypenv}

The results of Sections \ref{SectApproxNash} and \ref{SectExistence} will be established under the following additional assumption:
\begin{hypenv}
	\item[(ND)] \label{HypAddND} Non-degeneracy: $\sigma$ is a matrix of full rank.
\end{hypenv}


\section{$N$-player games} \label{SectPrelimitSystems}

Let $N\in \mathbb{N}$ be the number of players. Denote by $X^{N}_{i}(t)$ the private state of player $i$ at time $t\in [0,T]$. The evolution of the players' states depends on the strategies they choose as well as the initial distribution of states, which we indicate by $\nu_{N}$ (thus, $\nu_{N}\in \prbms{\mathbb{R}^{N\times d}}$). We assume that $\supp(\nu_{N}) \subset \times^{N} O$ and that $\nu_{N}$ is symmetric in the sense that
\[
	\nu^{N}\circ \mathfrak{s}^{-1} = \nu^{N}
\]
for all maps $\mathfrak{s}\!: (\mathbb{R}^{d})^{N} \rightarrow (\mathbb{R}^{d})^{N}$ of the form $(x_{1},\ldots,x_{N})\mapsto (x_{p(1)},\ldots,x_{p(N)})$ for some permutation $p$ of $(1,\ldots,N)$.

Here, we consider players using feedback strategies with full state information (up to the current time). Thus, let $\mathcal{U}_{N}$ denote the set of all progressively measurable functions $u\!: [0,T]\times \mathcal{X}^{N} \rightarrow \Gamma$. Elements of $\mathcal{U}_{N}$ represent individual strategies. A vector $(u_{1},\ldots,u_{N})$ of individual strategies is called a strategy vector or strategy profile. Given a strategy vector $\boldsymbol{u}= (u_{1},\ldots,u_{N})\in \times^{N}\mathcal{U}_{N}$, consider the system of equations 
\begin{equation} \label{EqPrelimitDynamics}
\begin{split}
	X^{N}_{i}(t) = X^{N}_{i}(0) &+  \int_{0}^{t} \left( u_{i}(s,\boldsymbol{X}^{N}) + \bar{b}\left(s, X^{N}_{i}(s), \int_{\mathbb{R}^{d}} w(y) \pi^{N}(s,dy)\right)\right)ds 
	\\&+ \sigma W^{N}_{i}(t),\quad t\in [0,T],\; i\in \{1,\ldots,N\},
\end{split}
\end{equation}
where $\boldsymbol{X}^{N} = (X^{N}_{1},\ldots,X^{N}_{N})$, $W^{N}_{1},\ldots,W^{N}_{N}$ are independent $d$-dimensional Wiener processes defined on some filtered probability space $(\Omega,\mathcal{F},(\mathcal{F}_{t}),\Prb)$ and $\pi^{N}(t,\cdot)$ is the empirical measure of the players' states at time $t$ that have not left $O$, that is,   
\[
	\pi^{N}_{\omega}(t,\cdot)\doteq \begin{cases}
\frac{1}{\bar{N}^{N}_{\omega}(t)}\sum_{j=1}^{N} \mathbf{1}_{[0,\tau^{X^{N}_{j}}(\omega))}(t)\cdot  \delta_{X^{N}_{j}(t,\omega)}(\cdot) &\text{if } \bar{N}^{N}_{\omega}(t) > 0,\\
	\delta_{0}(\cdot) &\text{if } \bar{N}^{N}_{\omega}(t) = 0,
	\end{cases}
\]
where
\[
	\bar{N}^{N}_{\omega}(t)\doteq \sum_{j=1}^{N}\mathbf{1}_{[0,\tau^{X^{N}_{j}}(\omega))}(t),\quad \omega\in \Omega.
\]

It will be convenient to rewrite \eqref{EqPrelimitDynamics} as a system of particles interacting through their unconditional empirical measure on the path space. To this end, set $\tau\doteq \tau^{\hat{X}}$, that is,
\[
	\tau(\phi) \doteq \inf\left\{ t \in [0,T] : \phi(t)\notin O\right\},\quad \phi\in \mathcal{X},
\]
and define $b\!: [0,T]\times \mathcal{X} \times \prbms{\mathcal{X}}\times \Gamma \rightarrow \mathbb{R}^{d}$ by
\begin{equation} \label{ExCoeffDrift}
	b(t,\phi,\theta,\gamma) \doteq \begin{cases}
	\gamma + \bar{b}\left(t, \phi(t), \frac{\int w(\tilde{\phi}(t)) \mathbf{1}_{[0,\tau(\tilde{\phi}))}(t) \theta(d\tilde{\phi})}{\int \mathbf{1}_{[0,\tau(\tilde{\phi}))}(t) \theta(d\tilde{\phi})} \right)
	&\text{if } \theta(\tau > t) > 0,\\
	\gamma + \bar{b}\left(t,\phi(t),w(0)\right) &\text{if } \theta(\tau > t) = 0.
	\end{cases}
\end{equation}
Then $b$ is measurable and progressive in the sense that, for all $t\in [0,T]$, all $\gamma\in \Gamma$,
\[
	b(t,\phi,\theta,\gamma) = b(t,\tilde{\phi},\tilde{\theta},\gamma) \text{ whenever } \phi_{|[0,t]} = \tilde{\phi}_{|[0,t]} \text{ and } \theta_{|\mathcal{G}_{t}} = \tilde{\theta}_{|\mathcal{G}_{t}}.
\]
The solutions of \eqref{EqPrelimitDynamics} are thus equivalently described by:
\begin{equation} \label{EqPrelimitDynamics2}
\begin{split}
	X^{N}_{i}(t) &= X^{N}_{i}(0) +  \int_{0}^{t} b\left(s,X^{N}_{i},\mu^{N}, u_{i}(s,\boldsymbol{X}^{N})\right)ds + \sigma W^{N}_{i}(t), \\
	&t\in [0,T],\; i\in \{1,\ldots,N\},
\end{split}
\end{equation}
where $W^{N}_{1},\ldots,W^{N}_{N}$ are independent $d$-dimensional Wiener processes defined on some filtered probability space $(\Omega,\mathcal{F},(\mathcal{F}_{t}),\Prb)$, and $\mu^{N}$ is the empirical measure of the players' state trajectories, that is,   
\[
	\mu^{N}_{\omega}(B)\doteq  \frac{1}{N} \sum_{j=1}^{N}  \delta_{X^{N}_{j}(\cdot,\omega)}(B),\quad B\in \Borel{\mathcal{X}},\; \omega\in \Omega.
\]

A solution of Eq.~\eqref{EqPrelimitDynamics} (equivalently, of Eq.~\eqref{EqPrelimitDynamics2}) under $\boldsymbol{u}\in \times^{N}\mathcal{U}_{N}$ with initial distribution $\nu_{N}$ is therefore a triple $((\Omega,\mathcal{F},(\mathcal{F}_{t}),\Prb),\boldsymbol{W}^{N},\boldsymbol{X}^{N})$ where $(\Omega,\mathcal{F},(\mathcal{F}_{t}),\Prb)$ is a filtered probability space satisfying the usual hypotheses, $\boldsymbol{W}^{N} = (W^{N}_{1},\ldots,W^{N}_{N})$ a vector of independent $d$-dimensional $(\mathcal{F}_{t})$-Wiener processes, and $\boldsymbol{X}^{N} = (X^{N}_{1},\ldots,X^{N}_{N})$ a vector of continuous $\mathbb{R}^{d}$-valued $(\mathcal{F}_{t})$-adapted processes such that Eq.~\eqref{EqPrelimitDynamics} (resp.\ Eq.~\eqref{EqPrelimitDynamics2}) holds $\Prb$-almost surely with strategy vector $\boldsymbol{u}$ and $\Prb\circ(\boldsymbol{X}^{N}(0))^{-1} = \nu_{N}$.

Let $\mathcal{U}^{N}_{fb}$ be the set of all strategy vectors $\boldsymbol{u}\in \times^{N}\mathcal{U}_{N}$ such that Eq.~\eqref{EqPrelimitDynamics} under $\boldsymbol{u}$ with initial distribution $\nu_{N}$ possesses a solution that is unique in law. If the non-degeneracy assumption \hypref{HypAddND} holds in addition to \Hyprefall, then Eq.~\eqref{EqPrelimitDynamics} is well posed given any strategy vector:

\begin{prop} \label{PropPrelimitExistUnique}
	Grant \hypref{HypAddND} in addition to \Hyprefall. Then $\mathcal{U}^{N}_{fb} = \times^{N}\mathcal{U}_{N}$.
\end{prop}

\begin{proof}
Let $\boldsymbol{u}\in \times^{N}\mathcal{U}_{N}$. The system of equations \eqref{EqPrelimitDynamics2} can be rewritten as one stochastic differential equation with state space $\mathbb{R}^{N\times d}$ driven by an $N\times d$-dimensional standard Wiener process with drift coefficient $\boldsymbol{b}_{\boldsymbol{u}}\!: [0,T]\times \mathcal{X}^{N} \rightarrow \mathbb{R}^{N\times d}$ given by
\[
\begin{split}
	&\boldsymbol{b}_{\boldsymbol{u}}(t,\boldsymbol{\phi}) \\
	&\doteq \trans{\left( b\left(t,\phi_{1},\frac{1}{N}\sum_{j=1}^{N}\delta_{\phi_{j}},u_{1}(s,\boldsymbol{\phi})\right),\ldots, b\left(t,\phi_{N},\frac{1}{N}\sum_{j=1}^{N}\delta_{\phi_{j}},u_{N}(s,\boldsymbol{\phi})\right) \right)}
	\end{split}
\]
and non-degenerate constant diffusion coefficient. Notice that $\boldsymbol{b}_{\boldsymbol{u}}$ is bounded and progressive with respect to the natural filtration in $\Borel{\mathcal{X}^{N}}$. Existence of a weak solution and uniqueness in law are now a consequence of Girsanov's theorem and the Stroock-Varadhan martingale problem;  cf.\ V.27 in \citet[pp.\,177-178]{rogerswilliams00b}.
\end{proof}

The $i$-th player evaluates a strategy vector $\boldsymbol{u} = (u_{1},\ldots,u_{N})\in \mathcal{U}^{N}_{fb}$ according to the cost functional
\begin{multline*}
	J^{N}_{i}(\boldsymbol{u})\doteq \Mean\Biggl[ \int_{0}^{\tau^{N}_{i}} f\left(s,X^{N}_{i}(s),\int_{\mathbb{R}^{d}} w(y)\pi^{N}(s,dy),u_{i}\left(s,\boldsymbol{X}^{N}\right) \right)ds \\
	+ F\left(\tau^{N}_{i}, X^{N}_{i}(\tau^{N}_{i}) \right) \Biggr],
\end{multline*}
where $\boldsymbol{X}^{N} = (X^{N}_{1},\ldots,X^{N}_{N})$ and $((\Omega,\mathcal{F},(\mathcal{F}_{t}),\Prb),\boldsymbol{W}^{N},\boldsymbol{X}^{N})$ is a solution of Eq.~\eqref{EqPrelimitDynamics} under $\boldsymbol{u}$ with initial distribution $\nu_{N}$,
\[
	\tau^{N}_{i}(\omega) \doteq \tau^{X^{N}_{i}}(\omega) \wedge T,\quad \omega\in \Omega,
\]
the random time horizon for player $i\in \{1,\ldots,N\}$, and $\pi^{N}(\cdot)$ the conditional empirical measure process induced by $(X^{N}_{1},\ldots,X^{N}_{N})$. The cost functional is well defined, and it is finite thanks to assumption \Hypref{HypMeasBounded}.

Given a strategy vector $\boldsymbol{u} = (u_{1},\ldots,u_{N})$ and an individual strategy $v\in \mathcal{U}_{N}$, let $[\boldsymbol{u}^{-i},v]\doteq (u_{1},\ldots,u_{i-1},v,u_{i+1},\ldots,u_{N})$ indicate the strategy vector that is obtained from $\boldsymbol{u}$ by replacing $u_{i}$, the strategy of player $i$, with $v$.

\begin{defn} \label{DefNash}
Let $\epsilon \geq 0$. A strategy vector $\boldsymbol{u} = (u_{1},\ldots,u_{N})\in \mathcal{U}^{N}_{fb}$ is called an \emph{$\epsilon$-Nash equilibrium} for the $N$-player game if for every $i\in \{1,\ldots,N\}$, every $v\in \mathcal{U}_{N}$ such that $[\boldsymbol{u}^{-i},v]\in \mathcal{U}^{N}_{fb}$,
\begin{equation} \label{EqDefNash}
	J^{N}_{i}(\boldsymbol{u}) \leq J^{N}_{i}\left([\boldsymbol{u}^{-i},v]\right) + \epsilon.
\end{equation}

If $\boldsymbol{u}$ is an $\epsilon$-Nash equilibrium with $\epsilon = 0$, then $\boldsymbol{u}$ is called a \emph{Nash equilibrium}.
\end{defn}

According to Definition~\ref{DefNash}, we consider the Nash equilibrium property with respect to feedback strategies with full state information (i.e., the states of the vector of individual processes up to current time). 


\section{Mean field games} \label{SectLimitSystems}

Let $\mathcal{M}$ denote the space of measurable flows of measures, that is,
\[
	\mathcal{M}\doteq \left\{ \mathfrak{p}\!: [0,T]\rightarrow \prbms{\mathbb{R}^{d}} : \mathfrak{p} \text{ is Borel measurable} \right\}.
\]
Given a flow of measures $\mathfrak{p}\in \mathcal{M}$ and a feedback strategy $u\in \mathcal{U}_{1}$, consider the equation
\begin{equation} \label{EqLimitDynamics}
\begin{split}
	X(t) &= X(0) +  \int_{0}^{t} \left( u(s,X) + \bar{b}\left(s,X(s), \int_{\mathbb{R}^{d}} w(y) \mathfrak{p}(s,dy)\right) \right)ds \\
	&\quad + \sigma W(t),\quad t\in [0,T],
\end{split}
\end{equation}
where $W$ is a $d$-dimensional Wiener process defined on some filtered probability space $(\Omega,\mathcal{F},(\mathcal{F}_{t}),\Prb)$.

Let $\mathcal{U}_{fb}$ denote the set of all feedback strategies $u\in \mathcal{U}_{1}$ such that Eq.~\eqref{EqLimitDynamics} possesses a solution that is unique in law given any initial distribution with support contained in $O$.

\begin{prop} \label{PropLimitExistUnique}
	Grant \hypref{HypAddND} in addition to \Hyprefall. Then $\mathcal{U}_{fb} = \mathcal{U}_{1}$.
\end{prop}

\begin{proof}
Existence and uniqueness in law are again a consequence of Girsanov's theorem;  cf.\ Proposition~\ref{PropPrelimitExistUnique}.
\end{proof}

The costs associated with a strategy $u\in \mathcal{U}_{fb}$, a flow of measures $\mathfrak{p}$, and an initial distribution $\nu\in \prbms{\mathbb{R}^{d}}$ with support in $O$ are given by
\begin{multline*}
	J(\nu,u;\mathfrak{p})\doteq \Mean\Biggl[ \int_{0}^{\tau} f\left(s,X(s),\int_{\mathbb{R}^{d}} w(y)\mathfrak{p}(s,dy),u(s,X) \right)ds \\
	+ F\left(\tau, X(\tau) \right) \Biggr],
\end{multline*}
where $((\Omega,\mathcal{F},(\mathcal{F}_{t}),\Prb),W,X)$ is a solution of \eqref{EqLimitDynamics} under $u$ with initial distribution $\nu$, and $\tau \doteq \tau^{X}\wedge T$ the random time horizon.

We will measure the minimal costs with respect to a class of stochastic open-loop strategies. To this end, let $\mathcal{A}$ denote the set of all quadruples $((\Omega,\mathcal{F},(\mathcal{F}_{t}),\Prb),\xi,\alpha,W)$ such that $(\Omega,\mathcal{F},(\mathcal{F}_{t}),\Prb)$ is a filtered probability space satisfying the usual hypotheses, $\xi$ an $\mathcal{F}_{0}$-measurable random variable with values in $O$, $\alpha$ a $\Gamma$-valued $(\mathcal{F}_{t})$-progressively measurable process, and $W$ a $d$-dimensional $(\mathcal{F}_{t})$-Wiener process. Any strategy $u\in \mathcal{U}_{fb}$, together with an initial distribution, induces an element of $\mathcal{A}$.

Given $((\Omega,\mathcal{F},(\mathcal{F}_{t}),\Prb),\xi,\alpha,W) \in \mathcal{A}$ and a flow of probability measures $\mathfrak{p}\in \mathcal{M}$, consider the stochastic integral equation
\begin{equation} \label{EqLimitDynamicsOL}
\begin{split}
	X(t) &= \xi +  \int_{0}^{t} \left( \alpha(s) + \bar{b}\left(s,X(s),\int_{\mathbb{R}^{d}} w(y) \mathfrak{p}(s,dy)\right) \right) ds \\
	&+ \sigma W(t),\quad t\in [0,T].
\end{split}
\end{equation}
Thanks to \Hypref{HypLipschitz}, $X$ is determined through Eq.~\eqref{EqLimitDynamicsOL} with $\Prb$-probability one; in particular, $X$ is defined on the given stochastic basis. The minimal costs associated with a flow of measures $\mathfrak{p}$ and an initial distribution $\nu\in \prbms{\mathbb{R}^{d}}$ with support in $O$ are now given by
\begin{multline*}
	V(\nu;\mathfrak{p})\doteq \inf_{((\Omega,\mathcal{F},(\mathcal{F}_{t}),\Prb),\xi,\alpha,W) \in \mathcal{A} : \Prb\circ \xi^{-1} = \nu} \\
	\Mean\left[ \int_{0}^{\tau} f\left(s,X(s),\int_{\mathbb{R}^{d}} w(y)\mathfrak{p}(s,dy),\alpha(s) \right)ds 
		+ F\left(\tau, X(\tau) \right) \right],
\end{multline*}
where $X$ is the process determined by $((\Omega,\mathcal{F},(\mathcal{F}_{t}),\Prb),\xi,\alpha,W)$ via Eq.~\eqref{EqLimitDynamicsOL}, and $\tau \doteq \tau^{X}\wedge T$ the random time horizon.

\begin{rem} \label{RemValueFnct}
Since any admissible feedback strategy induces a stochastic open-loop strategy, we always have
\[
	\inf_{u\in \mathcal{U}_{fb}} J(\nu,u;\mathfrak{p}) \geq V(\nu;\mathfrak{p}).
\]
If the non-degeneracy assumption \hypref{HypAddND} holds in addition to \Hyprefall, if $\bar{b}$ is continuous also in the time variable, and if the flow of measures $\mathfrak{p}$ is such that the mapping $t\mapsto \int w(y)\mathfrak{p}(t,dy)$ is continuous, then
\[
	\inf_{u\in \mathcal{U}_{fb}} J(\nu,u;\mathfrak{p}) = V(\nu;\mathfrak{p}).
\]
This follows, for instance, from the results of \citet{elkarouietalii87}; see, in particular, Proposition~2.6 and Remark~2.6.b) as well as Theorem~6.7 and Section~7 therein. Alternatively, one may use time discretization and discrete-time dynamic programming in analogy to Lemma~4.3 in \citet{fischer17}.

\end{rem}

The notion of solution we consider for the mean field game is the following:

\begin{defn} \label{DefMFGSolution}
A \emph{feedback solution of the mean field game} is a triple $(\nu,u,\mathfrak{p})$ such that
\begin{enumerate}[(i)]
	\item $\nu\in \prbms{\mathbb{R}^{d}}$ with $\supp(\nu) \subset O$, $u \in \mathcal{U}_{fb}$, and $\mathfrak{p}\in \mathcal{M}$;
	
	\item optimality property: strategy $u$ is optimal for $\mathfrak{p}$ and initial distribution $\nu$, that is,
\[
	J(\nu,u;\mathfrak{p}) = V(\nu;\mathfrak{p});
\]

	\item conditional mean field property: if $((\Omega,\mathcal{F},(\mathcal{F}_{t}),\Prb),W,X)$ is a solution of Eq.~\eqref{EqLimitDynamics} with flow of measures $\mathfrak{p}$, strategy $u$, and initial distribution $\nu$, then $\mathfrak{p}(t) = \Prb\left( X(t)\in \cdot \;|\; \tau^X > t \right)$ for every $t\in [0,T]$ such that $\Prb\left(\tau^X > t \right) > 0$.

\end{enumerate}
\end{defn}


\section{Approximate Nash equilibria from the mean field game} \label{SectApproxNash}

Throughout this section, we assume that the non-degeneracy condition \hypref{HypAddND} holds. If we have a feedback solution of the mean field game that satisfies a mild regularity condition, then we can construct a sequence of approximate Nash equilibria for the corresponding $N$-player game. This approximation result is the content of Theorem~\ref{ThApproximateNash} below.

In order to state the regularity condition, we set, for $\nu\in \prbms{\mathbb{R}^{d}}$ with support in $O$,
\begin{equation} \label{ExRefMeasure}
	\Theta_{\nu} \doteq \Law\left((\xi + \sigma W(t))_{t\in [0,T]} \right),
\end{equation}
where $W$ is a $d$-dimensional Wiener process and $\xi$ an independent $\mathbb{R}^{d}$-valued random variable with distribution $\nu$. Clearly, $\Theta_{\nu}$ is well-defined as an element of $\prbms{\mathcal{X}}$. 

The proof of Theorem~\ref{ThApproximateNash} below relies on the convergence, uniqueness  and regularity results given in the Appendix. The following subsets of probability measures will play an important role as they characterize the possible distributions of the limit processes. For $c \geq 0$, let $\mathcal{Q}_{\nu,c}$ be the set of all laws $\theta\in \prbms{\mathcal{X}}$ such that $\theta = \Prb\circ X^{-1}$ where
\begin{equation} \label{EqBoundedControlDynamics}
	X(t) \doteq \xi + \int_{0}^{t}v(s)ds + \sigma W(t),\quad t\in [0,T],
\end{equation}
$W$ is an $\mathbb{R}^{d}$-valued $(\mathcal{F}_{t})$-Wiener process defined on some filtered probability space $(\Omega,\mathcal{F},(\mathcal{F}_{t}),\Prb)$, $\xi$ is an $\mathbb{R}^{d}$-valued $\mathcal{F}_{0}$-measurable random variable with distribution $\Prb\circ \xi^{-1} = \nu$, and $v$ is an $\mathbb{R}^{d}$-valued $(\mathcal{F}_{t})$-progressively measurable bounded process with $\|v\|_{\infty} \leq c$. Clearly, $\mathcal{Q}_{\nu,0} = \{\Theta_{\nu}\}$. Also note that $\mathcal{Q}_{\nu,c}$ is compact and that any measure $\theta\in \mathcal{Q}_{\nu,c}$ is equivalent to $\Theta_{\nu}$; see Lemmata \ref{LemmaAppRegularityCompact} and \ref{LemmaAppRegularity}, respectively, in the Appendix.

We recall that a sequence of symmetric probability measures $(\nu_{N})_{N\in \mathbb{N}}$, with $\nu_N \in \mathcal P(\mathbb R^{N \cdot d})$ for all $N$, is called $\nu$-chaotic for some $\nu\in \prbms{\mathbb{R}^{d}}$ if, for each $k \in \mathbb{N}$ and for any choice of bounded and continuous functions $\psi_i : \mathbb R^d \to \mathbb R$, $i \in \{1,\ldots,k\}$, we have
\[
	\lim_{N \to \infty} \int_{\mathbb R^{N\cdot d}} \prod_{i=1}^k \psi_i (x_i)\; d\nu_N (x_1,\ldots,x_N) = \prod_{i=1}^k \int_{\mathbb R^d} \psi_i (x_i) \nu(dx_i).
\]

\begin{thrm} \label{ThApproximateNash}
Grant \hypref{HypAddND} in addition to \Hyprefall. Suppose the sequence of initial distributions $(\nu_{N})_{N\in \mathbb{N}}$ is $\nu$-chaotic for some $\nu\in \prbms{\mathbb{R}^{d}}$ with support in $O$. If $(\nu,u,\mathfrak{p})$ is a feedback solution of the mean field game such that, for Lebesgue-almost every $t\in [0,T]$,
\[
	\Theta_{\nu} \left(\left\{ \phi\in \mathcal{X} \;:\; u(t,\cdot) \text{ is discontinuous at }\phi \right\} \right) = 0,
\]
then
\[
	u^{N}_{i}(t,\boldsymbol{\phi}) \doteq u(t,\phi_{i}),\quad t\in [0,T],\; \boldsymbol{\phi}=(\phi_{1},\ldots,\phi_{N})\in \mathcal{X}^{N},\; i\in \{1,\ldots,N\},
\]
defines a strategy vector $\boldsymbol{u}^{N} = (u^{N}_{1},\ldots,u^{N}_{N}) \in \mathcal{U}^{N}_{fb}$. Moreover, for every $\epsilon > 0$, there exists $N_{0} = N_{0}(\epsilon)\in \mathbb{N}$ such that $\boldsymbol{u}^{N}$ is an $\epsilon$-Nash equilibrium for the $N$-player game whenever $N \geq N_{0}$.
\end{thrm}

\begin{proof}[Proof]
Let $(\nu,u,\mathfrak{p})$ be a feedback solution of the mean field game according to Definition~\ref{DefMFGSolution}; such a solution exists by hypothesis. For $N\in \mathbb{N}$, consider the strategy $\boldsymbol{u}^{N} = (u^{N}_{1},\ldots,u^{N}_{N})$ as in the statement. By Proposition~\ref{PropPrelimitExistUnique}, $\mathcal{U}^{N}_{fb} = \times^{N} \mathcal{U}_{N}$; in particular, $\boldsymbol{u}^{N}$ is a feedback strategy vector in $\mathcal U^N _{fb}$.

Let $\epsilon > 0$. By symmetry, it is enough to verify the $\epsilon$-Nash property for player one only. Thus, we have to show that there exists $N_{0} = N_{0}(\epsilon)\in \mathbb{N}$ such that for all $N \geq N_{0}$,
\begin{equation} \label{EqProofNash}
	J^{N}_{1}(\boldsymbol{u}^{N}) \leq \inf_{v\in \mathcal{U}_{N}} J^{N}_{1}\left([\boldsymbol{u}^{N,-1},v]\right) + \epsilon.
\end{equation}

\bigskip
\textbf{Step 1.}
We rewrite the system dynamics and costs in such way that we can apply the convergence, uniqueness and regularity results of the Appendix. In accordance with \eqref{ExCoeffDrift}, the definition of $b$ from Section~\ref{SectPrelimitSystems}, set, for $(t,\phi,\theta)\in [0,T]\times \mathcal{X}\times \prbms{\mathcal{X}}$,
\begin{equation} \label{ExProofNashDrift}
\begin{split}
	\hat{b}(t,\phi,\theta)&\doteq b\left(t,\phi,\theta, u(t,\phi)\right)\\
	&= \begin{cases}
	u(t,\phi) + \bar{b}\left(t, \phi(t), \frac{\int w(\tilde{\phi}(t)) \mathbf{1}_{[0,\tau(\tilde{\phi}))}(t) \theta(d\tilde{\phi})}{\int \mathbf{1}_{[0,\tau(\tilde{\phi}))}(t) \theta(d\tilde{\phi})} \right)
	&\text{if } \theta(\tau > t) > 0,\\
	u(t,\phi) + \bar{b}\left(t,\phi(t),w(0)\right) &\text{if } \theta(\tau > t) = 0,
\end{cases}
\end{split}
\end{equation}
where $\tau(\phi)\doteq \inf\{t\geq 0: \phi(t)\notin O\}$. Then, thanks to hypotheses \Hypref{HypMeasBounded} and \Hypref{HypActionSpace}, $\hat{b}$ is Borel measurable, progressive, and bounded. Thus, the measurability and boundedness assumptions (M) and (B) of the Appendix hold for $\hat{b}$. Moreover, assumption (C) and (L) of the Appendix, the conditions of almost continuity and of partial Lipschitz continuity, respectively, are satisfied for this choice of $\hat{b}$; see Section~\ref{AppAssumptionsCL} in the Appendix.

Let $((\Omega,\mathcal{F},(\mathcal{F}_{t}),\Prb),W,X)$ be a solution of Eq.~\eqref{EqLimitDynamics} with flow of measures $\mathfrak{p}$, feedback strategy $u$, and initial distribution $\nu$; such a solution exists and determines a unique measure
\[
	\theta_{\ast}\doteq \Prb\circ X^{-1}.
\]
Notice that $\theta_{\ast}\in \mathcal{Q}_{\nu,K}$ whenever $K\geq \|\hat{b}\|_{\infty}$. In terms of $\theta_{\ast}$ and the canonical process $\hat{X}$, we can rewrite the costs associated with strategy $u$, flow of measures $\mathfrak{p}$, and initial distribution $\nu$ as
\begin{multline*}
	J(\nu,u;\mathfrak{p}) = \Mean_{\theta_{\ast}}\Biggl[ \int_{0}^{\tau} f\left(s,\hat{X}(s),\int_{\mathbb{R}^{d}} w(y)\mathfrak{p}(s,dy),u(s,\hat{X}) \right)ds \\
	+ F\left(\tau, \hat{X}(\tau) \right) \Biggr],
\end{multline*} 
where $\tau\doteq \tau^{\hat{X}}\wedge T$. Define $\hat{f}\!: [0,T]\times \mathcal{X}\times \prbms{\mathcal{X}} \rightarrow \mathbb{R}$ by
\[
	\hat{f}(t,\phi,\theta)\doteq \begin{cases}
		f\left(t,\phi(t), \frac{\int w(\tilde{\phi}(t)) \mathbf{1}_{[0,\tau(\tilde{\phi}))}(t) \theta(d\tilde{\phi})}{\int \mathbf{1}_{[0,\tau(\tilde{\phi}))}(t) \theta(d\tilde{\phi})},u(t,\phi) \right)
		&\text{if } \theta(\tau > t) > 0,\\
		f\left(t,\phi(t), w(0),u(t,\phi) \right) &\text{if } \theta(\tau > t) = 0,
	\end{cases}
\]
and a function $\hat{G}\!: \mathcal{X}\times \prbms{\mathcal{X}} \rightarrow \mathbb{R}$ by
\begin{equation} \label{ExProofNashCosts}
	\hat{G}(\phi,\theta)\doteq \int_{0}^{\tau(\phi)\wedge T} \hat{f}\left(s,\phi,\theta \right)ds + F\bigl(\tau(\phi)\wedge T, \phi(\tau(\phi)\wedge T) \bigr).
\end{equation}
Then
\[
	J(\nu,u;\mathfrak{p}) = \Mean_{\theta_{\ast}}\left[ \hat{G}(\hat{X},\theta_{\ast}) \right].
\]
The function $\hat{G}$ is bounded thanks to hypothesis \Hypref{HypMeasBounded}. Moreover, since $\theta_{\ast}\in \mathcal{Q}_{\nu,K}$ for $K\geq \|\hat{b}\|_{\infty}$ and thanks to \Hypref{HypCont},
\begin{equation} \label{EqProofGhat}
	\begin{split}
		\Theta_{\nu}\bigl(\phi\in \mathcal{X}: \exists\, (\phi_{n},\theta_{n})\subset \mathcal{X}\times \prbms{\mathcal{X}} &\text{ such that } \hat{G}(\phi_{n},\theta_{n}) \not\to \hat{G}(\phi,\theta_{\ast}) \\
		&\text{ while } (\phi_{n},\theta_{n})\to (\phi,\theta_{\ast}) \bigr) = 0,
	\end{split}
\end{equation}
To be more precise, one checks that $\hat{f}$ satisfies condition (C) as does $\hat{b}$; see Section~\ref{AppAssumptionsCL} of the Appendix. Eq.~\eqref{EqProofGhat} is then a consequence of the dominated convergence theorem, the continuity of $F$ and Part~\eqref{LemmaAppRegularity2Cont} of Lemma~\ref{LemmaAppRegularity2}. Eq.~\eqref{EqProofGhat} says that $\hat{G}$ is $\Theta_{\nu}\otimes \delta_{\theta_{\ast}}$-almost surely continuous. By Lemma~\ref{LemmaAppRegularity}, any $\theta\in \mathcal{Q}_{\nu,c}$, $c\geq 0$, is equivalent to $\Theta_{\nu}$. It follows that $\hat{G}$ is $\theta\otimes \delta_{\theta_{\ast}}$-almost surely continuous given any $\theta\in \mathcal{Q}_{\nu,c}$, any $c\geq 0$.

Recall that $(\nu,u,\mathfrak{p})$ is a feedback solution of the mean field game. Thanks to the conditional mean field property in Definition~\ref{DefMFGSolution} and the construction of $\hat{b}$ according to \eqref{ExProofNashDrift}, we have an alternative characterization of $\theta_{\ast}$, namely as a McKean-Vlasov solution of Eq.~\eqref{EqAppLimitDynamics}; see Definition~\ref{DefAppMVSolution} in Section~\ref{AppConvergence} of the Appendix. By Proposition~\ref{PropAppMVUnique} in Section~\ref{AppUniqueness} of the Appendix and since $\hat{b}$ satisfies assumption (L) there, we actually have that $\theta_{\ast}$ is the unique McKean-Vlasov solution of Eq.~\eqref{EqAppLimitDynamics} with initial distribution $\nu$.

\bigskip
\textbf{Step 2.}
For $N\in \mathbb{N}$, let $((\Omega_{N},\mathcal{F}_{N},(\mathcal{F}^{N}_{t}),\Prb_{N}),\boldsymbol{W}^{N},\boldsymbol{X}^{N})$ be a solution of Eq.~\eqref{EqPrelimitDynamics2} under strategy vector $\boldsymbol{u}^{N}$ with initial distribution $\nu_{N}$. Let $\mu^{N}$ denote the associated empirical measure on the path space $\mathcal{X}$. We are going to show that
\[
	\lim_{N\to\infty} J^{N}_{1}(\boldsymbol{u}^{N}) = J(\nu,u;\mathfrak{p}).
\]

To this end, for $N\in \mathbb{N}$, set
\[
	\tilde{b}_{N}(t,\boldsymbol{\phi},\theta)\doteq \hat{b}(t,\phi_{1},\theta),\quad (t,\boldsymbol{\phi},\theta)\in [0,T]\times \mathcal{X}^{N}\times \prbms{\mathcal{X}},
\]
where $\hat{b}$ is given by \eqref{ExProofNashDrift}. From Step~1 we know that $\hat{b}$ satisfies the continuity assumptions (C) and (L) of the Appendix as well as the measurability and boundedness assumptions (M) and (B), respectively. Assumptions (M) and (B) also apply to the coefficients $\tilde{b}_{N}$. Assumptions (ND) and (I) of the Appendix (non-degeneracy of the diffusion matrix and chaoticity of the initial distributions) hold by hypothesis. Moreover, we have that $((\Omega_{N},\mathcal{F}_{N},(\mathcal{F}^{N}_{t}),\Prb_{N}),\boldsymbol{W}^{N},\boldsymbol{X}^{N})$ is a solution of Eq.~\eqref{EqAppPrelimitDynamics} with initial distribution $\nu_{N}$. We can therefore apply Lemmata \ref{LemmaAppTightness} and \ref{LemmaAppLimitSolutions} to the sequence of empirical measures $(\mu_{N})_{N\in \mathbb{N}}$; in combination with Proposition~\ref{PropAppMVUnique}, this yields
\[
	\Prb_{N}\circ \left(\mu^{N}\right)^{-1} \stackrel{N\to\infty}{\longrightarrow} \delta_{\theta_{\ast}} \text{ in } \prbms{\prbms{\mathcal{X}}},
\]
where $\theta_{\ast}$ is the measure identified in Step 1 as the unique McKean-Vlasov solution of Eq.~\eqref{EqAppLimitDynamics} with initial distribution $\nu$.

Symmetry of the coefficients, chaoticity of the initial distributions, and the Tanaka-Sznitman theorem 
\citep[Proposition 2.2 in][]{sznitman89} now imply that
\begin{equation} \label{EqProofNashConvergence}
	\Prb_{N}\circ \left(X^{N}_{1},\mu^{N}\right)^{-1} \stackrel{N\to\infty}{\longrightarrow} \theta_{\ast}\otimes \delta_{\theta_{\ast}} \text{ in } \prbms{\mathcal{X}\times \prbms{\mathcal{X}}},
\end{equation}

Recalling \eqref{ExProofNashCosts} and since $u^{N}_{1}(t,\boldsymbol{\phi}) = u(t,\phi_{1})$, we can rewrite the costs associated with $\boldsymbol{u}^{N}$ as
\[
	J^{N}_{1}(\boldsymbol{u}^{N}) = \Mean_{\Prb_{N}}\left[ \hat{G}\left(X^{N}_{1},\mu^{N} \right) \right].
\]	 
By Step~1, we know that $\hat{G}$ is bounded and, since $\theta_{\ast}\in \mathcal{Q}_{\nu,K}$ for $K \geq \|\hat{b}\|_{\infty}$, also that $\hat{G}$ is $\theta_{\ast}\otimes \delta_{\theta_{\ast}}$-almost surely continuous. By the mapping theorem \citep[Theorem~5.1 in][p.\,30]{billingsley68}, convergence according to \eqref{EqProofNashConvergence} therefore implies that
\[
	J^{N}_{1}(\boldsymbol{u}^{N}) = \Mean_{\Prb_{N}}\left[ \hat{G}\left(X^{N}_{1},\mu^{N} \right) \right] \stackrel{N\to\infty}{\longrightarrow} \Mean_{\theta_{\ast}}\left[ \hat{G}(\hat{X},\theta_{\ast}) \right] = J(\nu,u;\mathfrak{p}).
\]

\bigskip
\textbf{Step 3.} For $N\in \mathbb{N}\setminus \{1\}$, choose $v^{N}_{1}\in \mathcal{U}_{N}$ such that
\[
	J^{N}_{1}\left([\boldsymbol{u}^{N,-1},v^{N}_{1}]\right) \leq \inf_{v\in \mathcal{U}_{N}} J^{N}_{1}\left([\boldsymbol{u}^{N,-1},v]\right) + \epsilon/2.	
\]
Let $((\tilde{\Omega}_{N},\tilde{\mathcal{F}}_{N},(\tilde{\mathcal{F}}^{N}_{t}),\tilde{\Prb}_{N}),\boldsymbol{\tilde{W}}^{N},\boldsymbol{\tilde{X}}^{N})$ be a solution of Eq.~\eqref{EqPrelimitDynamics2} under the strategy vector $[\boldsymbol{u}^{N,-1},v^{N}_{1}]$ with initial distribution $\nu_{N}$. Let $\tilde{\mu}^{N}$ denote the associated empirical measure on path space. We are going to show that
\[
	\liminf_{N\to \infty} J^{N}_{1}([\boldsymbol{u}^{N,-1},v^{N}_{1}]) \geq V(\nu;\mathfrak{p}).
\]

To this end, redefine $\tilde{b}_{N}$, for $N\in \mathbb{N}\setminus \{1\}$, by
\[
	\tilde{b}_{N}(t,\boldsymbol{\phi},\theta)\doteq b\left(t,\phi_{1},\theta,v^{N}_{1}(t,\boldsymbol{\phi})\right),\quad (t,\boldsymbol{\phi},\theta)\in [0,T]\times \mathcal{X}^{N}\times \prbms{\mathcal{X}}.
\]
Assumptions (M), (B), (ND), and (I) of the Appendix continue to hold. With the new definition of the coefficient $\tilde{b}_{N}$, $((\tilde{\Omega}_{N},\tilde{\mathcal{F}}_{N},(\tilde{\mathcal{F}}^{N}_{t}),\tilde{\Prb}_{N}),\boldsymbol{\tilde{W}}^{N},\boldsymbol{\tilde{X}}^{N})$ is a solution of Eq.~\eqref{EqAppPrelimitDynamics} with initial distribution $\nu_{N}$. As in Step 2, we can apply Lemmata \ref{LemmaAppTightness} and \ref{LemmaAppLimitSolutions} in combination with Proposition~\ref{PropAppMVUnique}, but now to the sequence of empirical measures $(\tilde{\mu}_{N})_{N\in \mathbb{N}}$. This yields
\begin{equation} \label{EqProofNashEMtildeConvergence}
	\tilde{\Prb}_{N}\circ \left(\tilde{\mu}^{N}\right)^{-1} \stackrel{N\to\infty}{\longrightarrow} \delta_{\theta_{\ast}} \text{ in } \prbms{\prbms{\mathcal{X}}},
\end{equation}
where $\theta_{\ast}$ is still the unique McKean-Vlasov solution of Eq.~\eqref{EqAppLimitDynamics} with initial distribution $\nu$ found in Steps 1 and 2.

Now, interpret the feedback controls of the first player as stochastic open-loop relaxed controls. To this end, let $\mathcal{R}\doteq \mathcal{R}_{\Gamma}$ be the set of $\Gamma$-valued deterministic relaxed controls over $[0,T]$; see Appendix~\ref{AppRelaxed}. For $N\in \mathbb{N}$, let $\tilde{\rho}^{N}_{1}$ be the stochastic relaxed control induced by the feedback strategy $v^{N}_{1}$, that is, $\tilde{\rho}^{N}_{1}$ is the $\mathcal{R}$-valued $(\tilde{\mathcal{F}}^{N}_{t})$-adapted random variable determined by
\[
	\tilde{\rho}^{N}_{1,\omega}\bigl(B\times I\bigr)\doteq \int_{I}\delta_{v^{N}_{1}(t,\tilde{X}^{N}_{1}(\cdot,\omega))}(B)dt,\quad B\in \Borel{\Gamma},\; I\in \Borel{[0,T]},\;\omega \in \tilde{\Omega}_{N}.
\]

Let $\mathcal{A}_{rel}$ denote the set of all quadruples $((\Omega,\mathcal{F},(\mathcal{F}_{t}),\Prb),\xi,\rho,W)$ such that $(\Omega,\mathcal{F},(\mathcal{F}_{t}),\Prb)$ is a filtered probability space satisfying the usual hypotheses, $\xi$ an $\mathcal{F}_{0}$-measurable random variable with values in $O$, $\rho$ an $\mathcal{R}$-valued $(\mathcal{F}_{t})$-adapted random variable, and $W$ a $d$-dimensional $(\mathcal{F}_{t})$-Wiener process. Using the convergence of $(\tilde{\mu}^{N})$ according to \eqref{EqProofNashEMtildeConvergence} and weak convergence arguments analogous to those of Section~\ref{AppConvergence} in the Appendix, one verifies that the sequence
\begin{equation} \label{EqProofTightness}
	\left( \tilde{\Prb}_{N}\circ \left(\tilde{X}^{N}_{1}, \tilde{\rho}^{N}_{1}, \tilde{W}^{N}_{1}, \tilde{\mu}^{N}\right)^{-1} \right)_{N\in \mathbb{N}} \text{ is tight in } \prbms{\mathcal{X}\times \mathcal{R}\times \mathcal{X}\times \prbms{\mathcal{X}} }
\end{equation}
and that its limit points are concentrated on measures $\Prb\circ (X,\rho,W)^{-1}\otimes \delta_{\theta_{\ast}}$ where
$((\Omega,\mathcal{F},(\mathcal{F}_{t}),\Prb),\xi,\rho,W)\in \mathcal{A}_{rel}$
for some $\mathcal{F}_{0}$-measurable $\mathbb{R}^{d}$-valued random variable $\xi$ with $\Prb\circ \xi^{-1} = \nu$, and $X$ is a continuous $(\mathcal{F}_{t})$-adapted process satisfying
\begin{equation} \label{EqLimitDynamicsOLrel}
\begin{split}
	X(t) &= \xi +  \int_{\Gamma\times [0,t]}\! \gamma\; \rho(d\gamma,ds) + \int_{0}^{t} \bar{b}\left(s,X(s),\int_{\mathbb{R}^{d}} w(y) \mathfrak{p}(s,dy)\right) ds \\
	&+ \sigma W(t),\quad t\in [0,T],
\end{split}
\end{equation}
with flow of measures $\mathfrak{p}$ given by
\[
	\mathfrak{p}(t)\doteq \theta_{\ast}\left( \hat{X}(t)\in \cdot \,|\, \tau^{\hat{X}} > t \right),\quad t\in [0,T].
\]
Notice that $X$ is uniquely determined (with $\Prb$-probability one) by Eq.~\eqref{EqLimitDynamicsOLrel}. The stochastic relaxed control in \eqref{EqLimitDynamicsOLrel} actually corresponds to an ordinary stochastic open-loop control, namely to the $\Gamma$-valued process
\begin{equation} \label{EqProofInducedControl}
	\alpha(t,\omega)\doteq \int_{\Gamma} \gamma \dot{\rho}_{s,\omega}(d\gamma),\quad (t,\omega)\in [0,T]\times \Omega.
\end{equation}
The process $X$ is therefore the unique solution of Eq.~\eqref{EqLimitDynamicsOL} corresponding to $((\Omega,\mathcal{F},(\mathcal{F}_{t}),\Prb),\xi,\alpha,W) \in \mathcal{A}$ and the flow of measures $\mathfrak{p}$ induced by $\theta_{\ast}$.

In analogy with Step~1, define $\bar{f}\!: [0,T]\times \mathcal{X}\times \prbms{\mathcal{X}}\times \Gamma \rightarrow \mathbb{R}$ by
\[
	\bar{f}(t,\phi,\theta,\gamma)\doteq \begin{cases}
		f\left(t,\phi(t), \frac{\int w(\tilde{\phi}(t)) \mathbf{1}_{[0,\tau(\tilde{\phi}))}(t) \theta(d\tilde{\phi})}{\int \mathbf{1}_{[0,\tau(\tilde{\phi}))}(t) \theta(d\tilde{\phi})},\gamma \right)
		&\text{if } \theta(\tau > t) > 0,\\
		f\left(t,\phi(t),w(0),u(t,\phi) \right) &\text{if } \theta(\tau > t) = 0,
	\end{cases}
\]
and a function $\bar{G}\!: \mathcal{X}\times \prbms{\mathcal{X}}\times \mathcal{R} \rightarrow \mathbb{R}$ by
\begin{equation} \label{ExProofOLCosts}
\begin{split}
	\bar{G}(\phi,\theta,r)\doteq \int_{\Gamma\times [0,T]} \mathbf{1}_{[0,\tau(\phi)\wedge T)}(s)\cdot \bar{f}\left(s,\phi,\theta,\gamma \right) r(d\gamma,ds) \\
	+ F\bigl(\tau(\phi)\wedge T, \phi(\tau(\phi)\wedge T) \bigr).
\end{split}
\end{equation}
Then
\[
	J^{N}_{1}([\boldsymbol{u}^{N,-1},v^{N}_{1}]) = \Mean_{\tilde{\Prb}_{N}}\left[ \bar{G}\left(\tilde{X}^{N}_{1},\tilde{\mu}^{N},\tilde{\rho}^{N}_{1} \right) \right].
\]
The function $\bar{G}$ is bounded thanks to hypothesis \Hypref{HypMeasBounded}. By \Hypref{HypCont} and arguments analogous to those of Step~1, one checks that
\begin{equation} \label{EqProofGbar}
	\begin{split}
		\Theta_{\nu}\bigl(&\phi\in \mathcal{X}: \exists\, r\in \mathcal{R},\; (\phi_{n},\theta_{n},r_{n})\subset \mathcal{X}\times \prbms{\mathcal{X}}\times \mathcal{R} \text{ such that}\\
		&\hat{G}(\phi_{n},\theta_{n},r_{n}) \not\to \hat{G}(\phi,\theta_{\ast},r)
		\text{ while } (\phi_{n},\theta_{n},r_{n})\to (\phi,\theta_{\ast},r) \bigr) = 0,
	\end{split}
\end{equation}
By Lemma~\ref{LemmaAppRegularity}, any $\theta\in \mathcal{Q}_{\nu,c}$, $c\geq 0$, is equivalent to $\Theta_{\nu}$. In view of \eqref{EqProofGbar}, it follows that $\bar{G}$ is $Q$-almost surely continuous given any $Q\in \prbms{\mathcal{X}\times \prbms{\mathcal{X}}\times \mathcal{R}}$ such that $Q\circ \pi_{\mathcal{X}\times \prbms{\mathcal{X}}}^{-1} = \theta\otimes \delta_{\theta_{\ast}}$ for some $\theta\in \mathcal{Q}_{\nu,c}$, some $c\geq 0$, where $\pi_{\mathcal{X}\times \prbms{\mathcal{X}}}$ denotes the projection of $\mathcal{X}\times \prbms{\mathcal{X}}\times \mathcal{R}$ onto its first two components. Using the mapping theorem of weak convergence, we conclude that
\[
	\liminf_{N\to \infty} J^{N}_{1}([\boldsymbol{u}^{N,-1},v^{N}_{1}]) \geq \inf_{((\Omega,\mathcal{F},(\mathcal{F}_{t}),\Prb),\xi,\rho,W) \in \mathcal{A}_{rel} : \Prb\circ \xi^{-1} = \nu} \\
		\Mean\left[ \bar{G}\left( X,\theta_{\ast},\rho \right) \right].
\]

It remains to check that
\[
	\inf_{((\Omega,\mathcal{F},(\mathcal{F}_{t}),\Prb),\xi,\rho,W) \in \mathcal{A}_{rel} : \Prb\circ \xi^{-1} = \nu} \\
			\Mean\left[ \bar{G}\left( X,\theta_{\ast},\rho \right) \right] = V(\nu;\mathfrak{p}).
\]
Inequality ``$\geq$'' holds by definition of $\bar{G}$ and because $\mathcal{A} \subseteq \mathcal{A}_{rel}$ if one identifies ordinary stochastic open-loop controls with the induced stochastic relaxed controls. The opposite inequality is a consequence of what is called ``chattering lemma'', i.e., the fact that relaxed controls can be approximated by ordinary controls; see, for instance, Theorem~3.5.2 in \citet[p.\,59]{kushner90} or Section~4 in \citet{elkarouietalii87}. In our situation, the stochastic relaxed control in Eq.~\eqref{EqLimitDynamicsOLrel} corresponds to an ordinary stochastic open-loop control. If the running costs $f$ are additive convex in the control, then inequality ``$\leq$'' can be verified more directly by using \eqref{EqProofInducedControl} and Jensen's inequality.

\bigskip
\textbf{Step 4.} For every $N\in \mathbb{N}\setminus \{1\}$,
\[
\begin{split}
	& J^{N}_{1}(\boldsymbol{u}^{N}) - \inf_{v\in \mathcal{U}_{N}} J^{N}_{1}([\boldsymbol{u}^{N,-1},v]) \\
	&\leq J^{N}_{1}(\boldsymbol{u}^{N}) - J(\nu,u;\mathfrak{p}) + J(\nu,u;\mathfrak{p}) -  J^{N}_{1}([\boldsymbol{u}^{N,-1},v^{N}_{1}]) + \epsilon/2.
\end{split}
\]
By Steps 2 and 3, there exists $N_{0} = N_{0}(\epsilon)$ such that for all $N \geq N_{0}$,
\[
	J^{N}_{1}(\boldsymbol{u}^{N}) - J(\nu,u;\mathfrak{p}) + V(\nu;\mathfrak{p}) -  J^{N}_{1}([\boldsymbol{u}^{N,-1},v^{N}_{1}]) \leq \epsilon/2.
\]
Since $(\nu,u;\mathfrak{p})$ is a solution of the mean field game, $J(\nu,u;\mathfrak{p}) = V(\nu;\mathfrak{p})$. It follows that for all $N \geq N_{0}$,
\[
	J^{N}_{1}(\boldsymbol{u}^{N}) - \inf_{v\in \mathcal{U}_{N}} J^{N}_{1}([\boldsymbol{u}^{N,-1},v]) \leq \epsilon,
\]
which establishes \eqref{EqProofNash}.
\end{proof}


\section{Existence of solutions} \label{SectExistence}

Throughout this section, hypotheses \Hyprefall\ as well as \hypref{HypAddND} are in force. Under these and some additional assumptions, we show that, given any initial distribution, a feedback solution of the mean field game exists in the sense of Definition~\ref{DefMFGSolution}; see Theorem~\ref{nash-existence} in Subsection~\ref{SectExistMFG}. To prove Theorem~\ref{nash-existence}, we use a \mbox{BSDE} to approach to the control problem for fixed flow of measures (Subsection~\ref{SectExistOptControl}) and a fixed point argument, both in the spirit of \citet{carmonalacker15}.

In Subsection~\ref{SectExistContinuity}, we provide conditions that guarantee the existence of a feedback solution to the mean field game with continuous Markovian feedback strategy; see Corollary~\ref{crll:nash-cont}. Such a solution satisfies, in particular, the requirements of Theorem~\ref{ThApproximateNash}.

As above, let $\nu \in \prbms{\mathbb{R}^{d}}$ with support in $O$ denote the initial distribution. Since $\nu$ will be fixed, the dependence on $\nu$ will often be omitted in the notation.

\subsection{Optimal controls for a given flow of measures} \label{SectExistOptControl}

Let $\mathfrak p \in \mathcal M$ be a given flow of measures, and define the corresponding flow $m_{\mathfrak p}$ of means by
\begin{equation} \label{def-m-p}
	m_{\mathfrak p}(t)\doteq \int_{\mathbb{R}^{d}} w(y) \mathfrak p(t,dy),\quad t\in [0,T].
\end{equation}

Recall that $\hat{X}$ denotes the canonical process on the path space $\mathcal X \doteq \mathbf C ([0,T];\mathbb R^d)$ (equipped with the sup-norm topology). Also recall \eqref{ExRefMeasure}, the definition of $\Theta_{\nu}$. Set $\Omega\doteq \mathcal{X}$, $X \doteq \hat{X}$, let $\mathcal{F}$ be the $\Theta_{\nu}$-completion of the Borel $\sigma$-algebra $\Borel{\mathcal{X}}$, and set $\mathbf{P}\doteq \Theta_{\nu}$ (extended to all null sets). Let $(\mathcal F_t)$ be the $\mathbf P$-augmentation of the filtration generated by $X$. Lastly, set
\begin{align*}
	&\xi(\omega)\doteq X(0,\omega), & & W(t,\omega)\doteq \sigma^{-1}\left( X(t,\omega) - \xi(\omega)\right),& &(t,\omega)\in [0,T]\times \Omega,&
\end{align*}
which is well-defined thanks to \hypref{HypAddND}. Then $(\Omega,\mathcal{F},(\mathcal F_t),\mathbf{P})$ is a filtered probability space satisfying the usual assumptions, $W$ a $d$-dimensional $(\mathcal F_t)$-Wiener process, and $\xi$ an $\mathcal{F}_{0}$-measurable random variable with law $\nu$ independent of $W$. Moreover, by construction,
\begin{equation} \label{dynX}
	X(t) = \xi + \sigma W(t) , \quad t\in [0,T],
\end{equation}
and $\mathbf P \circ X^{-1} = \Theta_{\nu}$. 

Let $\mathcal{U}$ denote the set of all $\Gamma$-valued $(\mathcal F_t)$-progressively measurable processes. The filtration $(\mathcal F_t)$ is the $\mathbf{P}$-augmentation of the filtration generated by $X$. Any feedback strategy $u\in \mathcal{U}_{fb}$ therefore induces an element of $\mathcal{U}$ through $\mathrm{u}_{t}\doteq u(t,X)$, $t\in [0,T]$. On the other hand, given $\mathrm{u}\in \mathcal{U}$, we can find a progressive functional $u\in \mathcal{U}_{1}$ such that $u(\cdot,X) = \mathrm{u}_{\cdot}$ $\mathrm{Leb}_{T}\otimes \mathbf P$-almost surely, where $\mathrm{Leb}_{T}$ denotes Lebesgue measure on $[0,T]$. Lastly, by Proposition~\ref{PropLimitExistUnique} and \hypref{HypAddND}, $\mathcal{U}_{fb} = \mathcal{U}_{1}$. We can therefore identify elements of $\mathcal{U}$ with those of $\mathcal{U}_{fb}$, and vice versa.

Let $\mathrm{u} \in \mathcal{U}$. By Girsanov's theorem and the boundedness of $\bar b$, there exists a probability measure $\mathbf P^{\mathfrak p,\mathrm u} \sim \mathbf P$ and a $\mathbf P^{\mathfrak p,\mathrm u}$-Wiener process $W^{\mathfrak p,\mathrm u}$ such that
\[
	dX(t) = \left(\mathrm u_t + \bar b\left(t,X(t), m_{\mathfrak p}(t) \right) \right) dt + \sigma dW^{\mathfrak p,\mathrm u}(t).
\]
The costs associated with $\mathrm{u}$ are then given by
\[
	J^{\mathfrak p} (\mathrm u) \doteq \mathbf E^{\mathfrak p,\mathrm u} \left[ \int_0 ^{\tau} f\left(t,X(t), m_{\mathfrak p}(t) ,\mathrm u_t\right)dt + F( \tau, X_{\tau}) \right].
\]
If $u\in \mathcal{U}_{fb} = \mathcal{U}_{1}$ is such that $u(\cdot,X) = \mathrm{u}_{\cdot}$ $\mathrm{Leb}_{T}\otimes \mathbf P$-almost surely, then (cf.\ Section~\ref{SectLimitSystems})
\[
	J^{\mathfrak p} (\mathrm u) = J(\nu,u;\mathfrak{p}).
\]
Since elements of $\mathcal{U}$ can be identified with those of $\mathcal{U}_{fb}$, and vice versa, we have
\[
	V^{\mathfrak p} \doteq \inf_{\mathrm u \in \mathcal U} J^{\mathfrak p} (\mathrm u) = \inf_{u \in \mathcal{U}_{fb}} J(\nu,u;\mathfrak{p}).
\]
In view of Remark~\ref{RemValueFnct}, if $\mathfrak{p}$ is such that $t\mapsto m_{\mathfrak{p}}(t)$ is continuous, then
\[
	V^{\mathfrak p} = V(\nu;\mathfrak{p}).
\]

Define the Hamiltonian $h$ and the minimized Hamiltonian $H$ as follows
\begin{align*}
	h(t,x,m,z,\gamma) &\doteq f(t,x,m,\gamma) + z \cdot \sigma^{-1} b(t,x,m,\gamma),\\
	H(t,x,m,z) &\doteq \inf_{\gamma \in \Gamma} h(t,x,m,z,\gamma),
\end{align*}
where $b(t,x,m,\gamma)\doteq \gamma + \bar b(t,x,m)$.
The set in which the infimum is attained is denoted by
\[
	A(t,x,m,z) \doteq \left\{ \gamma \in \Gamma : h(t,x,m,z,\gamma) = H(t,x,m,z) \right\}.
\]
This set is nonempty since $\Gamma$ is compact and the function $h$ is continuous in $\gamma$ thanks to \Hypref{HypActionSpace} and \Hypref{HypCont}, respectively.

Now, consider the following \mbox{BSDE} with random terminal date $\tau \doteq \tau^X \wedge T$: 
\begin{equation} \label{bsde-H}
\begin{split}
	Y^{\mathfrak p}(t) =  F(\tau, X(\tau)) &+ \int_{t}^{T} H\left(s,X(s), m_{\mathfrak p}(s), Z^{\mathfrak p}(s)\right) \mathbf 1_{(s<\tau)} ds \\
	&- \int_{t}^{T} Z^{\mathfrak p}(s) dW(s),
\end{split}
\end{equation}
where $X$ follows the forward dynamics \eqref{dynX} and $W$ is a Wiener process under $\mathbf P$ as before. As in \cite{darlingpardoux97}, we adopt the convention that for any solution to the \mbox{BSDE} above:
\begin{align*}
	& \mathbf 1_{(s > \tau)} Y(s) = F(\tau, X(\tau)),& &\mathbf 1_{(s > \tau)} Z(s) =0,& &\mathbf 1_{(s > \tau)} H(s,x,m,z) = 0.&
\end{align*} 
Observe that since the driver is Lipschitz in $z$ and does not depend on $y$, assumptions (4), (5), (7) and (21) in \cite{darlingpardoux97} are satisfied. Moreover, their assumption (25) is also fulfilled due to the fact that the terminal time $\tau$, the terminal value $F$, the drift $\bar b$ and the running cost $f$ are all bounded; cf.\ \Hypref{HypMeasBounded}. Hence Theorem~3.4 of \citet{darlingpardoux97} applies, yielding that the \mbox{BSDE} \eqref{bsde-H} has a unique solution $(Y^{\mathfrak p},Z^{\mathfrak p})$ in the space of all progressively measurable processes $K \doteq (Y,Z)$ with values in $\mathbb R^d \times \mathbb R^{d\times d}$ such that
\[
	\mathbf E \left[ \int_0 ^{\tau \wedge T} |K(t)|^2 dt \right] < \infty .
\]
For each $\mathrm u \in \mathcal U$ we can proceed in a similar way to get existence and uniqueness (in the same space of processes as above) of the solution $(Y^{\mathfrak p,\mathrm u},Z^{\mathfrak p ,\mathrm u})$ to the following \mbox{BSDE}
\begin{equation} \label{bsde-h}
\begin{split}
	Y^{\mathfrak p ,\mathrm u}(t) =  F(\tau , X(\tau)) &+ \int_{t} ^{T} h \left(s,X(s), m_{\mathfrak p}(s) , Z^{\mathfrak p ,\mathrm u}(s), \mathrm u_s \right) \mathbf 1_{(s<\tau)} ds \\ 
	&- \int_{t} ^{T} Z^{\mathfrak p,\mathrm u}(s) dW(s).
\end{split}
\end{equation}
Changing measure from $\mathbf P$ to $\mathbf P^{\mathfrak p,\mathrm u}$ we obtain
\[
\begin{split}
	Y^{\mathfrak p ,\mathrm u}(t) = F(\tau, X(\tau)) &+ \int_{t} ^{T} f \left(s,X(s), m_{\mathfrak p}(s) , \mathrm u_s \right) \mathbf 1_{(s < \tau)} ds \\
	&- \int_{t} ^{T} Z^{\mathfrak p ,\mathrm u}(s) dW^{\mathfrak p ,\mathrm u}(s).
\end{split}
\]
Hence, taking conditional expectation with respect to $\mathbf P^{\mathfrak p ,\mathrm u}$ and $\mathcal F_t$ on both sides, we have
\[
	Y^{\mathfrak p ,\mathrm u}(t) = \mathbf E^{\mathfrak p ,\mathrm u} \left[ F(\tau, X(\tau)) + \int_{t}^{T} f \left (s,X(s), m_{\mathfrak p}(s), \mathrm u_s \right) \mathbf 1_{(s<\tau)} ds \mid \mathcal F_t \right],
\]
which implies $\mathbf E [Y^{\mathfrak p,\mathrm u}(0)] = \mathbf E^{\mathfrak p ,\mathrm u} [Y^{\mathfrak p ,\mathrm u}(0)] = J^{\mathfrak p} (\mathrm u)$. The first equality is due to the fact that the law of $\xi$ is the same under both probability measures.

Since $H \leq h$ and the forward state variable $X$ is the same for both \mbox{BSDEs} (\ref{bsde-H}) and (\ref{bsde-h}) (so that they both have the same random terminal time), we can apply \citet[Corollary 4.4.2]{darlingpardoux97} yielding $Y_t ^{\mathfrak p} \le Y_t ^{\mathfrak p ,\mathrm u}$ a.s. for all $t\in [0,T]$. This implies that $\mathbf E [Y^{\mathfrak p}(0)] \leq \mathbf E [Y^{\mathfrak p,\mathrm u}(0)] = J^{\mathfrak p} (\mathrm u)$ for all $\mathrm u \in \mathcal U$, and therefore $\mathbf E [Y^{\mathfrak p}(0)] \leq V^{\mathfrak p}$.

By a standard measurable selection argument, there exists a measurable function $\hat{u}\!: [0,T]\times \mathbb{R}^{d}\times \mathbb{R}^{d_0}\times \mathbb{R}^{d\times d} \rightarrow \Gamma$ such that
\[
	\hat{\mathrm u} (t,x,m,z) \in A(t,x,m,z) \text{ for all } (t,x,m,z).
\]
Hence, letting
\begin{equation} \label{opt-control}
	\mathrm u_t ^{\mathfrak p} \doteq \hat{\mathrm u}\left(t,X(t), m_{\mathfrak p}(t) ,Z^{\mathfrak p}(t) \right) ,
\end{equation} 
the uniqueness of the solution of \mbox{BSDEs} with random terminal time as in \citet[Theorem 3.4]{darlingpardoux97} gives $Y^{\mathfrak p}(t) = Y^{\mathfrak p,\mathrm u^{\mathfrak p}}(t)$, which in turn implies $V^{\mathfrak p} = J^{\mathfrak p} (\mathrm u^{\mathfrak p})$.

\subsection{Existence of solutions of the mean field game} \label{SectExistMFG}

We are going to apply the Brouwer-Schauder-Tychonoff fixed point theorem \citep[for instance,][17.56 Corollary]{aliprantisborder06}, which we recall for the reader's convenience:

\begin{thrm}[Brouwer-Schauder-Tychonoff] \label{thm:fixed}
Let $\mathcal K$ be a nonempty compact convex subset of a locally convex Hausdorff space, and let $\Psi\!: \mathcal K \to \mathcal K$ be a continuous function. Then the set of fixed points of $\Psi$ is compact and nonempty.
\end{thrm}

We need to identify a suitable space $\mathcal K$ and a good function $\Psi$ in order to apply the fixed theorem above to our setting and obtain the existence of a solution to the mean field game. 

In addition to \Hyprefall\ and \hypref{HypAddND}, let us make the following assumptions:

\begin{ass}	\label{ass2}
\
\begin{enumerate}
	\item[(i)] The coefficients $\bar{b}$ and $f$ are continuous (jointly in all their variables).
	
	\item[(ii)] The set $A(t,x,m,z)$ is a singleton for all $(t,x,m,z) \in [0,T] \times \mathbb R^d \times \mathbb{R}^{d_0}\times \mathbb R^{d\times d}$.
\end{enumerate}
\end{ass}

By Assumption~\ref{ass2}(ii), the function $\hat{u}$ appearing in \eqref{opt-control} is uniquely determined. Moreover, $\hat{u}$ is continuous:

\begin{lemma} \label{lemma:uhat}
	The function $\hat{u}$ is continuous (jointly in all its variables).
\end{lemma}

\begin{proof}
Assumption \ref{ass2}(i) implies that the (pre-)Hamiltonian $h$ is (jointly) continuous. The continuity of $\hat{u}$ now follows from Berge's Maximum Theorem \citep[for instance,][Theorem 17.31]{aliprantisborder06}, Assumption~\ref{ass2}(ii) and the compactness of $\Gamma$ according to~\Hypref{HypActionSpace}.
\end{proof}

We will apply the fixed point theorem to a restriction of the mapping
\begin{equation} \label{Psi}
	\Psi\!: \mathcal P(\mathcal X) \ni \mu \mapsto \mathbf P^{\mu} \circ X^{-1} \in \mathcal P(\mathcal X)
\end{equation}
where $\mathbf P^{\mu} \doteq \mathbf P^{\mathfrak p^\mu, \hat {\mathrm u}^\mu}$ according to Subsection~\ref{SectExistOptControl} with $\mathfrak p^\mu$ and $\hat u^\mu$ defined as follows:
\begin{subequations}
\begin{align}
\label{def-p^mu}
	\mathfrak p^\mu (t, \cdot) &\doteq \mu \left( X(t) \in \cdot \mid \tau^X > t \right), && \\
\label{def-uhat^mu}
	\hat {\mathrm u}_t ^\mu &\doteq \mathrm u_t ^{\mathfrak p^\mu} = \hat{\mathrm u} \left(t,X(t), \int w(y) \mathfrak p^\mu (t,dy) , Z^{\mathfrak p^\mu}(t)\right), &t\in [0,T].&
\end{align}
\end{subequations}
Thanks to Assumption \ref{ass2}(ii), the optimal control process $\hat{\mathrm u}^\mu$ is unique $\mathrm{Leb}_{T}\otimes \Prb$-almost surely. Let $\mathbf E_\mu$ denote expectation with respect to $\mu$ (not $\mathbf P^\mu$!).

Now, we need to specify the subspace of $\prbms{\mathcal X}$ with the right properties for applying the fixed point theorem. Recall the definition of $\mathcal{Q}_{\nu,c}$ with $c \geq 0$ from Section~\ref{SectApproxNash}. Choose $K > 0$ such that $K \geq \|b\|_{\infty}$, and set
\[
	\mathcal{K} \doteq \mathcal{Q}_{\nu,K}.
\]
Then $\mathcal{K}$ is non-empty and convex, and it is compact with respect to the topology inherited from $\prbms{\mathcal X}$ (i.e., the topology of weak convergence of measures), see Lemma~\ref{LemmaAppRegularityCompact} in the Appendix. Moreover, $\mathcal K$ can be seen as a subset of the dual space $\mathbf C_b (\mathcal X)^*$, where $\mathbf C_b(\mathcal X)$ is the space of all continuous bounded functions on $\mathcal X$. The dual space $\mathbf C_b(\mathcal X)^*$, equipped with the weak* topology $\sigma(\mathbf C_b (\mathcal X)^*,  \mathbf C_b(\mathcal X))$, is a locally convex topological vector space, inducing the weak convergence on $\mathcal P(\mathcal X) \supset \mathcal K$.

If $\mu\in \mathcal{K}$, then the flow of conditional means induced by $\mu$ is continuous: 
\begin{lemma} \label{lemma:flowm}
	Let $\mu \in \mathcal K$. Then the mapping $t\mapsto m^{\mu}(t)\doteq \int w(y)\mathfrak{p}^{\mu}(t,dy)$ is continuous, where $\mathfrak p^\mu $ is the flow of conditional measures induced by $\mu$ according to \eqref{def-p^mu}.
\end{lemma}

\begin{proof}
By Lemma~\ref{LemmaAppRegularity}, $\inf_{t\in [0,T]} \mu (\tau^X > t) > 0$. Therefore and by construction, for every $t\in [0,T]$,
\[
	m^{\mu}(t) = \frac{1}{\mu (\tau^X > t)} \mathbf E_{\mu} \left[ w(X(t)) \mathbf 1_{[0,\tau^X)} (t)\right].
\]
Let $t\in [0,T]$, and let $(t_{n}) \subset [0,T]$ be such that $t_{n}\to t$ as $n\to \infty$. For $\omega\in \Omega$, the mapping $s\mapsto \mathbf{1}_{[0,\tau^{X}(\omega))}(s)$ is discontinuous only at $s = \tau^{X}(\omega)$. By part~\eqref{LemmaAppRegularity2StochCont} of Lemma~\ref{LemmaAppRegularity2}, $\mu\left( \tau^{X} = t \right) = 0$. Since $w$ is bounded and continuous and $X$ has continuous trajectories, we have with $\mu$-probability one,
\begin{align*}
	& \mathbf{1}_{[0,\tau^{X})}(t_{n}) \stackrel{n\to\infty}{\longrightarrow} \mathbf{1}_{[0,\tau^{X})}(t), & & w\left(X(t_{n})\right)\mathbf{1}_{[0,\tau^{X})}(t_{n}) \stackrel{n\to\infty}{\longrightarrow} w\left(X(t)\right)\mathbf{1}_{[0,\tau^{X})}(t). & 
\end{align*}
The dominated convergence theorem now implies $\lim_{n\to\infty} m(t_{n}) = m(t)$.
\end{proof}

In order to apply Theorem \ref{thm:fixed} above, we have to show that $\Psi\!: \mathcal K \to \mathcal K$ is (sequentially) continuous. Let $(\mu^n)$ be a sequence in $\mathcal K$ converging to some measure $\mu \equiv \mu^{\infty} \in \mathcal K$. We want to show that $\Psi(\mu^n) \to \Psi(\mu)$ in $\mathcal{K} \subset \prbms{\mathcal{X}}$. To ease the notation, we set
\begin{align*}
	& \mathfrak p^n \doteq \mathfrak p^{\mu^n},& & Z^n \doteq Z^{\mathfrak p^n},& & \mathrm u_t ^n \doteq \mathrm u_t ^{\mathfrak p^n},& & 1\le n \le \infty.&
\end{align*}
We proceed as in \citet[Proof of Theorem 3.5]{carmonalacker15} and show that
\begin{equation} \label{REconv}
	\mathcal H(\mathbf P^{\mu} \mid \mathbf P^{\mu^n}) \stackrel{n\to \infty}{\longrightarrow} 0,
\end{equation}
where $\mathcal H(\cdot \mid \cdot)$ denotes the relative entropy:
\[
	\mathcal H(\tilde{\mu} \mid \bar{\mu}) \doteq \begin{cases}
	\int \log \left(\frac{d\tilde{\mu}}{d\bar{\mu}}\right) d\tilde{\mu} &\text{if } \tilde{\mu} \text{ absolutely continuous w.r.t. } \bar{\mu} \\
	\infty &\text{otherwise}.
	\end{cases}
\]
Observe that \eqref{REconv} implies $\Psi(\mu^n) \to \Psi(\mu)$ in the topology of weak convergence as well as in the topology of convergence in total variation.

In our situation, denoting by $\mathcal{E}(\cdot)$ the stochastic exponential of a local martingale, we have 
\[
\begin{split}
	\frac{d\mathbf P^{\mu^n}}{d\mathbf P^{\mu}} = \mathcal E \Biggl( \int_{0}^{\cdot} &\sigma^{-1}\left(b(t,X(t), m^{n}(t) , \mathrm u^{n}_t) - b(t,X(t), m^\infty (t) , \mathrm u^{\infty}_t)\right) \mathbf{1}_{[0,\tau^{X})}(t)\\
	& dW^{\mathfrak p^\infty}(t) \Biggr)(T),
\end{split}
\]
where $m^{n}$ denotes the flow of means induced by $\mathfrak{p}^{n}$ according to \eqref{def-m-p}:
\[
	m^{n} (t)\doteq \int w(y) \mathfrak p^n (t,dy), \quad 1\le n\le \infty.
\]
Since $\sigma^{-1}b$ is bounded, we obtain
\begin{align*}
	&\mathcal H(\mathbf P^{\mu} \mid \mathbf P^{\mu^n}) = - \mathbf E^{\mathbf P^\mu} \left[ \log \frac{d\mathbf P^{\mu^n}}{d\mathbf P^{\mu}} \right] \\
	&= \frac{1}{2} \mathbf E^{\mathbf P^\mu} \left[  \int_0 ^{\tau} \left | \sigma^{-1}b \left (t,X(t), m^{n}(t) , \mathrm u^{n}_t \right) - \sigma^{-1}b \left (t,X(t), m^{\infty}(t) , \mathrm u^{\infty}_t \right) \right |^2 dt \right].
\end{align*}
By \Hypref{HypCont}, the map $b(t,x,\cdot,\cdot)$ is continuous for all $(t,x)$. Hence, we can conclude by an application of the bounded convergence theorem provided we show that
\begin{equation} \label{conv-m-u}
	( m^{n} , \mathrm u^{n}) \stackrel{n\to \infty}{\longrightarrow} ( m^{\infty}, \mathrm u^{\infty}) \text{ in }\mathrm{Leb}_{T} \otimes \mathbf P^{\mu}\text{-measure}.
\end{equation}

First, we show that 
\begin{equation} \label{conv-m}
	m^{n} \stackrel{n\to \infty}{\longrightarrow} m^{\infty} \text{ in }\mathrm{Leb}_{T}\text{-measure}.
\end{equation}
By Lemma~\ref{LemmaAppRegularity}, $\inf_{n\in \mathbb{N}\cup\{\infty\}} \inf_{t\in [0,T]} \mu^{n}\left(\tau^X >t\right) > 0$. Hence, by construction, for every $t\in [0,T]$,
\[
	m^{n} (t) = \frac{1}{\mu^n (\tau^X > t)} \mathbf E_{\mu^n} \left[ w(X(t)) \mathbf 1_{[0,\tau^X)} (t)\right],\quad n\in \mathbb{N}.
\]
The mapping $\mathbf{1}_{[0,\tau^{X})}(t)$ is $\mu$-almost surely continuous for each fixed $t$ by part~\eqref{LemmaAppRegularity2IndCont} of Lemma~\ref{LemmaAppRegularity2}. By hypothesis, $w$ is bounded and continuous. The mapping theorem \citep[Theorem~5.1 in][p.\,30]{billingsley68} therefore implies that
\[
	\mathbf E_{\mu^n} \left[ w(X(t)) \mathbf 1_{[0,\tau^X)} (t)\right] \to \mathbf E_{\mu} \left[ w(X(t)) \mathbf 1_{[0,\tau^X)} (t)\right],
\]
as well as $\mu^n (\tau^X >t) \to \mu(\tau^X >t)$. This gives that $m^{n} (t) \to m^{\infty} (t)$ for all $t$, hence $m^{n}  \to m^{\infty} $ pointwise and in $\mathrm{Leb}_{T}$-measure.

Next, let us check that
\begin{equation} \label{conv-Z}
	\lim_{n \to \infty} \mathbf E \left[ \int_0 ^T \left| Z^{n}(t) - Z^{\infty}(t)  \right|^2 dt \right] =0.
\end{equation}
To this end, we use a stability result for \mbox{BSDEs} with random terminal time as in \citet[Theorem 2.4]{briandhu98}. According to that result, in order to get the convergence \eqref{conv-Z}, it suffices to show that
\[
	\mathbf E\left[  \int_0 ^{\tau} | H(s,X(s), m^{n} (s), Z^\infty (s)) - H(s,X(s), m^{\infty} (s), Z^\infty (s))|^2 ds \right] \stackrel{n\to \infty}{\longrightarrow} 0.
\]
Since the functions $f$ and $b$ in the definition of the minimized Hamiltonian $H$ are bounded and the controls take values in a compact set, we easily have
\[
	\left| H(s,X(s), m^{n} (s), Z^\infty (s)) - H(s,X(s), m^{\infty} (s), Z^\infty (s))\right|^2 \leq c_1 + c_2 |Z^\infty (s) |^2,
\]
where $c_1$, $c_2$ are some positive constants. Moreover, $\mathbf E [\int_0 ^{\tau} |Z^\infty (t) |^2 dt] < \infty$ (cf. \citet[Theorem 3.4]{darlingpardoux97}). We can therefore apply the dominated convergence theorem to obtain \eqref{conv-Z}.
The convergence in $\mathrm{Leb}_{T} \times \mathbf P^\mu$-measure follows from the equivalence $\mathbf P^\mu \sim \mathbf P$.

In view of \eqref{conv-m}, in order to establish \eqref{conv-m-u}, it remains to show that $\mathrm u^{n}_t \to \mathrm u^{\infty}_t$ $\mathbf P^\mu$-almost surely for all $t\in [0,T]$. Notice that for all $n\in \mathbb{N}\cup \{\infty\}$
\[
	\hat {\mathrm u}_t ^{n} = \hat{\mathrm u}\left(t,X(t), m^{n} (t),Z^{n}(t)\right).
\]
By Lemma~\ref{lemma:uhat}, $\hat{\mathrm u}$ is jointly continuous in all its variables. Equations \eqref{conv-m} and \eqref{conv-Z} together imply that $(m^{n},Z^{n}) \to (m^{\infty},Z^{\infty})$ in $\mathrm{Leb}_{T} \times \mathbf P^\mu$-measure. It follows that $\hat {\mathrm u}^{n} \to \hat {\mathrm u}^{\infty}$ in $\mathrm{Leb}_{T} \times \mathbf P^\mu$-measure as well.

Now, we can state and prove the main result of this section.

\begin{thrm} \label{nash-existence}
	Under \Hyprefall, \hypref{HypAddND}, and Assumptions~\ref{ass2}, there exists a feedback solution of the mean field game.
\end{thrm}

\begin{proof}
In view of the discussion above, we can apply Theorem~\ref{thm:fixed}
(the Brouwer-Schauder-Tychonoff fixed point theorem) to the map $\Psi\!: \mathcal K \to \mathcal K$. Thus, there exists $\mu\in \mathcal{K}$ such that $\Psi(\mu) = \mu$. Let $\mathfrak{p}\doteq \mathfrak{p}^{\mu}$ be the flow of conditional measures induced by the fixed point $\mu$ according to \eqref{def-p^mu}. Let $m_{\mathfrak{p}}$ be the flow of (conditional) means induced by $\mathfrak{p}$ according to \eqref{def-m-p}, and let $(Y^{\mathfrak{p}}, Z^{\mathfrak{p}})$ be the unique solution of the \mbox{BSDE} \eqref{bsde-H}. Notice that $Z^{\mathfrak{p}}$ is progressively measurable with respect to $(\mathcal{F}_{t})$, the $\mathbf{P}$-augmentation of the filtration generated by the canonical process $X = \hat{X}$. Hence we can find $u\in \mathcal{U}_{1}$ such that
\[
	u(\cdot,X) = \hat{u}\left(\cdot, X(\cdot), m_{\mathfrak{p}}(\cdot), Z^{\mathfrak{p}}(\cdot)\right)\quad \mathrm{Leb}_{T}\otimes \mathbf P\text{-almost surely}.
\]
Recalling the discussion of Subsection~\ref{SectExistOptControl} and the fact that in our situation $\mathcal{U}_{1} = \mathcal{U}_{fb}$, we see that 
\[
	J(\nu,u;\mathfrak{p}) = V^{\mathfrak{p}}.
\]
The flow of means $m_{\mathfrak{p}}$ is continuous thanks to Lemma~\ref{lemma:flowm}. This implies, by Remark~\ref{RemValueFnct}, that $V^{\mathfrak{p}} = V(\nu;\mathfrak{p})$. The triple $(\nu,u,\mathfrak{p})$ therefore satisfies the optimality property of Definition~\ref{DefMFGSolution}. By construction of $\Psi$ and the fixed point property of $\mu$, $(\nu,u,\mathfrak{p})$ also satisfies the conditional mean field property of Definition~\ref{DefMFGSolution}. It follows that $(\nu,u,\mathfrak{p})$ is a feedback solution of the mean field game.
\end{proof}

\subsection{Continuity of optimal controls} \label{SectExistContinuity}

Here, we work under all the hypotheses of the previous subsections plus a couple of additional assumptions to be introduced below. Let $\mathfrak p \in \mathcal M$ be a flow of measures such that the induced flow of means $m_{\mathfrak p}$ is continuous, that is, we assume that the mapping
\[
	t\mapsto m_{\mathfrak p}(t)\doteq \int_{\mathbb{R}^{d}} w(y) \mathfrak p(t,dy)
\]
is continuous. From Lemma~\ref{lemma:flowm} we know that this is the case if $\mathfrak{p}$ is the flow of measures induced by some $\mu \in \mathcal{K}$.

The aim here is to show that $\hat{\mathrm u}(\cdot,X(\cdot),m_{\mathfrak{p}}(\cdot),Z^{\mathfrak{p}}(\cdot))$, the optimal control process appearing in \eqref{opt-control}, corresponds to a Markovian feedback strategy that is continuous both in its time and its state variable. In view of Lemma~\ref{lemma:uhat} and the continuity of $m_{\mathfrak{p}}$, it suffices to show that $Z^{\mathfrak{p}}(t)$ can be expressed as a continuous function of time $t$ and the forward state variable $X(t)$.

Set, for $(t,x,z)\in [0,T]\times \mathbb{R}^{d}\times \mathbb{R}^{d\times d}$,
\[
	\tilde h(t,x,z) \doteq h(t,x,m(t),z,\hat{\mathrm u}(t,x,m_{\mathfrak{p}}(t),z)).
\]
Notice that $\tilde{h}(t,x,z) = H(t,x,m_{\mathfrak{p}}(t),z)$. Moreover, by Assumptions~\ref{ass2}, Lemma~\ref{lemma:uhat}, and the continuity of $m_{\mathfrak{p}}$, we have that $\tilde{h}$ is continuous (jointly in all its variables). By~\Hypref{HypMeasBounded} and \Hypref{HypActionSpace}, $\tilde{h}(\cdot,\cdot,z)$ is bounded for each $z\in \mathbb{R}^{d\times d}$.

Consider the following \mbox{BSDE} under $\mathbf P$:
\begin{equation} \label{bsde-P}
	Y(t) = F(\tau , X(\tau)) + \int_{t}^{T} \tilde h\left(s,X(s), Z(s) \right) \mathbf 1_{(s < \tau)} ds - \int_{t} ^{T} Z(s) dW(s),
\end{equation}
with the same convention regarding times exceeding $\tau$ as in Subsection~\ref{SectExistOptControl}. The \mbox{BSDE} above possesses a unique $L^2$-solution $(Y,Z)$, which coincides with the unique solution $(Y^{\mathfrak{p}},Z^{\mathfrak{p}})$ of the \mbox{BSDE} \eqref{bsde-H}. To prove the continuity property, we adopt the following procedure: first, we consider the \mbox{PDE} corresponding to the \mbox{BSDE} \eqref{bsde-P} and prove some regularity implying, in particular, the continuity of the gradient of the solution; second, we provide a verification argument essentially identifying the solution of the \mbox{PDE} and its gradient with the unique solution $(Y,Z)$ of \eqref{bsde-P}. This is performed in the proof of the next proposition. Before that, we have to introduce some notation on Sobolev and H\"older-type norms.

Set $D\doteq [0,T) \times O$, and let $\partial_p D \doteq ([0,T)\times \partial O) \cup (\{T\} \times \bar O)$ be the parabolic boundary of $D$. Moreover, let $\mathcal A$ denote the (parabolic) Dynkin operator associated with $X$ under $\mathbf P$, that is,
\[
	\mathcal A \doteq \partial_t + \frac{1}{2}\textrm{Tr}(a\partial^2 _{xx}),
\]
where $a\doteq \sigma \trans{\sigma}$. Let $L^\lambda(D)$ denote the space of $\lambda$-th power integrable functions on $D$, with $\| \cdot \|_{\lambda, D}$ norm in $L^\lambda (D)$. For $1 < \lambda < \infty$, let $\mathcal H^\lambda(D)$ denote the space of functions $\psi$ that together with their generalized partial derivatives $\partial_t \psi$, $\partial_{x_i} \psi$, $\partial_{x_i x_j}\psi$, $i,j=1,\ldots,n$, are in $L^\lambda (D)$. In $\mathcal H^\lambda (D)$ let us introduce the norm (of Sobolev type)
\begin{equation}
	\| \psi \|^{(2)} _{\lambda, D} \doteq \|\psi\|_{\lambda, D} + \|\partial_t \psi\|_{\lambda, D} + \sum_{i=1}^d \|\partial_{x_i}\psi\|_{\lambda, D} + \sum_{i,j=1}^d \|\partial_{x_i x_j} \psi\|_{\lambda, D}.
\end{equation}
Let us also consider H\"older-type norms as follows. For $0< \alpha \leq 1$, let
\begin{align*}
	\| \psi \|_D &\doteq \sup_{(t,x) \in D} | \psi(t,x) | , \\
	| \psi |^{\alpha} _D &\doteq \| \psi \|_{D} +\hspace{-1ex} \sup_{x,y \in \bar O, t \in [0,T]}\hspace{-1ex} \frac{| \psi(t,x)-\psi(t,y)|}{|x-y|^\alpha} +\hspace{-1ex} \sup_{s,t \in [0,T], x \in \bar O}\hspace{-1ex} \frac{| \psi(s,x)-\psi(t,x)|}{|s-t|^{\alpha/2}}.
\end{align*}
Finally, we let
\[
	| \psi | ^{1+\alpha} _D \doteq | \psi |^\alpha _D + \sum_{i=1}^d | \partial_{x_i} \psi |^{\alpha}_D .
\]

\begin{ass}\label{ass:F}
\
\begin{enumerate}
	\item[(i)] $\bar{b}$, $f$ belong to $\mathbf C^{0,1}([0,T]\times (\bar{O} \times \mathbb R^{d_{0}}))$ and $\mathbf C^{0,1}([0,T]\times (\bar{O} \times \mathbb R^{d_{0}} \times \Gamma))$, respectively. Thus, $b\in \mathbf C^{0,1}([0,T]\times (\bar{O} \times \mathbb R^{d_{0}} \times \Gamma))$.

	\item[(ii)] $F(T,\cdot) \in \mathbf C^2 (\bar O)$ and $F(t,x) = \tilde F(t,x)$ for all $t\in [0,T]$ and $x \in \partial O$, where $\tilde F \in \mathbf C^{1,2}(\bar D)$. Moreover, for some $\lambda > \frac{d+2}{2}$,
\[ 
	\| F \|^{(2)} _{\lambda, \partial_p D} < \infty.
\]
\end{enumerate}
\end{ass} 

Assumption~\ref{ass:F}(ii) on the terminal costs $F$ corresponds to assumption (E7) in \citet[Appendix E]{flemingrishel75}.

\begin{prop} \label{prop:cont}
Grant \Hyprefall, \hypref{HypAddND} as well as Assumptions \ref{ass2} and \ref{ass:F}. The optimal control is then given by
\[
	\mathrm u^{\mathfrak{p}} _t = \hat{\mathrm u}(t, X(t) ,m_{\mathfrak{p}}(t), \partial_x \varphi(t,X(t)) \sigma), \quad t\in [0,\tau],
\] 
where the function $\varphi$ is the unique solution in $\mathcal H^\lambda (D)$ of the Cauchy problem
\begin{align}\label{pde}
	& -\mathcal A \varphi - \tilde h(\cdot, \varphi, \partial_x \varphi \sigma) =0 \text{ on }D,& & \varphi=F \text{ on }\partial_p O,&
\end{align}
and it satisfies 
\begin{equation}\label{estimate-pde1}
\| \varphi \|^{(2)} _{\lambda, D} \leq M\left( \| f \|_{\lambda, D} + \| F\|^{(2)} _{\lambda, \partial_p D} \right),
\end{equation}
with constant $M$ depending only on bounds on $b$ and $\sigma$ on $D$.
Moreover,
\begin{equation}\label{estimate-pde2}
	| \varphi |_D ^{1+\alpha} \le M^\prime \| \varphi \|^{(2)} _{\lambda, D}, \quad \alpha = 1 - \frac{d+2}{\lambda},
\end{equation}
for some constant $M^\prime$ depending on $D$ and $\lambda$.

As a consequence, the function $(t,x) \mapsto \hat{\mathrm u}(t,x,m(t),\partial_x \varphi (t,x)\sigma)$ is continuous over $[0,T] \times \bar O$.
\end{prop}

\begin{proof}
First, consider the Cauchy problem \eqref{pde}. The existence and uniqueness of a solution in $\mathcal H^\lambda (D)$ satisfying the estimate \eqref{estimate-pde1} can be found in \citet[Appendix E, Proof of Theorem VI.6.1, pp. 208-209]{flemingrishel75}. Observe that our assumptions together with the continuity of $\tilde h$ imply the hypotheses (6.1)--(6.3) in \citet[p.\,167]{flemingrishel75}, with the exception of (6.3)(b). In fact, in our setting the drift $b$ and the running cost $f$ are only continuous in $t$ and not necessarily $\mathbf C^1$; cf.\ Assumption~\ref{ass:F}(i). However, a careful inspection of the proof by Fleming and Rishel shows that we can apply their approximation argument excluding the last sentence. Hence, there exists a non-increasing sequence $(\varphi_n)$ converging pointwise to the unique solution $\varphi \in \mathcal H^\lambda (D)$ of the Cauchy problem \eqref{pde}. Moreover, each $\varphi_n$ satisfies the inequality \eqref{estimate-pde1}.
Using \citet[p.\,80, p.\,342]{lady-solo-ural}, the estimate (\ref{estimate-pde1}) gives (\ref{estimate-pde2}) if $\lambda > d+2$, and in this case $\alpha = 1-\frac{d+2}{\lambda}$ where the constant $M^\prime$ depends on $D$ and $\lambda$.
Notice that the upper bounds in \eqref{estimate-pde1} and \eqref{estimate-pde2} are both uniform in $n$. Therefore, letting $n \to \infty$, we have that the solution $\varphi$ satisfies the same bounds. In particular, \eqref{estimate-pde2} yields that the gradient $\partial_x \varphi$ is well-defined in the classical sense and it is continuous on $D$. 

To complete the proof, it suffices to show that the pair of processes $(\varphi(t,X(t)), \partial_x \varphi(t,X(t)) \sigma)_{t\in [0,\tau]}$ coincides with the unique $L^2$ solution of the \mbox{BSDE} \eqref{bsde-P}. We use a verification argument based on a generalization of It\^o's formula as in \citet[Theorem 1, Sect.\,2.10]{krylov80}. Applying that formula between $t$ and $\tau$, on the random set $\{t \leq \tau\}$, we have with $\mathbf P$-probability one
\[
\begin{split}
	\varphi(\tau,X(\tau)) = \varphi(t,X(t)) &+ \int_t ^\tau \mathcal A \varphi(s,X(s)) ds \\
	& + \int_t ^\tau \partial_x \varphi(s,X(s))\sigma dW(s).
\end{split}
\]
Being $\varphi$ a solution in $\mathcal H^\lambda (D)$ of the Cauchy problem \eqref{pde}, we have (by re-arranging the terms)
\[
\begin{split}
	\varphi(t,X(t)) = F(\tau, X(\tau)) &- \int_t ^\tau \tilde h (s,X(s), \partial_x \varphi(s,X(s))\sigma) ds \\
	&- \int_t ^\tau \partial_x \varphi(s,X(s))\sigma dW(s),
\end{split}
\]
which gives, by uniqueness of the solution $(Y,Z)$ of the BSDE \eqref{bsde-P}, that
\begin{align*}
	Y(t) = \varphi(t,X(t)), \quad Z(t) = \partial_x \varphi (t,X(t)) \sigma,
\end{align*}
on the event $\{t \leq \tau\}$ for all $t \in [0,T]$.
\end{proof}

Proposition~\ref{prop:cont} and Theorem~\ref{nash-existence} yield (together with Lemma~\ref{lemma:flowm}) the following corollary. Observe that the control actions can be chosen in an arbitrary way when the state process is outside $\cl (O)$.

\begin{crll} \label{crll:nash-cont}
	Grant \Hyprefall, \hypref{HypAddND} as well as Assumptions \ref{ass2} and \ref{ass:F}. Then there exist a feedback solution of the mean field game $(\nu,u,\mathfrak{p})$ and a continuous function $\tilde{u}\!: [0,T]\times \mathbb{R}^{d} \rightarrow \Gamma$ such that
	\[
		u(t,\phi) = \tilde{u}(t,\phi(t)) \text{ for all } (t,\phi)\in [0,T]\times \mathcal{X}.
	\]
	In particular, $u$ is continuous on $[0,T]\times \mathcal{X}$.
\end{crll}

\bigskip
We end this section by sketching the \mbox{PDE} approach to mean field games with absorption. For simplicity, we only consider a setting analogous to that of \citet{lasrylions06b}, where the dynamics are of calculus-of-variation-type with additive noise and the running costs split into two parts, one depending on the control, the other on the measure variable. Thus, we assume that
\begin{align*}
	&\sigma \equiv \sigma \Id_{d}, & & \bar{b}(t,x,y) \equiv 0,& & f(t,x,y,\gamma) = f_{0}(t,x,\gamma) + f_{1}(t,x,y) &
\end{align*}
for some scalar constant $\sigma > 0$ and suitable functions $f_{0}\!: [0,T]\times \mathbb{R}^{d}\times \Gamma \rightarrow \mathbb{R}$ and $f_{1}\!: [0,T]\times \mathbb{R}^{d}\times \mathbb{R}^{d_{0}} \rightarrow \mathbb{R}$. Let us also assume that the initial distribution $\nu$ is absolutely continuous with respect to Lebesgue measure with density $m_{0}$.

Let $(\nu,u,\mathfrak{p})$ be a solution according to Corollary~\ref{crll:nash-cont}, and let $\tilde{u}$ be the associated Markov feedback strategy.

Set $H(t,x,z) \doteq \max_{\gamma\in \Gamma} \left\{ -\gamma\cdot z - f_{0}(t,x,\gamma)\right\}$. Let $V$ be the unique solution of the Hamilton-Jacobi Bellman equation
\begin{equation} \label{EqPDEappHJB}
	-\partial_{t} V - \frac{\sigma^2}{2} \Delta V + H(t,x,\nabla V) = f_{1}\left(t,x,\int w(y) \mathfrak{p}(t,dy)\right) \text{ in } [0,T)\times O
\end{equation}
with boundary condition $V(t,x) = F(t,x)$ in $\{ T\}\times \cl(O) \cup [0,T)\times \partial O$, and let $m$ be the unique solution of the Kolmogorov forward equation
\begin{equation} \label{EqPDEappKF}
	\partial_{t}m - \frac{\sigma^2}{2} \Delta m + \mathrm{div} \left(m(t,x)\cdot \tilde{u}(t,x) \right) = 0 \text{ in } (0,T]\times O
\end{equation}
with initial condition $m(0,x) = m_{0}(x)$ and boundary condition $m(t,x) = 0$ in $(0,T]\times \partial O$. Then
\begin{align} \label{EqPDEappRel}
 & \tilde{u}(t,x) = -D_{z}H\left(t,x,\nabla V(t,x)\right),& & \mathfrak{p}(t,dx) = \frac{m(t,x)}{\int_{O}m(t,y)dy}dx.&
\end{align}

A non-standard feature in \eqref{EqPDEappRel} is the renormalization by the factor $1/\int_{O}m(t,y)dy$. This is related to the fact that $m(t,\cdot)$ will be the density of a sub-probability measure, not necessarily a probability measure if $t > 0$. Relationship \eqref{EqPDEappRel} clearly allows to eliminate the measure flow $\mathfrak{p}$ from Eq.~\eqref{EqPDEappHJB} and the feedback control $\tilde{u}$ from Eq.~\eqref{EqPDEappKF}, yielding a coupled system of (backward) Hamilton-Jacobi Bellman equation and Kolmogorov forward equation. This gives the \mbox{PDE} system for the mean field game with absorption.


\section{A counterexample} \label{SectCounterExmpl}

If the diffusion coefficient is degenerate and no additional controllability assumptions are made, then a regular feedback solution of the mean field game need not induce a sequence of approximate Nash equilibria with vanishing error.

The following (counter-)example is constructed in such a way that, for a certain initial condition, the individual state processes of the $N$-player games cannot leave the set $O$ of non-absorbing states before terminal time, while this is possible in the mean field game limit. Exit before terminal time can actually occur only at one fixed instant, if at all. With the choice of the costs we make, leaving before terminal time is preferable. The individual strategies that the solution of the mean field game induces in the $N$-player games will therefore push the state processes towards the boundary if they start from certain initial points, but will fail to make them exit (for $N$ odd, the probability of exit equals zero, for $N$ even it will tend to zero as $N\to \infty$). The failure to leave the set $O$ in the $N$-player games will produce costs that are higher than those of an alternative individual response strategy which does not try to make its state process exit.

To achieve the described effect, we choose deterministic state dynamics; the only source of randomness comes from the initial conditions. Individual states will have three components, where one component (the third) corresponds to time, while the second component simply keeps its own initial condition. Only the first component is controlled; its evolution is driven, apart from the control, by an average over the second component of the states of all players (a mean with respect to the measure variable in the limit). That interaction term is responsible for the different controllability of the $N$-player systems with respect to the limit system. Notice that the construction also relies on the particular choice of the initial conditions, which are singular with respect to Lebesgue measure. To be specific, consider the following data for our systems:
\begin{itemize}
	\item dimensions: $d = 3$, $d_{0} = 1$;
	
	\item time horizon: $T = 2$;
	
	\item set of control actions: $\Gamma \doteq \left\{ \gamma\in \mathbb{R}^{3} :  \gamma_1\in [-1,1], \gamma_2 = 0 = \gamma_3 \right\}$;
	
	\item set of non-absorbing states
	\[
	\begin{split}
		O\doteq \Bigl\{ x\in \mathbb{R}^{3} : & -4 < x_{1} \, , \; -2 < x_{2} < 2 \, , \; -1 < x_{3} < \frac{11}{5} \, ,\\
		&\; x_{1} < 1 + e^{x_{3}-1} \Bigr\}
	\end{split}
	\]

	\item measure integrand: $w$ bounded Lipschitz and such that $w(x)= x_{2}$ for all $x\in \cl(O)$;
	
	\item drift function:
	\[
		\bar{b}(t,x,y)\doteq \begin{pmatrix} -|y|\wedge \frac{1}{4}\\ 0\\ 1 \end{pmatrix},\quad (t,x,y)\in [0,2]\times \mathbb{R}^{3} \times \mathbb{R};
	\]
	
	\item dispersion coefficient: $\sigma\equiv 0$;
	
	\item running costs: $f\equiv 1$;
	
	\item terminal costs: $F$ non-negative bounded Lipschitz and such that
	\[
		F(t,x)= 1 + \frac{x_{3}}{12}\cdot x_{1} \text{ for all } (t,x)\in [0,2]\times \cl(O).
	\]
	
\end{itemize}

\begin{rem}
The set of non-absorbing states $O$ defined above has a boundary that is only piecewise smooth. Let $B \subset \mathbb{R}^3$ be the line segment given by
\[
	B\doteq \left\{ x\in \mathbb{R}^{3} : x_1 = 2 \, , \; -1\leq x_2 \leq 1 \, , \; x_3 = 1 \right\}.
\]
As will become clear below, the (counter-)example works for any bounded open set $O$ as long as $O$ contains
\[
	\left\{ x\in \mathbb{R}^{3} : -1\leq x_2 \leq 1 \, , \; 0 \leq x_3 \leq 2 \, , \; -1 - \frac{5}{4}x_3 \leq x_1 \leq 1+x_3 \right\} \setminus B,
\]
while $B \cap O = \emptyset$ (hence $B \subset \partial O$). There are bounded open sets with this property and smooth ($\mathbf{C}^{2}$ or even $\mathbf{C}^{\infty}$) boundary.
\end{rem}

Let $\rho$ denote the Rademacher distribution on $\Borel{\mathbb{R}}$, that is, $\supp(\rho) = \{-1,1\}$ and $\rho(\{-1\}) = 1/2 = \rho(\{1\})$. Define a probability measure $\nu$ on $\Borel{\mathbb{R}^{3}}$ by
\[
	\nu\doteq \rho\otimes \rho\otimes \delta_{0},
\]
and choose, for $N\in \mathbb{N}$, the initial distribution for the $N$-player game according to $\nu_{N}\doteq \otimes^{N} \nu$.

The dynamics of the $N$-player game are thus given by
\begin{equation*}
\begin{split}
\begin{pmatrix}
	X^{N}_{i,1}(t) \\ X^{N}_{i,2}(t) \\ X^{N}_{i,3}(t)
\end{pmatrix}
	&= \begin{pmatrix}
	 	X^{N}_{i,1}(0) \\ X^{N}_{i,2}(0) \\ 0
	 \end{pmatrix}
	 \\
	 &\;
	 + \int_{0}^{t} \begin{pmatrix}
	 u_{i,1}(s,\boldsymbol{X}^{N}) - \Bigl|\frac{1}{\bar{N}^{N}(s)}\sum_{j=1}^{N} \mathbf{1}_{[0,\tau^{N}_{j})}(s)\cdot X^{N}_{j,2}(s)\Bigr|\wedge \frac{1}{4} \\ 0\\ 1
	 \end{pmatrix}
	 ds,
\end{split}
\end{equation*}
where $X^{N}_{i,k}(0)$, $k\in \{1,2\}$, $i\in \{1,\ldots,N\}$, are i.i.d.\ Rademacher and $\boldsymbol{u} = (u_{1},\ldots,u_{N}) \in \mathcal{U}_{fb}^{N}$ is an admissible strategy vector. Clearly, for all $s\in [0,2]$,
\[
	\frac{1}{\bar{N}^{N}(s)}\sum_{j=1}^{N} \mathbf{1}_{[0,\tau^{N}_{j})}(s)\cdot X^{N}_{j,2}(s) = \frac{1}{\bar{N}^{N}(s)}\sum_{j=1}^{N} \mathbf{1}_{[0,\tau^{N}_{j})}(s)\cdot X^{N}_{j,2}(0).
\]  
Randomness thus enters the system only through the initial condition. We may therefore fix the stochastic basis. To this end, let $\xi^{N}_{i,k}$, $k\in \{1,2\}$, $i\in \{1,\ldots,N\}$, be i.i.d.\ Rademacher random variables defined on some probability space $(\Omega_{N},\mathcal{F}^{N},\Prb_{N})$. As filtration, we may take any filtration that makes the $\xi^{N}_{i,k}$ measurable at time zero. The dynamics of the $N$-player system can be rewritten as
\begin{equation*}
\begin{split}
\begin{pmatrix}
	X^{N}_{i,1}(t) \\ X^{N}_{i,2}(t) \\ X^{N}_{i,3}(t)
\end{pmatrix}
	= \begin{pmatrix} \xi^{N}_{i,1} + \int_{0}^{t}\left( u_{i,1}(s,\boldsymbol{X}^{N}) - \Bigl|\frac{1}{\bar{N}^{N}(s)}\displaystyle \sum_{j=1}^{N} \mathbf{1}_{[0,\tau^{N}_{j})}(s)\cdot \xi^{N}_{j,2}\Bigr|\wedge \frac{1}{4}\right)ds \\
		 \xi^{N}_{i,2} \\
		 t
		 \end{pmatrix}
\end{split}
\end{equation*}

Since $u_{i,1}$ takes values in $[-1,1]$ and $\xi^{N}_{i,1}$ values in $\{-1,1\}$, we have, for $\Prb_{N}$-almost all $\omega\in \Omega_{N}$,
\begin{equation} \label{EqExmplComponentOne}
	-1 - \frac{5}{4}t \leq X^{N}_{i,1}(t,\omega) \leq 1 + t \text{ for all } t\in [0,2].
\end{equation}
By construction of $O$, $X^{N}_{i}(\cdot,\omega)$ can leave $O$ before the terminal time only if $X^{N}_{i,1}(1,\omega) = 2$; this is possible only if $\sum_{j=1}^{N} \xi^{N}_{j,2}(\omega) = 0$. But
\[
	\Prb_{N}\left( \sum_{j=1}^{N} \xi^{N}_{j,2}(\omega) = 0 \right) = \begin{cases}
		0 &\text{if $N$ is odd,}\\
		\binom{N}{N/2} \frac{1}{2^{N}} &\text{if $N$ is even.}
	\end{cases}
\]
Since $\binom{N}{N/2}\frac{1}{2^{N}} \to 0$ as $N\to \infty$, we may assume for simplicity that $N$ is odd. The dynamics of the $N$-player game then reduce to
\begin{equation} \label{EqExmplPrelimitDynamics}
\begin{pmatrix}
	X^{N}_{i,1}(t) \\ X^{N}_{i,2}(t) \\ X^{N}_{i,3}(t)
\end{pmatrix}
	 = \begin{pmatrix} \xi^{N}_{i,1} + \int_{0}^{t} u_{i,1}(s,\boldsymbol{X}^{N})ds - t\cdot \left( \Bigl|\frac{1}{N}\sum_{j=1}^{N} \xi^{N}_{j,2}\Bigr|\wedge \frac{1}{4}\right) \\
	 \xi^{N}_{i,2} \\
	 t
	 \end{pmatrix}
\end{equation}
$t\in [0,2]$, for any admissible strategy vector $\boldsymbol{u}$. The associated costs for player $i$ are, in view of \eqref{EqExmplComponentOne},
\[
	J^{N}_{i}(\boldsymbol{u}) = 2 + \Mean_{N}\left[ 1 + \frac{1}{6}\int_{0}^{2} u_{i,1}(s,\boldsymbol{X}^{N})ds - \frac{1}{3} \left( \Bigl|\frac{1}{N}\sum_{j=1}^{N} \xi^{N}_{j,2}\Bigr|\wedge \frac{1}{4}\right) \right].\footnote{It is easy to see that if $\boldsymbol{u}$ is such that $u_{i,1}\equiv -1$ for all $i\in \{1,\ldots,N\}$, then $\boldsymbol{u}$ is a Nash equilibrium for the $N$-player game.}
\]

Let us turn to the limit model. Given a flow of measures $\mathfrak{p}\in \mathcal{M}$ and a stochastic-open loop control $((\Omega,\mathcal{F},(\mathcal{F}_{t}),\Prb),\xi,\alpha,W) \in \mathcal{A}$ such that $\Prb\circ \xi^{-1} = \nu$, the dynamics are given by
\begin{equation} \label{EqExmplLimitDynamics}
\begin{pmatrix}
	X_{1}(t) \\ X_{2}(t) \\ X_{3}(t)
\end{pmatrix}
	 = \begin{pmatrix}
	 	 	\xi_{1} \\ \xi_{2} \\ 0
	 \end{pmatrix}
+ \int_{0}^{t} \begin{pmatrix} \alpha_{1}(s) - \left|\int_{\mathbb{R}^3} w(y) \mathfrak{p}(s,dy) \right|\wedge \frac{1}{4} \\
	 	 	 0\\
	 	 	 1
	 	 	 \end{pmatrix}
	 	 	  ds,\quad t\in [0,2].
\end{equation}
Notice that $\xi_{1}$, $\xi_{2}$ are independent Rademacher variables.

Suppose that $\mathfrak{p}$ is such that, for all $t\in [0,2]$, $\supp(\mathfrak{p}(t))\subseteq \cl(O)$ and $\int_{\mathbb{R}^3} w(y) \mathfrak{p}(t,dy) = 0$. The dynamics then reduce to
\begin{equation} \label{EqExmplReducedLimitDynamics}
\begin{pmatrix}
	X_{1}(t) \\ X_{2}(t) \\ X_{3}(t)
\end{pmatrix}
	 = \begin{pmatrix}
	 	 	\xi_{1} + \int_{0}^{t}\alpha_{1}(s)ds \\ \xi_{2} \\ t
	 \end{pmatrix},
\end{equation}
while the associated expected costs are equal to $\Mean[\zeta_{\alpha}]$ where
\[
	\zeta_{\alpha}(\omega)\doteq \tau^X(\omega)\wedge 2 + 1 + \frac{\tau^{X}(\omega)\wedge 2}{12}\cdot X_{1}(\tau^{X}(\omega)\wedge 2,\omega),\quad \omega\in \Omega.
\]
Since $\alpha_{1}$ takes values in $[-1,1]$ and $\xi_{1}$ values in $\{-1,1\}$, we have, $\Prb$-almost surely,
\begin{equation} \label{EqExmplComponentOneLimit}
	-1 - t \leq X_{1}(t) \leq 1 + t \text{ for all } t\in [0,2].
\end{equation}
By construction of $O$, it follows that, for $\Prb$-almost every $\omega\in \Omega$, $X(\cdot,\omega)$ leaves $O$ before time $T=2$ if and only if $\xi_{1}(\omega) = 1$ and $\alpha_{1}(t,\omega) = 1$ for Lesbegue almost every $t\in [0,1]$. In this case, $\tau^{X}(\omega) = 1$, $X_{1}(1,\omega) = 2$, and $\zeta_{\alpha}(\omega) = 2 + 1/6$. If $X(\cdot,\omega)$ does not leave $O$ before the terminal time, then, by \eqref{EqExmplComponentOneLimit}, $\zeta_{\alpha}(\omega) \geq 2+1/2$, and the optimal control is to choose $\alpha_{1}(t,\omega) = -1$ for almost every $t\in [0,2]$. Therefore, if for $\Prb$-almost every $\omega\in \Omega$,
\begin{equation} \label{EqOptOpenLoop}
	\alpha(t,\omega) = \begin{cases}
	\trans{(1, 0, 0)} &\text{if } \xi_{1}(\omega) = 1 \text{ and } t\in [0,1],\\ 
	\trans{(-1, 0, 0)} &\text{if } \xi_{1}(\omega) = -1 \text{ or } t > 1,
	\end{cases}
\end{equation}
then
\[
	\Mean[\zeta_{\alpha}] = \frac{1}{2}\left(2+\frac{1}{6}\right) + \frac{1}{2}\left(2+\frac{1}{2}\right) = 2 + \frac{1}{3}
\]
and $\alpha$ is optimal in the sense that
\[
	\Mean[\zeta_{\alpha}] = V(\nu;\mathfrak{p}).
\]	

Now, choose $((\Omega,\mathcal{F},(\mathcal{F}_{t}),\Prb),\xi,\alpha,W) \in \mathcal{A}$ such that $\Prb\circ \xi^{-1} = \nu$ and $(\xi,\alpha)$ satisfies \eqref{EqOptOpenLoop} $\Prb$-almost surely. Let $X$ be the unique strong solution of Eq.~\eqref{EqExmplReducedLimitDynamics}, and define a flow of measures $\mathfrak{p}_{\ast}$ according to
\[
	\mathfrak{p}_{\ast}(t,\cdot)\doteq \Prb(X\in \cdot \;|\; \tau^{X} > t ),\quad t\in [0,2].
\]
Notice that $\Prb(\tau^{X} > t) \geq 1/2$ for all $t\in [0,2]$; thus, $\mathfrak{p}_{\ast}$ is well defined. By construction, $\supp(\mathfrak{p}_{\ast}) \subseteq \cl(O)$, which implies
\[
	\int_{\mathbb{R}^{3}} w(y)\mathfrak{p}_{\ast}(t,dy) = \int_{\mathbb{R}^{3}} y_{2}\mathfrak{p}_{\ast}(t,dy).
\]

We are going to show that $\int_{\mathbb{R}^{3}} w(y)\mathfrak{p}_{\ast}(t,dy) = 0$ for all $t\in [0,2]$. By definition of $\nu$, $X_{1}(0)$, $X_{2}(0)$, $X_{3}(0)$ are independent. Moreover, $X_{3}$ is (almost surely) deterministic, while $\alpha_{1}$ is measurable with respect to $\sigma(\xi_{1}) = \sigma(X_{1}(0))$. This implies that the real-valued processes $X_{1}$, $X_{2}$, $X_{3}$ are independent. The time of first exit $\tau^{X}$ can be rewritten in terms of $X_{1}$ and $X_{3}$ only. It follows that $\tau^{X}$ and $X_{2}$ are independent, hence
\begin{multline*}
	\int_{\mathbb{R}^3}w(y) \mathfrak{p}_{\ast}(t,dy) = \int_{\mathbb{R}^3}y_{2} \mathfrak{p}_{\ast}(t,dy) \\
	= \Mean_{\Prb}\left[ X_{2}(t) \right] = \Mean_{\Prb}\left[ X_{2}(0) \right] = \int_{\mathbb{R}} z \rho(dz) = 0.
\end{multline*}
As a consequence, $X$ solves Eq.~\eqref{EqExmplLimitDynamics} with flow of measures $\mathfrak{p} = \mathfrak{p}_{\ast}$ and the associated costs are optimal in the sense that
\[
	\Mean[\zeta_{\alpha}] = V(\nu;\mathfrak{p}_{\ast}).
\]

Recall that $\mathfrak{p}_{\ast}$ was defined as the conditional flow of measures for the law of $X$. Since $X$ solves Eq.~\eqref{EqExmplLimitDynamics} with flow of measures $\mathfrak{p} = \mathfrak{p}_{\ast}$, the conditional mean field property of Definition~\ref{DefMFGSolution} holds.

Relation \eqref{EqOptOpenLoop} between the open-loop control $\alpha$ and the initial condition $\xi$ induces a feedback strategy in $\mathcal{U}_{1}$, namely
\[
	u^{\ast}(t,\phi) = \begin{cases}
	\trans{(1, 0, 0)} &\text{if } \phi_{1}(0) = 1 \text{ and } t\in [0,1],\\ 
	\trans{(-1, 0, 0)} &\text{if } \phi_{1}(0) = -1 \text{ or } t > 1, \\
	\text{arbitrary} &\text{otherwise.}
	\end{cases}
\]
For pairs $(t,\phi)$ where the control is unspecified, we may choose the control actions in such a way that $u^{\ast}$ becomes Lipschitz continuous in the state variable at every point in time; to be specific, set
\[
	u^{\ast}(t,\phi)\doteq \begin{cases}
	\trans{(-1\vee \phi_{1}(0) \wedge 1, 0, 0)} &\text{if } t\in [0,1],\\ 
	\trans{(-1, 0, 0)} &\text{if } t > 1.
	\end{cases}
\]
Eq.~\eqref{EqLimitDynamics} is well posed under $u^{\ast}$ given the flow of measures $\mathfrak{p}_{\ast}$ and any initial distribution with support in $O$. Let $((\Omega_{\ast},\mathcal{F}_{\ast},\Prb_{\ast}),(\mathcal{F}^{\ast}_{t}),\bar{W},X^{\ast})$ be a solution of Eq.~\eqref{EqLimitDynamics} under $u^{\ast}$ with flow of measures $\mathfrak{p}_{\ast}$ and initial distribution $\nu = \rho\otimes \rho\otimes \delta_{0}$. Then, by construction,
\[
	\Prb_{\ast}\circ (X^{\ast})^{-1} = \Prb\circ X^{-1}.
\]	
This implies the optimality property
\[
	J(\nu,u^{\ast};\mathfrak{p}_{\ast}) = V(\nu;\mathfrak{p}_{\ast}),
\]
but also the conditional mean field property of Definition~\ref{DefMFGSolution}, namely
\[
	\mathfrak{p}_{\ast}(t) = \Prb(X\in \cdot \;|\; \tau^{X} > t ) = \Prb_{\ast}(X^{\ast}\in \cdot \;|\; \tau^{X^{\ast}} > t ) \text{ for all } t\in [0,2].
\]
It follows that $(\nu,u^{\ast},\mathfrak{p}_{\ast})$ is a feedback solution of the mean field game.

Let us check whether or not the feedback strategy $u^{\ast}$ induces a sequence of approximate Nash equilibria in analogy with Theorem~\ref{ThApproximateNash}. Thus, for $N\in \mathbb{N}$,
define $\boldsymbol{u}^{N} = (u^{N}_{1},\ldots,u^{N}_{N})$ by
\[
	u^{N}_{i}(t,\boldsymbol{\phi}) \doteq u^{\ast}(t,\phi_{i}),\quad t\in [0,T],\; \boldsymbol{\phi}=(\phi_{1},\ldots,\phi_{N})\in \mathcal{X}^{N},\; i\in \{1,\ldots,N\}.
\]
Equation~\eqref{EqPrelimitDynamics} has a unique solution under $\boldsymbol{u}^{N}$ with initial distribution $\nu_{N}$ because $\bar{b}$, $w$ are Lipschitz continuous in the state variable and $\boldsymbol{u}^{N}$ is measurable with respect to $\sigma(\boldsymbol{X}^{N}(0))$. Therefore, $\boldsymbol{u}^{N}\in \mathcal{U}^{N}_{fb}$. For simplicity, assume again that $N$ is odd. Let $\boldsymbol{X}^{N}$ be the unique strong solution of Equation~\eqref{EqExmplPrelimitDynamics} under $\boldsymbol{u}^{N}$. Then, for every $i\in \{1,\ldots,N\}$,
\begin{align*}
	J^{N}_{i}(\boldsymbol{u}^{N}) &= 2 + \Mean_{N}\left[ 1 + \frac{1}{6}\int_{0}^{2} u^{\ast}(s,X^{N}_{i})ds - \frac{1}{3} \left( \Bigl|\frac{1}{N}\sum_{j=1}^{N} \xi^{N}_{j,2}\Bigr|\wedge \frac{1}{4}\right) \right] \\
	&= 3 - \frac{1}{2}\cdot \frac{2}{6} - \frac{1}{3} \Mean_{N}\left[ \Bigl|\frac{1}{N}\sum_{j=1}^{N} \xi^{N}_{j,2}\Bigr|\wedge \frac{1}{4} \right] \\
	&\geq \frac{33}{12}.
\end{align*}
Suppose player one deviates from $\boldsymbol{u}^{N}$ by choosing the strategy that is constant and equal to $-1$. Denote that strategy by $v$. Notice that $[\boldsymbol{u}^{N,-1},v]\in \mathcal{U}^{N}_{fb}$. For the associated costs, we have
\begin{align*}
	J^{N}_{1}([\boldsymbol{u}^{N,-1},v]) &= 2 + \Mean_{N}\left[ 1 - \frac{2}{6} - \frac{1}{3} \left( \Bigl|\frac{1}{N}\sum_{j=1}^{N} \xi^{N}_{j,2}\Bigr|\wedge \frac{1}{4}\right) \right] \\
	&= 3 - \frac{1}{3} - \frac{1}{3} \Mean_{N}\left[ \Bigl|\frac{1}{N}\sum_{j=1}^{N} \xi^{N}_{j,2}\Bigr|\wedge \frac{1}{4} \right] \\
	&\leq \frac{32}{12}.
\end{align*}
Player one can thus save costs of at least $1/12$ by deviating from $\boldsymbol{u}^{N}$ for every $N$ odd (asymptotically, also for $N$ even). It follows that the strategy vectors induced by the feedback solution $(\nu,u^{\ast},\mathfrak{p}_{\ast})$ of the mean field game do not yield approximate Nash equilibria with vanishing error.


\begin{appendix}

\section*{Appendix}

Let $\nu\in \prbms{\mathbb{R}^{d}}$ with support in $O$. Recall the definition of $\Theta_{\nu}\in \prbms{\mathcal{X}}$ in \eqref{ExRefMeasure} at the beginning of Section~\ref{SectApproxNash}, and also the definition of the sets $\mathcal{Q}_{\nu,c}\subset \prbms{\mathcal{X}}$, $c\geq 0$, given there.

Let $\sigma$ be a $d\times d$-matrix, and let
\begin{align*}
	& \hat{b}\!: [0,T]\times \mathcal{X} \times \prbms{\mathcal{X}} \rightarrow \mathbb{R}^{d}, & &\tilde{b}_{N}\!: [0,T]\times \mathcal{X}^{N} \times \prbms{\mathcal{X}} \rightarrow \mathbb{R}^{d},& &N\in \mathbb{N},&
\end{align*}
be functions such that the following hold:
\begin{enumerate}
	\item[(ND)] Non-degeneracy: $\sigma$ is invertible.

	\item[(M)] Measurability: $\hat{b}$, $\tilde{b}_{N}$, $N\in \mathbb{N}$, are Borel measurable and progressive in the sense that, for all $t\in [0,T]$,
	\[
		\hat{b}(t,\phi,\theta) = \hat{b}(t,\tilde{\phi},\tilde{\theta}) \text{ whenever } \phi_{|[0,t]} = \tilde{\phi}_{|[0,t]} \text{ and } \theta_{|\mathcal{G}_{t}} = \tilde{\theta}_{|\mathcal{G}_{t}},
	\]
	and analogously for $\tilde{b}_{N}$.
	
	\item[(B)] Boundedness: there exists a finite constant $K > 0$ such that
	\[
		\|\hat{b}\|_{\infty} \vee \sup_{N\in \mathbb{N}} \|\tilde{b}_{N}\|_{\infty} \leq K.
	\]
\end{enumerate}

For each $N\in \mathbb{N}$, let $\nu_{N}\in \prbms{\mathbb{R}^{N\times d}}$ be symmetric with $\supp(\nu_{N}) \subset O^{N}$, as above. 
We assume, in addition to (ND), (M), and (B): 
\begin{enumerate}

	\item[(I)] Initial distributions: the sequence $(\nu_{N})_{N\in \mathbb{N}}$ is $\nu$-chaotic, where $\nu$ has support in $O$.
	
	\item[(C)] Almost continuity: for Lebesgue a.e.\ $t\in [0,T]$, every $\theta\in \mathcal{Q}_{\nu,K}$,
	\[
	\begin{split}
		\Theta_{\nu}\Bigl(\phi\in \mathcal{X}: \exists\, (\phi_{n},\theta_{n})\subset \mathcal{X}\times \prbms{\mathcal{X}} \text{ such that } \hat{b}(t,\phi_{n},\theta_{n}) \not\to \hat{b}(t,\phi,\theta) \\
		\text{ while } (\phi_{n},\theta_{n})\to (\phi,\theta) \Bigr) = 0.
	\end{split}
	\]
\end{enumerate}

In Section~\ref{AppUniqueness} of this Appendix, we will make the following assumption of partial Lipschitz continuity:
\begin{itemize}
	\item[(L)] Lipschitz continuity in the measure variable: There exists a finite constant $L > 0$ such that for all $t\in [0,T]$, all $\phi\in \mathcal{X}$,
	\[
		\left| \hat{b}(t,\phi,\theta) - \hat{b}(t,\phi,\tilde{\theta}) \right| \leq L\cdot \mathrm{d}_{t}(\theta,\tilde{\theta}) \text{ whenever } \theta, \tilde{\theta} \in \mathcal{Q}_{\nu,K},
	\]
\end{itemize}
where the distances $\mathrm{d}_{t}$ are pseudo-metrics derived from the total variation distance; see below.

\section{Convergence of empirical measures} \label{AppConvergence}

For $N\in \mathbb{N}$, consider the system of equations
\begin{equation} \label{EqAppPrelimitDynamics}
\begin{split}
	X^{N}_{1}(t) &= X^{N}_{1}(0) +  \int_{0}^{t} \tilde{b}_{N}\left(s,\boldsymbol{X}^{N},\mu^{N}\right)ds + \sigma W^{N}_{1}(t), \\
	X^{N}_{i}(t) &= X^{N}_{i}(0) +  \int_{0}^{t} \hat{b}\left(s,X^{N}_{i},\mu^{N}\right)ds + \sigma W^{N}_{i}(t), \\
	&i\in \{2,\ldots,N\},\; t\in [0,T],
\end{split}
\end{equation}
where $W^{N}_{1},\ldots,W^{N}_{N}$ are independent $d$-dimensional Wiener processes defined on some filtered probability space $(\Omega,\mathcal{F},(\mathcal{F}_{t}),\Prb)$, and $\mu^{N}$ is the empirical measure of the players' state trajectories, that is,   
\[
	\mu^{N}_{\omega}(\cdot)\doteq  \frac{1}{N} \sum_{j=1}^{N}  \delta_{X^{N}_{j}}(\cdot,\omega),\quad \omega\in \Omega.
\]

A solution of Eq.~\eqref{EqAppPrelimitDynamics} with initial distribution $\nu_{N}$ is given by a triple $((\Omega,\mathcal{F},(\mathcal{F}_{t}),\Prb),\boldsymbol{W}^{N},\boldsymbol{X}^{N})$ such that $(\Omega,\mathcal{F},(\mathcal{F}_{t}),\Prb)$ is a filtered probability space satisfying the usual hypotheses, $\boldsymbol{W}^{N} = (W^{N}_{1},\ldots,W^{N}_{N})$ a vector of independent $d$-dimensional $(\mathcal{F}_{t})$-Wiener processes, and $\boldsymbol{X}^{N} = (X^{N}_{1},\ldots,X^{N}_{N})$ a vector of continuous $\mathbb{R}^{d}$-valued $(\mathcal{F}_{t})$-adapted processes such that Eq.~\eqref{EqAppPrelimitDynamics} holds $\Prb$-almost surely with $\Prb\circ(\boldsymbol{X}^{N}(0))^{-1} = \nu_{N}$.

As in Section~\ref{SectPrelimitSystems}, existence and uniqueness in law of solutions to Eq.~\eqref{EqAppPrelimitDynamics} hold thanks to Girsanov's theorem and assumptions (ND), (M), and (B). Now, for each $N\in \mathbb{N}$, take a solution $((\Omega_{N},\mathcal{F}_{N},(\mathcal{F}^{N}_{t}),\Prb_{N}),\boldsymbol{W}^{N},\boldsymbol{X}^{N})$ of Eq.~\eqref{EqAppPrelimitDynamics} with initial distribution $\nu_{N}$, and let $\mu^{N}$ be the associated empirical measure on the path space $\mathcal{X}$.

\begin{lemma} \label{LemmaAppTightness}
	Grant (ND), (M), (B), (I), and (C). Then $(\Prb_{N}\circ(\mu^{N})^{-1})_{N\in \mathbb{N}}$ is tight in $\prbms{\prbms{\mathcal{X}}}$, and its limit points have support in $\mathcal{Q}_{\nu,K}$.
\end{lemma}

\begin{proof}
For $N\in \mathbb{N}$, let $\iota_{N}$ be the intensity measure of $\Prb_{N}\circ(\mu^{N})^{-1}$, that is,
\[
	\iota_{N}(A)\doteq \Mean_{N}\left[ \mu^{N}(A) \right],\quad A\in \Borel{\mathcal{X}}.
\]
Tightness of $(\Prb_{N}\circ(\mu^{N})^{-1})$ in $\prbms{\prbms{\mathcal{X}}}$ is then equivalent to the tightness of $(\iota_{N})$ in $\prbms{\mathcal{X}}$  \citep[cf.\ (2.5) in][p.\,178]{sznitman89}. By construction,
\[
	\iota_{N}(A) = \frac{1}{N} \sum_{i=1}^{N} \Prb_{N}\left( X^{N}_{i} \in A \right),\quad A\in \Borel{\mathcal{X}}.
\]
It is therefore enough to check that the family $(\Prb\circ(X^{N}_{i})^{-1})_{N\in \mathbb{N}, i\in \{1,\ldots,N\}}$ is tight. Now, for $N\in \mathbb{N}$, $i\in \{1,\ldots,N\}$,
\[
	\Prb_{N}\left( X^{N}_{i}(0) \in \cl(O) \right) = 1,
\]
and $\cl(O)$ is compact since $O$ is open and bounded. Moreover, thanks to assumption (B), for all $s,t \in [0,T]$,
\[
	\left|X^{N}_{i}(t) - X^{N}_{i}(s)\right| \leq K\cdot |t-s| + |\sigma|\cdot \left|W^{N}_{i}(t) - W^{N}_{i}(s) \right|,
\]
where we recall that $W^{N}_{i}$ is a standard Wiener process under $\Prb_{N}$. Tightness of $(\Prb_N \circ(X^{N}_{i})^{-1})$ is now a consequence of the Kolmogorov-Chentsov tightness criterion \citep[for instance, Corollary~16.9 in][p.\,313]{kallenberg01}.

\bigskip
As to the support of the limit points of $(\Prb_{N}\circ(\mu^{N})^{-1})$, we interpret the drift terms appearing in \eqref{EqAppPrelimitDynamics} as stochastic relaxed controls. To this end, set $\mathcal{R}\doteq \mathcal{R}_{B_{K}(0)}$, where $B_{K}(0) \subset \mathbb{R}^{d}$ is the closed ball of radius $K$ around the origin. Then $\mathcal{R}$ is compact (cf.\ Appendix~\ref{AppRelaxed}). For $N\in \mathbb{N}$, let $\rho^{N}_{i}$ be the $\mathcal{R}$-valued $(\mathcal{F}^{N}_{t})$-adapted random measure determined by
\begin{align*}
	\rho^{N}_{i,\omega}\bigl(B\times I\bigr)&\doteq \begin{cases}
	\int_{I}\delta_{\tilde{b}_{N}(t,\boldsymbol{X}^{N}(\cdot,\omega),\mu^{N}_{\omega})}(B)dt &\text{if } i = 1,\\
	\int_{I}\delta_{\hat{b}(t,X^{N}_{i}(\cdot,\omega),\mu^{N}_{\omega})}(B)dt &\text{if } i > 1,
	\end{cases} \\
	& B\in \Borel{\Gamma},\; I\in \Borel{[0,T]},\;\omega \in \Omega_{N}.
\end{align*}
We can rewrite Eq.~\eqref{EqAppPrelimitDynamics} in terms of $\rho^{N}_{1},\ldots,\rho^{N}_{N}$:
\begin{equation} \label{EqAppPrelimitDynamics2}
\begin{split}
	X^{N}_{i}(t) &= X^{N}_{i}(0) +  \int_{B_{K}(0)\times [0,t]} y\, \rho^{N}_{i}(dy,ds) + \sigma W^{N}_{i}(t),\\
	& i\in \{1,\ldots,N\},\; t\in [0,T].
\end{split}
\end{equation}
Now, form the extended empirical measure
\[
	Q^{N}_{\omega}\doteq \frac{1}{N} \sum_{i=1}^{N} \delta_{(X^{N}_{i}(\cdot,\omega),\rho^{N}_{i,\omega})},\quad \omega\in \Omega_{N}.
\]
Thus, $Q^{N}$ is a $\prbms{\mathcal{X}\times \mathcal{R}}$-valued random variable, and the projection on its first component coincides with $\mu^{N}$. The family $(\Prb_{N}\circ(Q^{N})^{-1})_{N\in \mathbb{N}}$ is tight in $\prbms{\prbms{\mathcal{X}\times \mathcal{R}}}$ thanks to the first part of the proof and the fact that $\mathcal{R}$ is compact.

With a slight abuse of notation, let $(\hat{X},\hat{\rho})$ denote the canonical process on $\mathcal{X}\times \mathcal{R}$. Let $(\Prb_{n}\circ(Q^{n})^{-1})_{n\in I}$ be a convergent subsequence of $(\Prb_{N}\circ(Q^{N})^{-1})_{N\in \mathbb{N}}$, and let $Q$ be a $\prbms{\mathcal{X}\times \mathcal{R}}$-valued random variable defined on some probability space $(\Omega,\mathcal{F},\Prb)$ such that
\[
	Q^{n} \to Q \text{ in distribution as } I\ni n\to\infty.
\]
We have to show that $Q_{\omega}\circ \hat{X}^{-1}\in \mathcal{Q}_{\nu,K}$ for $\Prb$-almost every $\omega\in \Omega$.
To this end, define a process $\hat{W}$ on $\mathcal{X}\times \mathcal{R}$ by
\[
	\hat{W}(t)\doteq \sigma^{-1}\left(\hat{X}(t) - \hat{X}(0) -  \int_{\bar{B}_{K}(0)\times [0,t]} y\, \hat{\rho}(dy,ds) \right),\quad t\in [0,T].
\]
By a martingale argument similar to that in the proof of Lemma~\ref{LemmaAppLimitSolutions}, but using Eq.~\eqref{EqAppPrelimitDynamics2}, one checks that $\hat{W}$ is a standard Wiener process under $Q_{\omega}$ for $\Prb$-almost every $\omega\in \Omega$. This entails that $\hat{X}$ solves Eq.~\eqref{EqBoundedControlDynamics} under $Q_{\omega}$ with $B_{K}(0)$-valued control process
\[
	v(t,(\phi,r))\doteq \int_{B_{K}(0)} y\, \dot{r}_{t}(dy),\quad t\in [0,T],\; (\phi,r)\in \mathcal{X}\times \mathcal{R}.
\]
It follows that $Q_{\omega}\circ \hat{X}^{-1}\in \mathcal{Q}_{\nu,K}$ for $\Prb$-almost every $\omega\in \Omega$.
\end{proof}

In order to further characterize the limit points of $(\Prb_{N}\circ(\mu^{N})^{-1})_{N\in \mathbb{N}}$, consider, for $\theta\in \prbms{\mathcal{X}}$, the equation
\begin{equation} \label{EqAppLimitDynamics}
	X(t) = X(0) +  \int_{0}^{t} \hat{b}\left(s,X,\theta\right)ds + \sigma W(t),\quad t\in [0,T],
\end{equation}
where $W$ is a $d$-dimensional Wiener process defined on some filtered probability space. A solution of Eq.~\eqref{EqAppLimitDynamics} with measure $\theta\in \prbms{\mathcal{X}}$ and initial distribution $\nu$ is a triple $((\Omega,\mathcal{F},(\mathcal{F}_{t}),\Prb),W,X)$ such that $(\Omega,\mathcal{F},(\mathcal{F}_{t}),\Prb)$ is a filtered probability space satisfying the usual hypotheses, $W$ a $d$-dimensional $(\mathcal{F}_{t})$-Wiener process, and $X$ an $\mathbb{R}^{d}$-valued $(\mathcal{F}_{t})$-adapted process such that \eqref{EqAppLimitDynamics} holds $\Prb$-almost surely with $\Prb\circ(X(0))^{-1} = \nu$ and drift coefficient $\hat{b}(\cdot,\cdot,\theta)$. Again thanks to Girsanov's theorem and assumptions (ND), (M), and (B), existence and uniqueness in law of solutions to Eq.~\eqref{EqAppLimitDynamics} hold for each fixed $\theta\in \prbms{\mathcal{X}}$.

Recall that $\hat{X}$ denotes the canonical process on $\mathcal{X}$.
\begin{defn} \label{DefAppMVSolution}
	A measure $\theta\in \prbms{\mathcal{X}}$ is called a \emph{McKean-Vlasov solution} of Eq.~\eqref{EqAppLimitDynamics} if  there exists a solution $((\Omega,\mathcal{F},(\mathcal{F}_{t}),\Prb),W,X)$ of Eq.~\eqref{EqAppLimitDynamics} with initial distribution $\theta\circ (\hat{X}(0))^{-1}$ and measure $\theta$ such that $\Prb\circ X^{-1} = \theta$.
\end{defn}

By uniqueness in law (with fixed measure $\theta$), if $\Prb\circ X^{-1} = \theta$ holds for one solution $((\Omega,\mathcal{F},(\mathcal{F}_{t}),\Prb),W,X)$ of Eq.~\eqref{EqAppLimitDynamics} with measure $\theta$ and initial distribution $\theta\circ (\hat{X}(0))^{-1}$, then it holds for any such solution of Eq.~\eqref{EqAppLimitDynamics}. According to the next lemma, limit points of the sequence of empirical measures $(\mu^{N})_{N\in \mathbb{N}}$ are almost surely concentrated on McKean-Vlasov solutions of Eq.~\eqref{EqAppLimitDynamics}. This yields, in particular, existence of McKean-Vlasov solutions. Those solutions do not necessarily have the same law.

\begin{lemma} \label{LemmaAppLimitSolutions}
	Grant (ND), (M), (B), (I), and (C). Let $(\Prb_{n}\circ(\mu^{n})^{-1})_{n\in I}$ be a convergent subsequence of the family $(\Prb_{N}\circ(\mu^{N})^{-1})_{N\in \mathbb{N}}$, and let $\mu$ be a $\prbms{\mathcal{X}}$-valued random variable defined on some probability space $(\Omega,\mathcal{F},\Prb)$ such that
	\[
		\mu^{n} \to \mu \text{ in distribution as } I\ni n\to\infty.
	\]
	Then $\mu_{\omega}$ is a McKean-Vlasov solution of Eq.~\eqref{EqAppLimitDynamics} for $\Prb$-almost every $\omega\in \Omega$.
\end{lemma}

\begin{proof}
Thanks to hypothesis (I), we have $\mu_{\omega}\circ (\hat{X}(0))^{-1} = \nu$ for $\Prb$-almost every $\omega\in \Omega$.

In the proof, we use the characterization of solutions to Eq.~\eqref{EqAppLimitDynamics} with fixed measure variable through a martingale problem in the sense of \citet{stroockvaradhan79}; also see \citet[Section 5.4]{karatzasshreve91}. Since the coefficients in \eqref{EqAppLimitDynamics} are bounded, we can employ a ``true'' martingale problem instead of a local martingale problem and work with test functions that have compact support. For a test function $g\in \mathbf{C}_{c}^{2}(\mathbb{R}^{d})$ and a measure $\theta\in \prbms{\mathcal{X}}$, define the process $M_{g}^{\theta}$ on $(\mathcal{X},\Borel{\mathcal{X}})$ by
\[
\begin{split}
	M_{g}^{\theta}(t,\phi)&\doteq g\bigl(\phi(t)\bigr) - g\bigl(\phi(0)\bigr) \\
	&- \int_{0}^{t} \left( \hat{b}(s,\phi,\theta)\cdot \nabla g + \frac{1}{2}\sum_{i,j=1}^{d} (\sigma\trans{\sigma})_{ij} \frac{\partial^{2} g}{\partial x_{i}\partial x_{j}} \right) \bigl(\phi(s)\bigr)ds.
\end{split}
\]

Recall that $(\mathcal{G}_{t})$ is the canonical filtration in $\Borel{\mathcal{X}}$ and that $\hat{X}$ is the coordinate process on $\mathcal{X}$. We have to show that, for $\Prb$-almost every $\omega\in \Omega$, $\mu_{\omega}$ solves the martingale problem associated with $\hat{b}(\cdot,\cdot,\mu_{\omega})$ and $\sigma\trans{\sigma}$, that is, for every test function $g\in \mathbf{C}_{c}^{2}(\mathbb{R}^{d})$, the process $M_{g}^{\mu_{\omega}}$ is a $(\mathcal{G}_{t})$-martingale under $\mu_{\omega}$. Although $(\mathcal{G}_{t})$ is a ``raw'' filtration (i.e., not necessarily right-continuous or $\mu_{\omega}$-complete), checking the martingale property with respect to $(\mathcal{G}_{t})$ is sufficient; see, for instance, Problem 5.4.13 in \citet[pp.\,318-319, 392]{karatzasshreve91}. Indeed, if $M_{g}^{\mu_{\omega}}$ is a $(\mathcal{G}_{t})$-martingale under $\mu_{\omega}$ for every $g\in \mathbf{C}_{c}^{2}(\mathbb{R}^{d})$, then $((\mathcal{X},\Borel{\mathcal{X}},(\mathcal{G}^{\mu_{\omega}}_{t+}),\mu_{\omega}),\hat{W},\hat{X})$ is a solution of Eq.~\eqref{EqAppLimitDynamics} with initial distribution $\nu$, where $(\mathcal{G}^{\mu_{\omega}}_{t+})$ indicates the right-continuous $\mu_{\omega}$-augmentation of $(\mathcal{G}_{t})$ and $\hat{W}$ is defined by
\[
	\hat{W}(t)\doteq \sigma^{-1}\left(\hat{X}(t) - \hat{X}(0) -  \int_{0}^{t} \hat{b}\left(s,\hat{X},\mu_{\omega}\right)ds \right),\quad t\in [0,T].
\]
Since $\mu_{\omega}\circ\hat{X}^{-1} = \mu_{\omega}$, it then follows that $\mu_{\omega}$ is a McKean-Vlasov solution of Eq.~\eqref{EqAppLimitDynamics} in the sense of Definition~\ref{DefAppMVSolution}. For this implication to hold, it suffices to take a countable collection of test functions $g\in \mathbf{C}_{c}^{2}(\mathbb{R}^{d})$ that approximate the $d$-variate monomials of first and second order.

Let $\theta\in \prbms{\mathcal{X}}$. The processes $M_{g}^{\theta}$ are, by construction and thanks to assumptions (B) and (M), bounded, measurable, and $(\mathcal{G}_{t})$-adapted.  The martingale property of $M_{g}^{\theta}$ is equivalent to having
\begin{equation} \label{EqAppProofMProp}
	\Mean_{\theta}\left[ \psi\cdot \left(M_{g}^{\theta}(t_{1}) - M_{g}^{\theta}(t_{0})\right) \right] = 0
\end{equation}
for every choice of $(t_{0},t_{1},\psi)\in [0,T]^{2}\times \mathbf{C}_{b}(\mathcal{X})$ such that $t_{0} \leq t_{1}$ and $\psi$ is $\mathcal{G}_{t_{0}}$-measurable. Since the $\sigma$-algebras $\mathcal{G}_{t}$ are countably generated, the processes $M_{g}^{\theta}$ have continuous trajectories, and since the test functions can be taken from a countable family, we can choose a countable collection of test parameters $\mathcal{T} \subset [0,T]^{2}\times \mathbf{C}_{b}(\mathcal{X})\times \mathbf{C}^{2}_{c}(\mathbb{R}^{d})$ with the following two properties: First, for every $(t_{0},t_{1},\psi,g)\in \mathcal{T}$ we have $t_{0} \leq t_{1}$ and $\psi$ is $\mathcal{G}_{t_{0}}$-measurable; second, if $\theta\in \prbms{\mathcal{X}}$ is such that \eqref{EqAppProofMProp} holds for every $(t_{0},t_{1},\psi,g)\in \mathcal{T}$, then $\theta$ is a McKean-Vlasov solution of Eq.~\eqref{EqAppLimitDynamics}. In the following three steps, we will show that there exists $\bar{\Omega}\in \mathcal{F}$ such that $\Prb(\bar{\Omega}) = 1$ and, for every $\omega\in \bar{\Omega}$, \eqref{EqAppProofMProp} with $\theta = \mu_{\omega}$ holds for all $(t_{0},t_{1},\psi,g)\in \mathcal{T}$.

\bigskip
\textbf{Step 1.}
Let $(t_{0},t_{1},\psi,g)\in \mathcal{T}$. Define a function $\Psi = \Psi_{(t_{0},t_{1},\psi,g)}\!: \prbms{\mathcal{X}} \rightarrow \mathbb{R}$ by
\[
	\Psi(\theta) = \Psi_{(t_{0},t_{1},\psi,g)}(\theta)\doteq \Mean_{\theta}\left[ \psi\cdot \left(M_{g}^{\theta}(t_{1}) - M_{g}^{\theta}(t_{0})\right) \right].
\]
Notice that $\Psi$ is well defined (the expectation on the right-hand side is finite), Borel measurable, and bounded. We claim that $\Psi$ is continuous at every $\theta\in \mathcal{Q}_{\nu,K}$. To see this, let $\theta\in \mathcal{Q}_{\nu,K}$ and $(\theta_{n})_{n\in \mathbb{N}} \subset \prbms{\mathcal{X}}$ be such that $\theta_{n}\to \theta$ as $n\to \infty$. Define bounded measurable functions $h_n , h : [0,T]\times \mathcal{X}\rightarrow \mathbb{R}$ according to
\begin{align*}
	h_{n}(s,\phi)&\doteq \psi(\phi)\cdot \hat{b}(s,\phi,\theta_{n})\cdot \left(\sum_{i=1}^{d} \frac{\partial g}{\partial x_{i}}\bigl(\phi(s)\bigr) \right), & n\in \mathbb{N},& \\
	h(s,\phi)&\doteq \psi(\phi)\cdot \hat{b}(s,\phi,\theta)\cdot \left(\sum_{i=1}^{d} \frac{\partial g}{\partial x_{i}}\bigl(\phi(s)\bigr) \right), &(s,\phi)\in [0,T]\times \mathcal{X}.&
\end{align*}
By hypothesis, $\theta_{n}\to \theta$ in the sense of weak convergence of probability measures. The functions $\psi$, $\sigma$ (constant) as well as $g$ together with its first and second order partial derivatives are all bounded and continuous. In order to establish the convergence of $\Psi(\theta_{n})$ to $\Psi(\theta)$ it is therefore enough to verify that
\[
	\int_{\mathcal{X}} \int_{t_{0}}^{t_{1}} h_{n}(s,\phi)ds\, \theta_{n}(d\phi) \stackrel{n\to\infty}{\longrightarrow} \int_{\mathcal{X}} \int_{t_{0}}^{t_{1}} h(s,\phi)ds\, \theta(d\phi).
\]
The functions $h$, $h_{n}$, $n\in \mathbb{N}$, are uniformly bounded. By dominated convergence and the Fubini-Tonelli theorem, it thus suffices to check that
\begin{equation} \label{EqAppProofInnerCont}
	\int_{\mathcal{X}} h_{n}(s,\phi)\, \theta_{n}(d\phi) \stackrel{n\to\infty}{\longrightarrow} \int_{\mathcal{X}} h(s,\phi)\, \theta(d\phi) \text{ for almost every } s\in [0,T].
\end{equation}
For $s\in [0,T]$, set
\[
	E_{s}\doteq \left\{ \phi\in \mathcal{X}\,:\, \exists (\phi_{n})\subset \mathcal{X}: h_{n}(s,\phi_{n}) \not\to h(s,\phi) \text{ while } \phi_{n} \to \phi\right\}.
\]
The functions $\psi$, $\frac{\partial g}{\partial x_{i}}$, $i\in \{1,\ldots,d\}$, are continuous and bounded. The choice of $\theta\in \mathcal{Q}_{\nu,K}$ entails that $\theta$ is absolutely continuous with respect to $\Theta_{\nu}$. By (C), the assumption of almost continuity, it follows that
\[
	\theta(E_{s}) = 0 \text{ for Lebesgue almost every } s\in [0,T].
\]
The extended mapping theorem \citep[Theorem~5.5 in][p.\,34]{billingsley68} now implies that \eqref{EqAppProofInnerCont} holds. It follows that $\Psi$ is continuous at $\theta\in \mathcal{Q}_{\nu,K}$.

\bigskip
\textbf{Step 2.}
Let again $(t_{0},t_{1},\psi,g)\in \mathcal{T}$. We are going to show that
\begin{equation} \label{EqAppProofVarZero}
	\Mean_{\Prb}\left[ \left( \Psi_{(t_{0},t_{1},\psi,g)}(\mu)\right)^{2} \right] = 0.
\end{equation}
Recall that $\Psi = \Psi_{(t_{0},t_{1},\psi,g)}$ is bounded and, by the previous step, continuous at every $\theta\in \mathcal{Q}_{\nu,K}$. By Lemma~\ref{LemmaAppTightness}, we have $\Prb\left(\mu \in \mathcal{Q}_{\nu,K}\right) = 1$. By hypothesis, $\mu^{n} \to \mu$ in distribution as $I\ni n \to \infty$. The mapping theorem \citep[Theorem~5.1 in][p.\,30]{billingsley68} thus implies that
\[
	\Mean_{\Prb}\left[ \left(\Psi(\mu)\right)^{2} \right] = \lim_{I\ni n\to\infty} \Mean_{\Prb_{n}}\left[ \left(\Psi(\mu^{n})\right)^{2} \right].
\]
Now, for every $n\in I$, using It{\^o}'s formula and Eq.~\eqref{EqAppPrelimitDynamics},
\begin{align*}
	&\Mean_{\Prb_{n}}\left[ \left(\Psi(\mu^{n})\right)^{2} \right] \\
	&= \Mean_{\Prb_{n}}\left[ \left(\frac{1}{n}\sum_{i=1}^{n} \psi(X^{n}_{i})\cdot \left(M_{g}^{\mu^{n}}(t_{1},X^{n}_{i}) - M_{g}^{\mu^{n}}(t_{0},X^{n}_{i})\right) \right)^{2} \right] \\
	\begin{split}
	&= \frac{1}{n^{2}} \Mean_{\Prb_{n}}\left[ \left(\psi(X^{n}_{1})\cdot \int_{t_{0}}^{t_{1}} \left(\tilde{b}_{n}(s,\boldsymbol{X}^{n},\mu^{n}) - \hat{b}(s,X^{n}_{1},\mu_{\omega}) \right)\cdot \nabla g\bigl(X^{n}_{1}(s)\bigr)ds \right.\right. \\
	&\qquad \left.\left. + \sum_{i=1}^{n} \psi(X^{n}_{i})\cdot \int_{t_{0}}^{t_{1}} \trans{\nabla g\bigl(X^{n}_{i}(s)\bigr)} \sigma dW^{n}_{i}(s) \right)^{2} \right].
	\end{split}
\end{align*}
The functions $\psi$, $\sigma$ (constant), $\nabla g$, $b$, and $\hat{b}_{n}$ are bounded, uniformly in $n\in I$; cf.\ assumption (B). Since $\psi$ is $\mathcal{G}_{t_{0}}$-measurable, the random variables $\psi(X^{n}_{1}),\ldots,\psi(X^{n}_{n})$ are $\mathcal{F}^{n}_{t_{0}}$-measurable. The Wiener processes $W^{n}_{1},\ldots,W^{n}_{n}$ are independent. Setting
\[
	\tilde{B}^{n}_{1}\doteq \int_{t_{0}}^{t_{1}} \left(\tilde{b}_{n}(s,\boldsymbol{X}^{n},\mu^{n}) - \hat{b}(s,X^{n}_{1},\mu_{\omega}) \right)\cdot \nabla g\bigl(X^{n}_{1}(s)\bigr)ds,
\]
we obtain
\begin{align*}
	\begin{split}
	&\Mean_{\Prb_{n}}\left[ \left(\Psi(\mu^{n})\right)^{2} \right] \\
	&= \frac{1}{n^{2}} \Mean_{\Prb_{n}}\left[ \left(\psi(X^{n}_{1})\cdot \tilde{B}^{n}_{1} \right)^{2} \right] \\
	&\quad + \frac{2}{n^{2}} \Mean_{\Prb_{n}}\left[ \sum_{i=1}^{n} \psi(X^{n}_{i})\cdot \psi(X^{n}_{1})\cdot \tilde{B}^{n}_{1} \cdot \int_{t_{0}}^{t_{1}} \trans{\nabla g\bigl(X^{n}_{i}(s)\bigr)} \sigma dW^{n}_{i}(s) \right] \\
	&\quad + \frac{1}{n^{2}} \Mean_{\Prb_{n}}\left[ \sum_{i=1}^{n} \left(\psi(X^{n}_{i})\right)^{2} \cdot \int_{t_{0}}^{t_{1}} \trans{\nabla g\bigl(X^{n}_{i}(s)\bigr)} \sigma\trans{\sigma}\nabla g\bigl(X^{n}_{i}(s)\bigr) ds \right]
	\end{split}\\
	&\stackrel{I\ni n\to \infty}{\longrightarrow} 0,
\end{align*}
where convergence to zero follows from the uniform boundedness of the integrands (and the Cauchy-Schwarz inequality).

\bigskip
\textbf{Step 3.}
By Step~2, we have that \eqref{EqAppProofVarZero} holds for every $\mathfrak{t} = (t_{0},t_{1},\psi,g)\in \mathcal{T}$. For every $\mathfrak{t}\in \mathcal{T}$, we can therefore select $A_{\mathfrak{t}}\in \mathcal{F}$ such that $\Prb(A_{\mathfrak{t}}) = 1$ and $\Psi(\mu_{\omega}) = 0$ for every $\omega\in A_{\mathfrak{t}}$. Set $\bar{\Omega}\doteq \bigcap_{\mathfrak{t}\in \mathcal{T}} A_{\mathfrak{t}}$. Then $\Prb(\bar{\Omega}) = 1$ since $\mathcal{T}$ is countable, and $\Psi(\mu_{\omega}) = 0$ for every $\omega\in \bar{\Omega}$. Hence, for $\Prb$-almost every $\omega\in \Omega$, $\mu_{\omega}$ is such that \eqref{EqAppProofMProp} with $\theta = \mu_{\omega}$ holds for every $(t_{0},t_{1},\psi,g)\in \mathcal{T}$.
\end{proof}


\section{Uniqueness of McKean-Vlasov solutions} \label{AppUniqueness}

Here, we consider McKean-Vlasov solutions of Eq.~\eqref{EqAppLimitDynamics} in the sense of Definition~\ref{DefAppMVSolution} and obtain a uniqueness result.

On $\prbms{\mathcal{X}}$ define the following distances derived from the total variation distance by setting, for $t\in [0,T]$,
\[
	\mathrm{d}_{t}(\theta,\tilde{\theta})\doteq \sup_{A\in \mathcal{G}_{t}} \left|\theta(A) - \tilde{\theta}(A) \right|,\quad \theta, \tilde{\theta}\in \prbms{\mathcal{X}}.
\]
Notice that $\mathrm{d}_{t}$ are pseudo-metrics for $t\in [0,T)$, while $\mathrm{d}_{T}$ is a true metric since $\mathcal{G}_{T} = \Borel{\mathcal{X}}$.

Let $K$ be the constant from assumption (B), $\nu\in \prbms{\mathbb{R}^{d}}$ the initial distribution. We will assume the following partial Lipschitz property with respect to the measure argument of $\hat{b}$:
\begin{itemize}
	\item[(L)] There exists a finite constant $L > 0$ such that for all $t\in [0,T]$, all $\phi\in \mathcal{X}$,
	\[
		\left| \hat{b}(t,\phi,\theta) - \hat{b}(t,\phi,\tilde{\theta}) \right| \leq L\cdot \mathrm{d}_{t}(\theta,\tilde{\theta}) \text{ whenever } \theta, \tilde{\theta} \in \mathcal{Q}_{\nu,K}.
	\]
\end{itemize}
Notice that condition (L) does not require $\hat{b}$ to be Lipschitz continuous or simply continuous with respect to the trajectory variable $\phi\in \mathcal{X}$. In fact, for the following uniqueness result to be valid, $\hat{b}$ need not satisfy (C), the assumption of almost continuity, either.

\begin{prop} \label{PropAppMVUnique}
Grant (ND), (M), (B), and (L). Then there exists at most one McKean-Vlasov solution of Eq.~\eqref{EqAppLimitDynamics} with initial distribution $\nu$.
\end{prop}

\begin{proof}
Suppose $\theta, \tilde{\theta}\in \prbms{\mathcal{X}}$ are McKean-Vlasov solutions of Eq.~\eqref{EqAppLimitDynamics} with initial distribution $\nu = \theta\circ (\hat{X}(0))^{-1} = \tilde{\theta}\circ (\hat{X}(0))^{-1}$. We have to show that $\theta = \tilde{\theta}$.

Observe that $\theta$ and $\tilde{\theta}$, being McKean-Vlasov solutions of Eq.~\eqref{EqAppLimitDynamics} with initial distribution $\nu$, are elements of $\mathcal{Q}_{\nu,K}$. In particular, the Lipschitz property expressed in (L) applies to $\theta, \tilde{\theta}$.

\bigskip
\textbf{Step 1.}
Set $\gamma\doteq \Theta_{\nu}$ and, for $t\in [0,T]$,
\begin{align*}
	\begin{split}
	Y(t)&\doteq \exp\left( \int_{0}^{t} \left(\sigma\trans{\sigma}\right)^{-1} \hat{b}(s,\hat{X},\theta) \cdot d\hat{X}(s) \right.\\
	&\quad \left.- \frac{1}{2} \int_{0}^{t} \trans{\hat{b}(s,\hat{X},\theta)} \left(\sigma\trans{\sigma}\right)^{-1} \hat{b}(s,\hat{X},\theta) ds \right),
	\end{split}\\
	\begin{split}
	\tilde{Y}(t)&\doteq \exp\left( \int_{0}^{t} \left(\sigma\trans{\sigma}\right)^{-1} \hat{b}(s,\hat{X},\tilde{\theta}) \cdot d\hat{X}(s) \right.\\
	&\quad \left.- \frac{1}{2} \int_{0}^{t} \trans{\hat{b}(s,\hat{X},\tilde{\theta})} \left(\sigma\trans{\sigma}\right)^{-1} \hat{b}(s,\hat{X},\tilde{\theta})ds \right).
	\end{split}
\end{align*}

Then $Y$, $\tilde{Y}$ are well defined under $\gamma$ and, for every $t\in [0,T]$,
\begin{align} \label{EqAppDensities}
	& \frac{d\theta}{d\gamma}\Big | _{\mathcal{G}_{t}} = Y(t),& & \frac{d\tilde{\theta}}{d\gamma}\Big |_{\mathcal{G}_{t}} = \tilde{Y}(t).&
\end{align}

The equalities in \eqref{EqAppDensities} are a consequence of the Cameron-Martin-Girsanov formula; see Theorem~6.4.2 in \citet[p.\,154]{stroockvaradhan79}. For the sake of completeness, we derive \eqref{EqAppDensities} from a standard version of Girsanov's formula. First observe that $Y$, $\tilde{Y}$ are well defined if the stochastic integral that appears in the exponential is well defined. This is the case if we take $\gamma$ as our reference probability measure since $\hat{X}$ is a vector of continuous martingales under $\gamma$, while the integrands $\hat{b}(s,\hat{X},\theta)$ and $\hat{b}(s,\hat{X},\tilde{\theta})$, respectively, are bounded progressively measurable processes with respect to $(\mathcal{G}_{t})$ thanks to assumptions (M) and (B).	
	
As to the densities given in \eqref{EqAppDensities}, it is enough to prove the assertion for $\theta$, the proof for $\tilde{\theta}$ being completely analogous. Let $((\Omega,\mathcal{F},(\mathcal{F}_{t}),\Prb),W,X)$ be a solution of Eq.~\eqref{EqAppLimitDynamics} with initial distribution $\nu$ and measure $\theta$ such that $\Prb\circ X^{-1} = \theta$; such a solution exists by hypothesis.

Define a process $Z$ on $(\Omega,\mathcal{F})$ by setting
\[
	Z(t)\doteq e^{-\int_{0}^{t} \sigma^{-1}\hat{b}(s,X,\theta)\cdot dW(s) - \frac{1}{2} \int_{0}^{t} \left| \sigma^{-1}\hat{b}(s,X,\theta) \right|^{2}ds },\quad t\in [0,T].
\]
Then $Z$ is a martingale with respect to $\Prb$ and $(\mathcal{F}_{t})$, and
\[
	Q(A)\doteq \Mean_{\Prb}\left[Z(T) \mathbf{1}_{A}\right],\quad A\in \mathcal{F},
\]
defines a probability measure such that
\[
	\frac{dQ}{d\Prb}\Big |_{\mathcal{F}_{t}} = Z(t) \text{ for every }t\in [0,T].
\]
By Girsanov's Theorem \citep[Theorem~3.5.1][p.\,191]{karatzasshreve91} and Eq.~\eqref{EqAppLimitDynamics}, we have that
\[
	\sigma^{-1}\left( X(\cdot)-X(0)\right) = W(\cdot) + \int_{0}^{\cdot} \sigma^{-1}\hat{b}(s,X,\theta)ds
\]
is an $(\mathcal{F}_{t})$-Wiener process under $Q$. Thus, since $X(0)$ is $\mathcal{F}_{0}$-measurable, $X(0)$ and $X(\cdot)-X(0)$ are independent under $Q$. It follows that
\[
	Q\circ X^{-1} = \Prb\circ\left(X(0) + \sigma W\right)^{-1} = \gamma.
\]

Using again Eq.~\eqref{EqAppLimitDynamics}, we obtain with probability one under $\Prb$ as well as under $Q$ that for all $t\in [0,T]$,
\begin{align*}
	Z(t) &= e^{-\int_{0}^{t} \left(\sigma\trans{\sigma}\right)^{-1}\hat{b}(s,X,\theta)\cdot dX(s) + \frac{1}{2} \int_{0}^{t} \left|\sigma^{-1}\hat{b}(s,X,\theta)\right|^{2}ds }, \\
	\frac{1}{Z(t)} &= e^{ \int_{0}^{t} \left(\sigma\trans{\sigma}\right)^{-1}\hat{b}(s,X,\theta)\cdot dX(s) - \frac{1}{2} \int_{0}^{t} \trans{\hat{b}(s,X,\theta)} \left(\sigma\trans{\sigma}\right)^{-1} \hat{b}(s,X,\theta)ds }.
\end{align*} 
The process $X$ is a vector of continuous semimartingales with respect to $\Prb$ as well as $Q$, with quadratic covariation processes
\[
	\left\langle X_{i},X_{j} \right\rangle (t)
	= t\cdot \left(\sigma\trans{\sigma}\right)_{ij},\quad t\in [0,T], \text{ $\Prb$-/$Q$-almost surely.}
\]
It follows that $1/Z$ is a stochastic exponential, hence a local martingale, under $Q$. Since $\hat{b}$ is bounded, $1/Z$ is a true martingale under $Q$. As a consequence,
\[
	\frac{d\Prb}{dQ}\Big |_{\mathcal{F}_{t}} = \frac{1}{Z(t)} \text{ for every }t\in [0,T].
\]

Recall that $Q\circ X^{-1} = \gamma$. Comparing the expressions for $1/Z$ and $Y$, we find that
\[
	Q\circ \left(\frac{1}{Z(t)}, X \right)^{-1} = \gamma\circ \left(Y(t), \hat{X} \right)^{-1} \text{ for all } t\in [0,T].
\]
Since $\theta = \Prb\circ X^{-1}$, it follows that, for every $B\in \mathcal{G}_{t}$,
\[
	\theta(B) = \Mean_{\Prb}\left[ \mathbf{1}_{B}(X)\right] = \Mean_{Q}\left[ \frac{1}{Z(t)}\mathbf{1}_{B}(X)\right] = \Mean_{\gamma}\left[ Y(t) \mathbf{1}_{B}(\hat{X})\right],
\]
hence $\frac{d\theta}{d\gamma}\big |_{\mathcal{G}_{t}} = Y(t)$ for all $t\in [0,T]$.

\bigskip
\textbf{Step 2.}
We are going to show that there exists a constant $C\in (0,\infty)$ depending only on $T$, $K$, and $\sigma$, such that for every bounded and progressively measurable functional $\psi\!: [0,T]\times \mathcal{X} \rightarrow \mathbb{R}$, every $t\in [0,T]$,
\begin{equation*}
	\left| \Mean_{\theta}\left[ \psi(t,\hat{X})\right] - \Mean_{\tilde{\theta}}\left[ \psi(t,\hat{X})\right] \right|^{2} \leq C\cdot \|\psi\|_{\infty}^{2} \int_{0}^{t} \mathrm{d}_{s}(\theta,\tilde{\theta})^{2} ds.
\end{equation*}

Indeed, by \eqref{EqAppDensities}, for every $t\in [0,T]$,
\[
	\Mean_{\theta}\left[ \psi(t,\hat{X})\right] - \Mean_{\tilde{\theta}}\left[ \psi(t,\hat{X})\right] = \Mean_{\gamma}\left[ \left(Y(t) - \tilde{Y}(t)\right) \psi(t,\hat{X})\right]
\]
Under $\gamma$, $\hat{X}$ is a martingale with quadratic covariation processes
\[
	\left\langle X_{i},X_{j} \right\rangle (t)
	= t\cdot \left(\sigma\trans{\sigma}\right)_{ij},\quad t\in [0,T],
\]
while $Y$, $\tilde{Y}$ are stochastic exponentials. In fact,
setting
\begin{align*}
	M(t)&\doteq \int_{0}^{t} \left(\sigma\trans{\sigma}\right)^{-1} \hat{b}(s,\hat{X},\theta)\cdot d\hat{X}(s), \\
	\tilde{M}(t)&\doteq \int_{0}^{t} \left(\sigma\trans{\sigma}\right)^{-1} \hat{b}(s,\hat{X},\tilde{\theta})\cdot d\hat{X}(s),
\end{align*}
we have, with probability one under $\gamma$,
\begin{align*}
	& Y(t)= \exp\left( M(t) - \frac{1}{2}\langle M\rangle(t)\right), & 	& \tilde{Y}(t)= \exp\left( \tilde{M}(t) - \frac{1}{2}\langle \tilde{M}\rangle(t)\right). &
\end{align*}
Therefore (by It{\^{o}}'s formula), with probability one under $\gamma$,
\begin{align*}
	& Y(t)= 1 + \int_{0}^{t} Y(s)dM(s), &
	& \tilde{Y}(t)= 1 + \int_{0}^{t} \tilde{Y}(s)d\tilde{M}(s), & &t\in [0,T].&
\end{align*}
	
The invertibility of $\sigma$ and the boundedness assumption (B) imply that
\begin{equation} \label{EqAppUniqueSecondMoment}
	\sup_{t\in [0,T]} \left\{ \Mean_{\gamma}\left[ \left|Y(t)\right|^{2} \right] \vee \Mean_{\gamma}\left[ \left|\tilde{Y}(t) \right|^{2} \right] \right\} \leq e^{T\cdot K^{2} \left\| \sigma^{-1} \right\|^{2}} \doteq e^{T C_{0}},
\end{equation}
where $\|\cdot\|$ denotes the Hilbert-Schmidt matrix norm. The bound \eqref{EqAppUniqueSecondMoment} holds since
\[
	Y(t)^{2} = \underbrace{\exp\left( 2M(t) - \frac{1}{2}\langle 2M\rangle(t) \right)}_{\text{stochastic exponential under }\gamma} \exp\left( \langle M\rangle(t) \right),
\]
while $\sup_{t\in[0,T]}|\langle M\rangle(t)| \leq T\cdot C_{0}$ $\gamma$-almost surely by the Cauchy-Schwarz inequality; analogously for the tilde part.

Using It{\^o}'s isometry, \eqref{EqAppUniqueSecondMoment}, and assumption (L), we obtain, for every $t\in [0,T]$,
\begin{align*}
	&\Mean_{\gamma}\left[ \left|Y(t) - \tilde{Y}(t) \right|^{2} \right] \\
	&= \Mean_{\gamma}\left[ \left| \int_{0}^{t} \left(\sigma\trans{\sigma}\right)^{-1} \left(Y(s) \hat{b}(s,\hat{X},\theta) - \tilde{Y}(s) \hat{b}(s,\hat{X},\tilde{\theta})\right) \cdot d\hat{X}(s) \right|^{2} \right] \\
	&\leq \int_{0}^{t} \Mean_{\gamma}\left[ \left\| \sigma^{-1} \right\|^{2} \left|Y(s) \hat{b}(s,\hat{X},\theta) - \tilde{Y}(s) \hat{b}(s,\hat{X},\tilde{\theta}) \right|^{2} \right]ds \\
	&\leq 2\left\| \sigma^{-1} \right\|^{2} L^{2} e^{T C_{0}} \int_{0}^{t} \mathrm{d}_{s}(\theta,\tilde{\theta})^{2} ds + 2C_{0} \int_{0}^{t} \Mean_{\gamma}\left[ \left|Y(s) - \tilde{Y}(s) \right|^{2} \right]ds,
\end{align*}
hence, thanks to Gronwall's lemma,
\[
	\Mean_{\gamma}\left[ \left|Y(t) - \tilde{Y}(t) \right|^{2} \right] \leq 2\left\| \sigma^{-1} \right\|^{2} L^{2} e^{3T C_{0}} \int_{0}^{t} \mathrm{d}_{s}(\theta,\tilde{\theta})^{2} ds,\quad t\in [0,T].
\]
It follows that, for every $t\in [0,T]$, 
\begin{multline*}
	\left| \Mean_{\theta}\left[ \psi(t,\hat{X})\right] - \Mean_{\tilde{\theta}}\left[ \psi(t,\hat{X})\right] \right|^{2} \\
	\leq \Mean_{\gamma}\left[ \left|Y(t) - \tilde{Y}(t)\right|^{2} |\psi(t,\hat{X})|^{2} \right] \\
	\leq \|\psi\|_{\infty}^{2}\cdot  2\left\| \sigma^{-1} \right\|^{2} L^{2} e^{3T C_{0}} \int_{0}^{t} \mathrm{d}_{s}(\theta,\tilde{\theta})^{2} ds.
\end{multline*}

\bigskip
\textbf{Step 3.}
For $t\in [0,T]$, $A\in \mathcal{G}_{t}$, define $\psi_{(A,t)}\!: [0,T]\times \mathcal{X} \rightarrow \mathbb{R}$ through
\[
	\psi_{(A,t)}(s,\phi)\doteq \begin{cases}
	\mathbf{1}_{A}(\phi) &\text{if } s \geq t, \\
	0 &\text{otherwise.}
	\end{cases}
\]
Then $\psi_{(A,t)}$ is bounded and progressively measurable with $\|\psi_{(A,t)}\|_{\infty} = 1$. By Step 2, there exists a finite constant $C$, depending only on $T$, $K$, $\sigma$, such that for every $t\in [0,T]$,
\begin{multline*}
	\sup_{A\in \mathcal{G}_{t}}\left| \theta(A) - \tilde{\theta}(A)\right|^{2} \\
	= \sup_{A\in \mathcal{G}_{t}} \left| \Mean_{\theta}\left[ \psi_{(A,t)}(t,\hat{X})\right] - \Mean_{\tilde{\theta}}\left[ \psi_{(A,t)}(t,\hat{X})\right] \right|^{2} \\
	\leq C\cdot \int_{0}^{t} \mathrm{d}_{s}(\theta,\tilde{\theta})^{2} ds.
\end{multline*}
By the definition of the total variation semi-distances, it follows that
\[
	\mathrm{d}_{t}(\theta,\tilde{\theta})^{2} \leq C\cdot \int_{0}^{t} \mathrm{d}_{s}(\theta,\tilde{\theta})^{2} ds \text{ for all }t\in [0,T],
\]
hence, thanks to Gronwall's lemma, $\mathrm{d}_{T}(\theta,\tilde{\theta}) = 0$. Since $\mathrm{d}_{T}$ is a true metric, we obtain $\theta = \tilde{\theta}$.
\end{proof}


\section{Regularity results} \label{AppRegularity}

Here, we collect some (well known) regularity results on the exit time $\tau^{\hat{X}}$ with respect to measures in $\mathcal{Q}_{\nu,K}$. Recall that $\tau^{\hat{X}}$ is the time of first exit from $O$ on path space:
\[
	\tau^{\hat{X}}(\phi) = \tau(\phi) \doteq \inf\left\{ t\geq 0 : \phi(t)\notin O\right\},\quad \phi\in \mathcal{X},
\]
where $O$ is a bounded open set satisfying \Hypref{HypStateSpace}.

\begin{lemma} \label{LemmaAppRegularityCompact}
Let $K \geq 0$. Then $\mathcal{Q}_{\nu,K}$ is compact in $\prbms{\mathcal{X}}$.
\end{lemma}

\begin{proof}
Recall that $\mathcal{Q}_{\nu,K}$ is the set of all laws of processes of the form
\[
	X(t) = \xi + \int_{0}^{t}v(s)ds + \sigma W(t),\quad t\in [0,T],
\]
where $W$ is an $\mathbb{R}^{d}$-valued $(\mathcal{F}_{t})$-Wiener process defined on some filtered probability space $(\Omega,\mathcal{F},(\mathcal{F}_{t}),\Prb)$, $\xi$ is an $\mathbb{R}^{d}$-valued $\mathcal{F}_{0}$-measurable random variable with distribution $\Prb\circ \xi^{-1} = \nu$, and $v$ is an $\mathbb{R}^{d}$-valued $(\mathcal{F}_{t})$-progressively measurable bounded process with $\|v\|_{\infty} \leq K$. By the boundedness of the control processes $v$ and the Kolmogorov-Chentsov tightness criterion \citep[for instance, Corollary~16.9 in][p.\,313]{kallenberg01}, we have that $\mathcal{Q}_{\nu,K}$ is tight, hence precompact in $\prbms{\mathcal{X}}$. By interpreting the control processes $v$ as relaxed controls, using arguments analogous to those of the second part of the proof of Lemma~\ref{LemmaAppTightness}, one checks that $\mathcal{Q}_{\nu,K}$ coincides with its own closure. It follows that $\mathcal{Q}_{\nu,K}$ is compact.
\end{proof}

\begin{lemma} \label{LemmaAppRegularity}
Let $K > 0$. Any measure $\theta\in \mathcal{Q}_{\nu,K}$ is equivalent to $\Theta_{\nu}$. Moreover, there exists a strictly positive constant $c_{0}$ depending only on $d$, $\nu$, $T$, and $K$, such that
\[
	\inf_{\theta\in \mathcal{Q}_{\nu,K}} \Mean_{\theta}\left[ \mathbf{1}_{[0,\tau^{\hat{X}})}(T) \right] \geq c_{0} >0.
\]
\end{lemma}

\begin{proof}
Let $\theta\in \mathcal{Q}_{\nu,K}$. Let $(\Omega,\mathcal{F},(\mathcal{F}_{t}),\Prb)$ be a filtered probability space carrying an $\mathbb{R}^{d}$-valued $(\mathcal{F}_{t})$-Wiener process $W$, an $\mathbb{R}^{d}$-valued $\mathcal{F}_{0}$-measurable random variable $\xi$ with $\Prb\circ \xi^{-1} = \nu$, and an $\mathbb{R}^{d}$-valued $(\mathcal{F}_{t})$-progressively measurable bounded process $v$ with $\|v\|_{\infty} \leq K$ such that the process
\[
	X(t) \doteq \xi + \int_{0}^{t}v(s)ds + \sigma W(t),\quad t\in [0,T],
\]
has law $\Prb\circ X^{-1} = \theta$. Set
\[
	B(t)\doteq \xi + \sigma W(t),\quad t\in [0,T].
\]

Define a process $Z$ on $(\Omega,\mathcal{F})$ by setting
\[
	Z(t)\doteq e^{-\int_{0}^{t} \sigma^{-1}v(s)\cdot dW(s) - \frac{1}{2} \int_{0}^{t} \left| \sigma^{-1}v(s) \right|^{2}ds },\quad t\in [0,T].
\]
Then $Z$ is a martingale with respect to $\Prb$ and $(\mathcal{F}_{t})$, and
\[
	Q(A)\doteq \Mean_{\Prb}\left[Z(T) \mathbf{1}_{A}\right],\quad A\in \mathcal{F},
\]
defines a probability measure such that
\[
	\frac{dQ}{d\Prb}_{|\mathcal{F}_{t}} = Z(t) \text{ for every }t\in [0,T].
\]
By using Girsanov's theorem as in the first step of the proof of Proposition~\ref{PropAppMVUnique} and thanks to the boundedness of $v$, we find that
\[
	\tilde{W}\doteq \sigma^{-1}\left(X(\cdot) - \xi \right) = W(\cdot) + \int_{0}^{\cdot} \sigma^{-1}v(s)ds
\]
is an $(\mathcal{F}_{t})$-Wiener process under $Q$,
\[
	Q\circ X^{-1} = \Prb\circ \left(\xi + \sigma W \right)^{-1} = \Prb\circ B^{-1},
\]
and
\[
	\frac{d\Prb}{dQ}_{|\mathcal{F}_{t}} = \frac{1}{Z(t)} \text{ for every }t\in [0,T].
\]

Since a $d$-dimensional Brownian motion stays in an open ball for any fixed finite time with strictly positive probability, there exists a constant $c_{0} > 0$ depending only on $d$, $\nu$, and $T$, such that
\[
	\inf_{t\in [0,T]} \Mean_{\Prb}\left[ \mathbf{1}_{[0,\tau^{B})}(t) \right] = \Mean_{\Prb}\left[ \mathbf{1}_{[0,\tau^{B})}(T) \right] \geq c_{0} > 0.
\]
Using H{\"o}lder's inequality, we find that 
\begin{multline*}
	\Mean_{\theta}\left[ \mathbf{1}_{[0,\tau^{\hat{X}})}(T) \right]
	= \Mean_{\Prb}\left[ \mathbf{1}_{[0,\tau^{X})}(T) \right]
	= \Mean_{Q}\left[ \frac{1}{Z(T)}\cdot \mathbf{1}_{[0,\tau^{X})}(T) \right] \\
	\geq \Mean_{Q}\left[ \sqrt{\mathbf{1}_{[0,\tau^{X})}(T)} \right]^{2}  \Mean_{Q}\left[ Z(T) \right]^{-1}
	\\
	= \Mean_{P}\left[ \mathbf{1}_{[0,\tau^{B})}(T) \right]^{2}  \Mean_{Q}\left[ Z(T) \right]^{-1} \geq \frac{c_{0}^{2}}{\Mean_{Q}\left[ Z(T) \right]}.
\end{multline*}

Now,
\[
	Z(T) = \underbrace{e^{-\int_{0}^{T} \sigma^{-1}v(s)\cdot d\tilde{W}(s) - \frac{1}{2} \int_{0}^{T} \left| \sigma^{-1}v(s) \right|^{2}ds} }_{\text{corresponds to a $Q$-stochastic exponential}} \cdot e^{\int_{0}^{T} \left| \sigma^{-1}v(s) \right|^{2}ds},
\]
hence
\[
	\frac{1}{\Mean_{Q}\left[ Z(T) \right]} \geq e^{-T\cdot \|\sigma^{-1}\|^{2}\cdot \|v\|_{\infty}^{2}}.
\]
\end{proof}

\begin{lemma} \label{LemmaAppRegularity2}
Let $K > 0$, and  let $\theta\in \mathcal{Q}_{\nu,K}$. Then the following hold:
\begin{enumerate}[(a)]

	\item \label{LemmaAppRegularity2Finite} $\tau^{\hat{X}} < \infty$ $\theta$-almost surely;
	
	\item \label{LemmaAppRegularity2Cont} the mapping $\mathcal{X}\ni \phi \mapsto \tau^{\hat{X}}(\phi) \in [0,\infty]$ is continuous $\theta$-almost surely;
	
	\item \label{LemmaAppRegularity2StochCont} $\theta\left(\tau^{\hat{X}} = t \right) = 0$ for every $t\geq 0$;
	
	\item \label{LemmaAppRegularity2IndCont} the mapping $\mathcal{X}\ni \phi \mapsto \mathbf{1}_{[0,\tau^{\hat{X}}(\phi))}(t) \in \mathbb{R}$ is continuous $\theta$-almost surely for every $t\geq 0$;
	
\end{enumerate}
\end{lemma}

\begin{proof}
Since $\theta$ is equivalent to $\Theta_{\nu}$ by Lemma~\ref{LemmaAppRegularity}, it is enough to check the above properties for $\theta = \Theta_{\nu}$.
	
Property~\eqref{LemmaAppRegularity2Finite} is a consequence of the law of the iterated logarithm (as time tends to infinity), the non-degeneracy of $\sigma$, and the fact that $O$ is bounded. Property~\eqref{LemmaAppRegularity2Cont} follows again from the law of the iterated logarithm, now as time goes to zero, the non-degeneracy of $\sigma$ and the fact that $O$ is open with smooth boundary; cf., for instance, \citet[pp.\,260-261]{kushnerdupuis01}.
	
Property \eqref{LemmaAppRegularity2IndCont} is a consequence of properties \eqref{LemmaAppRegularity2Cont} and \eqref{LemmaAppRegularity2StochCont}. As to \eqref{LemmaAppRegularity2StochCont}, by continuity of trajectories,
\[
	\Theta_{\nu}\left(\tau^{\hat{X}} = t \right) \leq \Theta_{\nu}\left( \hat{X}(t)\in \partial O\right) \text{ for every } t \geq 0.
\]
The boundary $\partial O$ has Lebesgue measure zero as it is bounded and of class $\boldsymbol{C}^{2}$ by hypothesis. By the non-degeneracy of $\sigma$, $\hat{X}(t)$ is absolutely continuous w.r.t.\ Lebesgue measure, hence $\Theta_{\nu}( \hat{X}(t)\in \partial O) = 0$. It follows that $\Theta_{\nu}(\tau^{\hat{X}} = t ) = 0$.

\end{proof}


\section{Assumptions (C) and (L)} \label{AppAssumptionsCL}

Throughout this section, we assume that $\hat{b}$ is defined by \eqref{ExProofNashDrift}, that is, for $(t,\phi,\theta)\in [0,T]\times \mathcal{X}\times \prbms{\mathcal{X}}$,
\[
	\hat{b}(t,\phi,\theta) = \begin{cases}
	u(t,\phi) + \bar{b}\left(t, \phi(t), \frac{\int w(\tilde{\phi}(t)) \mathbf{1}_{[0,\tau(\tilde{\phi}))}(t) \theta(d\tilde{\phi})}{\int \mathbf{1}_{[0,\tau(\tilde{\phi}))}(t) \theta(d\tilde{\phi})} \right)
	&\text{if } \theta(\tau > t) > 0,\\
	u(t,\phi) + \bar{b}\left(t,\phi(t),w(0)\right) &\text{if } \theta(\tau > t) = 0,
\end{cases}
\]
where $w$ is bounded continuous according to hypotheses \Hypref{HypMeasBounded} and \Hypref{HypCont}, $\bar{b}$ is bounded measurable and Lipschitz in its third variable according to \Hypref{HypMeasBounded} and \Hypref{HypLipschitz}, $\bar{b}(t,\cdot,\cdot)$ is continuous uniformly in $t\in [0,T]$ thanks to \Hypref{HypCont}, and $u$ is a feedback strategy such that, for Lebesgue-almost every $t\in [0,T]$,
\[
	\Theta_{\nu} \left(\left\{ \phi\in \mathcal{X} \;:\; u(t,\cdot) \text{ is discontinuous at }\phi \right\} \right) = 0.
\]
We are going to show that $\hat{b}$ thus defined satisfies conditions (C) and (L) above.

To establish condition (C), choose $K \geq \|\hat{b}\|_{\infty}$. We have to show that for Lebesgue almost every $t\in [0,T]$, every $\theta\in \mathcal{Q}_{\nu,K}$,
\begin{equation} \label{EqConditionC}
\begin{split}
		\Theta_{\nu}\Bigl(\phi\in \mathcal{X}: \exists\, (\phi_{n},\theta_{n})\subset \mathcal{X}\times \prbms{\mathcal{X}} \text{ s.th.\ } \hat{b}(t,\phi_{n},\theta_{n}) \not\to \hat{b}(t,\phi,\theta)\\ \text{ while } (\phi_{n},\theta_{n})\to (\phi,\theta) \Bigr) = 0.
	\end{split}
\end{equation}
Let $t\in [0,T]$ be such that $u(t,\cdot)$ is $\Theta_{\nu}$-almost surely continuous; this is true for Lebesgue almost every $t\in [0,T]$ by the continuity assumption on $u$. Let $\theta\in \mathcal{Q}_{\nu,K}$. Using Part~\eqref{LemmaAppRegularity2IndCont} of Lemma~\ref{LemmaAppRegularity2}, we find $A_{t,\theta}\in \Borel{\mathcal{X}}$ such that $\Theta_{\nu}(A_{t,\theta}) = 1$ and, for every $\phi\in A_{t,\theta}$, the mappings $u(t,\cdot)$ and $\mathbf{1}_{[0,\tau^{\hat{X}}(\cdot))}(t)$ are continuous at $\phi$. In view of Part~ \eqref{LemmaAppRegularity2Cont} of Lemma~\ref{LemmaAppRegularity2}, one can choose $A_{t,\theta}$ such that also the mapping $\tau^{\hat{X}}(\cdot)$ is continuous on $A_{t,\theta}$. Since $\Theta_{\nu}$ and $\theta$ are equivalent, we have $\theta(A_{t,\theta}) = 1$. Now, let $\phi\in A_{t,\theta}$, and let $(\phi_{n},\theta_{n})\subset \mathcal{X}\times \prbms{\mathcal{X}}$ be such that $(\phi_{n},\theta_{n})\to (\phi,\theta)$ as $n\to\infty$. Then, by the mapping theorem and since $\mathbf{1}_{[0,\tau^{\hat{X}}(\cdot))}(t)$ is $\theta$-almost surely continuous,
\[
	\int_{\mathcal{X}} \mathbf{1}_{[0,\tau(\tilde{\phi}))}(t) \theta_{n}(d\tilde{\phi}) \stackrel{n\to\infty}{\longrightarrow} \int_{\mathcal{X}} \mathbf{1}_{[0,\tau(\tilde{\phi}))}(t) \theta(d\tilde{\phi}).
\]
Since $w$ is bounded and continuous, we also have
\[
	\int_{\mathcal{X}} w(\tilde{\phi}(t))\cdot \mathbf{1}_{[0,\tau(\tilde{\phi}))}(t) \theta_{n}(d\tilde{\phi}) \stackrel{n\to\infty}{\longrightarrow} \int_{\mathcal{X}} w(\tilde{\phi}(t))\cdot \mathbf{1}_{[0,\tau(\tilde{\phi}))}(t) \theta(d\tilde{\phi}).
\]
In view of Lemma~\ref{LemmaAppRegularity},
\[
	\int_{\mathcal{X}} \mathbf{1}_{[0,\tau(\tilde{\phi}))}(t) \theta(d\tilde{\phi}) = \Mean_{\theta}\left[ \mathbf{1}_{[0,\tau^{\hat{X}})}(t) \right] = \theta\left( \tau > t \right) > 0.
\]
Since $\bar{b}(t,\cdot,\cdot)$ is bounded and continuous, it follows that
\begin{multline*}
	\hat{b}(t,\phi_{n},\theta_{n}) = u(t,\phi_{n}) + \bar{b}\left(t, \phi(t), \frac{\int w(\tilde{\phi}(t)) \mathbf{1}_{[0,\tau(\tilde{\phi}))}(t) \theta_{n}(d\tilde{\phi})}{\int \mathbf{1}_{[0,\tau(\tilde{\phi}))}(t) \theta_{n}(d\tilde{\phi})} \right) \\
	\stackrel{n\to\infty}{\longrightarrow} u(t,\phi) + \bar{b}\left(t, \phi(t), \frac{\int w(\tilde{\phi}(t)) \mathbf{1}_{[0,\tau(\tilde{\phi}))}(t) \theta(d\tilde{\phi})}{\int \mathbf{1}_{[0,\tau(\tilde{\phi}))}(t) \theta(d\tilde{\phi})} \right) = \hat{b}(t,\phi,\theta),
\end{multline*}
and we conclude that \eqref{EqConditionC} holds.

As to condition~(L) from Section~\ref{AppUniqueness}, we have to show that, given $K \geq \|\hat{b}\|_{\infty}$, there exists $L > 0$ (possibly depending on $d$, $\nu$, $T$, and $K$) such that for all $t\in [0,T]$, all $\phi\in \mathcal{X}$,
\begin{equation} \label{EqConditionL}
	\left| \hat{b}(t,\phi,\theta) - \hat{b}(t,\phi,\tilde{\theta}) \right| \leq L\cdot \mathrm{d}_{t}(\theta,\tilde{\theta}) \text{ whenever } \theta, \tilde{\theta} \in \mathcal{Q}_{\nu,K}.
\end{equation}
Thanks to Lemma~\ref{LemmaAppRegularity},
\[
	\inf_{t\in [0,T]} \inf_{\theta\in \mathcal{Q}_{\nu,K}} \theta\left( \tau > t \right) = \inf_{\theta\in \mathcal{Q}_{\nu,K}} \theta\left( \tau > T \right) \geq c_{0} > 0
\]
for some strictly positive constant $c_{0}$ depending only on $d$, $\nu$, $T$, and $K$. Denote by $\bar{L}$ the Lipschitz constant of $\bar{b}$ with respect to its third variable according to \Hypref{HypLipschitz}. Let $\theta, \tilde{\theta} \in \mathcal{Q}_{\nu,K}$. Then for all $\phi\in \mathcal{X}$,
\begin{align*}
	& \left| \hat{b}(t,\phi,\theta) - \hat{b}(t,\phi,\tilde{\theta}) \right| \\
	\begin{split}
	&= \left| \bar{b}\left(t, \phi(t), \frac{\int w(\tilde{\phi}(t)) \mathbf{1}_{[0,\tau(\tilde{\phi}))}(t) \theta(d\tilde{\phi})}{\theta(\tau > t)} \right) \right. \\
	&\quad \left. - \bar{b}\left(t, \phi(t), \frac{\int w(\tilde{\phi}(t)) \mathbf{1}_{[0,\tau(\tilde{\phi}))}(t) \tilde{\theta}(d\tilde{\phi})}{\tilde{\theta}(\tau > t)} \right) \right|
	\end{split}\\
	\begin{split}
	&\leq \frac{\bar{L}}{c_{0}} \left| \int w(\tilde{\phi}(t)) \mathbf{1}_{[0,\tau(\tilde{\phi}))}(t) \theta(d\tilde{\phi}) - \int w(\tilde{\phi}(t)) \mathbf{1}_{[0,\tau(\tilde{\phi}))}(t) \tilde{\theta}(d\tilde{\phi}) \right| \\
	&\quad + \frac{\bar{L}}{c_{0}^{2}} \cdot \|w\|_{\infty} \cdot \left|\theta(\tau > t) - \tilde{\theta}(\tau > t)\right|.
	\end{split}
\end{align*}
By definition and since $\{\tau > t\}\in \mathcal{G}_{t}$,
\[
	\left|\theta(\tau > t) - \tilde{\theta}(\tau > t)\right| \leq \mathrm{d}_{t}(\theta,\tilde{\theta}).
\]
The measures $\theta$, $\tilde{\theta}$ are both equivalent to $\Theta_{\nu}$, and their restrictions to $\mathcal{G}_{t}$ are equivalent to the restriction of $\Theta_{\nu}$ to $\mathcal{G}_{t}$. Denoting by $Z_{t}$ and $\tilde{Z}_{t}$, respectively, the densities with respect to the restriction of $\Theta_{\nu}$, we have
\[
	\mathrm{d}_{t}(\theta,\tilde{\theta}) = \frac{1}{2} \int_{\mathcal{X}} \left| Z_{t}(\tilde{\phi}) - \tilde{Z}_{t}(\tilde{\phi}) \right| \Theta_{\nu}(d\tilde{\phi}).
\]
It follows that
\begin{multline*}
	\left| \int w(\tilde{\phi}(t)) \mathbf{1}_{[0,\tau(\tilde{\phi}))}(t) \theta(d\tilde{\phi}) - \int w(\tilde{\phi}(t)) \mathbf{1}_{[0,\tau(\tilde{\phi}))}(t) \tilde{\theta}(d\tilde{\phi}) \right| \\
	\leq 2\|w\|_{\infty}\cdot \mathrm{d}_{t}(\theta,\tilde{\theta}).
\end{multline*}
Consequently, for all $t\in [0,T]$, all $\phi\in \mathcal{X}$, all $\theta, \tilde{\theta} \in \mathcal{Q}_{\nu,K}$,
\[
	\left| \hat{b}(t,\phi,\theta) - \hat{b}(t,\phi,\tilde{\theta}) \right|
	\leq \frac{\bar{L}}{c_{0}^{2}}\cdot \|w\|_{\infty} \left(2c_{0} + 1 \right)\cdot \mathrm{d}_{t}(\theta,\tilde{\theta}),
\]
which implies \eqref{EqConditionL}.


\section{Relaxed controls}\label{AppRelaxed}

In the proofs of Section~\ref{SectApproxNash}, Appendix~\ref{AppConvergence} and Appendix~\ref{AppRegularity}, we need the concept of relaxed controls. For a Polish space $\mathcal{S}$, let $\mathcal{R}_{\mathcal{S}}$ denote the space of all deterministic $\mathcal{S}$-valued relaxed controls over the time interval $[0,T]$, that is,
\[
	\mathcal{R}_{\mathcal{S}}\doteq \left\{ r : r\text{ positive measure on }\mathcal{B}(\mathcal{S}\times [0,T]): r(\mathcal{S}\times[0,t]) = t , \, t\in[0,T] \right\}.
\]	
If $r\in \mathcal{R}_{\mathcal{S}}$, then the time derivative of $r$ exists almost everywhere as a measurable mapping $\dot{r}_{t}\!: [0,T]\rightarrow \mathcal{P}(\mathcal{S})$ such that $r(dy,dt) = \dot{r}_{t}(dy)dt$. The topology of weak convergence of measures turns $\mathcal{R}_{\mathcal{S}}$ into a Polish space. Notice that $\mathcal{R}_{\mathcal{S}}$ is compact if $\mathcal{S}$ is compact. Any $\mathcal{S}$-valued $(\mathcal{F}_{t})$-adapted process $\alpha$ defined on some filtered probability space $(\Omega,\mathcal{F},\Prb)$ induces an $\mathcal{R}_{\mathcal{S}}$-valued random variable $\rho$, the corresponding stochastic relaxed control, according to 
\[
	\rho_{\omega}\bigl(B\times I\bigr)\doteq \int_{I}\delta_{\alpha(t,\omega)}(B)dt,\quad B\in \Borel{\Gamma},\; I\in \Borel{[0,T]},\;\omega \in \Omega.
\]
The random measure $\rho$ is $(\mathcal{F}_{t})$-adapted in the sense that its restriction to $\mathcal{S}\times [0,t]$ is $\mathcal{F}_{t}$-measurable for every $t\in [0,T]$ or, equivalently, that (a version of) the time derivative process $\dot{\rho}_{\cdot}$ is $(\mathcal{F}_{t})$-adapted. For details on relaxed controls, see, for instance, \citet{elkarouietalii87} or \citet{kushner90}.

\end{appendix}

\section*{Acknowledgements}

The authors thank two anonymous Referees for their careful reading of the manuscript and helpful comments and suggestions.


\bibliographystyle{abbrvnat}

\end{document}